\documentclass[a4paper]{amsart}
\usepackage[utf8]{inputenc}
\usepackage[T1]{fontenc}

\usepackage{amssymb}
\usepackage{amsfonts}
\usepackage{amsthm}
\usepackage{mathrsfs}
\usepackage{array}
\usepackage{bbm}
\usepackage{stmaryrd}
\usepackage[shortlabels]{enumitem}

\usepackage{slashed}

\usepackage{csquotes}

\usepackage{mathtools}
\mathtoolsset{showonlyrefs,showmanualtags}

\usepackage[svgnames]{xcolor}
\usepackage[unicode,plainpages=false]{hyperref}

\usepackage{enumitem}

\newcommand\R{\mathbb{R}}
\newcommand\N{\mathbb{N}}
\newcommand\Z{\mathbb{Z}}

\newcommand\vp{\dot{v}}

\newcommand\bnp{\begin{proof}}
\newcommand\enp{\end{proof}}

\newcommand\bneq{\begin{eqnarray*}\left\lbrace \begin{array}{rcl}}
\newcommand\eneq{\end{array} \right.\end{eqnarray*}}
\newcommand\bneqn{\begin{eqnarray}\left\lbrace \begin{array}{rcl}}
\newcommand\eneqn{\end{array} \right.\end{eqnarray}}

\newcommand\nor[2]{\left\|#1\right\|_{#2}}

\newcommand{\m}[1]{
\ifdefequal{#1}{1}
{\mathbbm{#1}}
{\mathbb{#1}}
}
\newcommand{\gh}[1]{\mathfrak{#1}}
\newcommand{\q}[1]{\mathcal{#1}}
\newcommand{\mc}[1]{\mathscr{#1}}

\newcommand{\bs}[1]{\boldsymbol{#1}}

\newcommand{\ds}{\displaystyle}
\newcommand{\be}{\begin{gather}}
\newcommand{\ee}{\end{gather}}
\newcommand{\ba}{\begin{align*}}
\newcommand{\ea}{\end{align*}}

\newcommand{\e}{\varepsilon}

\newcommand{\imp}{\Longrightarrow}

\newcommand{\loc}{\mathrm{loc}}

\newcommand{\opD}{\gh D}
\newcommand{\opR}{\gh R}

\DeclareMathOperator{\Span}{\text{Span}}

\DeclareMathOperator{\sgn}{\text{sgn}}

\DeclareMathOperator{\Id}{\text{Id}}

\renewcommand{\le}{\leqslant}
\renewcommand{\ge}{\geqslant}

\numberwithin{equation}{section}
\theoremstyle{plain}
\newtheorem{thm}{Theorem}[section]
\newtheorem*{thm*}{Theorem}

\newtheorem{prop}[thm]{Proposition}
\newtheorem{cor}[thm]{Corollary}
\newtheorem{lem}[thm]{Lemma}

\theoremstyle{definition}
\newtheorem{definition}[thm]{Definition}

\theoremstyle{remark}
\newtheorem{remark}[thm]{Remark}
\newtheorem{claim}[thm]{Claim}
\newtheorem{example}[thm]{Example}

\allowdisplaybreaks

\title{A scattering operator for some nonlinear elliptic equations}
\author{Raphaël Côte and Camille Laurent }


\begin{document}

\begin{abstract}
We consider non linear elliptic equations of the form
\[ \Delta u = f(u,\nabla u). \]
for suitable analytic nonlinearity $f$, in the vinicity of infinity in $\m R^d$, that is on the complement of a compact set.
We show that there is a \emph{one-to-one correspondence} between the non linear solution $u$ defined there, and the linear solution $u_L$ to the Laplace equation, such that, in an adequate space, $u - u_L\to 0$ as $|x|\to +\infty$. This is a kind of scattering operator.

Our results apply in particular for the energy critical and supercritical pure power elliptic equation and for the 2d (energy critical) harmonic maps and the $H$-system. Similar results are derived for solutions defined on the neighborhood of a point in $\m R^d$. 

The proofs are based on a conformal change of variables, and studied as an evolution equation (with the radial direction playing the role of time) in spaces with analytic regularity on spheres (the directions orthogonal to the radial direction).
\end{abstract}

\maketitle

\tableofcontents

\section{Introduction}

\subsection{Motivations and setting of the problem}

The purpose of this article is to give a classification of solutions of certain nonlinear elliptic equations, by their behavior at infinity. We consider equations of the form
\begin{equation} \label{eq:ell} 
\Delta u = f(u, \nabla u),
\end{equation}
where the nonlinearity $f$ is analytic, and with an extra emphasis on the elliptic nonlinear equation with $\dot H^1$ critical power nonlinearity, conformal equations in dimension $2$, and smooth harmonic maps. Roughly speaking, we will construct, in these considered examples, a scattering operator: we prove that when considering the vicinity of (spatial) infinity, there is a one-to-one correspondence between linear solutions of
\begin{equation} \label{eq:ell_lin_intro}
\Delta u_L = 0 \quad \text{on} \quad  \m R^d \setminus B(0,1), \quad u_L|_{\m S^{d-1}} = u_0,
\end{equation}
where $u_0$ is a given function on $\m S^{d-1}$ and nonlinear solutions of \eqref{eq:ell} defined for sufficiently large $x$; and furthermore nonlinear solutions behave as a linear one in an appropriate space. 

This space is strong enough to distinguish each linear solution from another one only from their asymptotic behavior. For instance, since all linear solutions converge to $0$ at infinity, the space we consider should be much finer than $L^{\infty}$ or $\dot{H}^1$. The space we use specifically translates the behaviour of the linear elliptic solutions. In particular, it implies some analyticity in the angular variable. 

Note also that the full classification is obtained in some general examples that are critical or with additional assumptions. Yet, the construction of nonlinear solutions from their behavior at infinity (that is one part of the scattering operator) is made in a great generality, see for instance Theorem \ref{thmexistglobP}.

\bigskip

The problem we consider is natural and has its own interest; we believe it will also prove useful for related evolution problems.
Indeed, one extra motivation comes from the evidence that the asymptotic behavior of non linear object like the well-known soliton plays a fundamental role in dynamical contexts, as it drives the interactions: for example, the construction of blow up solutions, the construction of multi-solitons, the analysis of collision of solitons, the soliton resolution conjecture etc. 

Let us elaborate somehow on this last example, in the case of the energy-critical wave equation, which was studied by Duyckaerts-Kenig-Merle. They develop in particular the channel of energy method: in the 3D radial setting \cite{DuyKMClassif}, the authors manage to conclude that some initial datum giving rise to ``nonlinear non radiative solutions'' should behave at infinity as the Newtonian potential $\frac{1}{r}$ and next, should actually \emph{be} the ground state $W: x \mapsto (1+|x|^2/3)^{-1/2}$ (up to scaling). This idea to ``catch the ground-state by the tail'' has been extended with many more subtleties to other dimensions and other equations \cite{CKLS:15I,CKLS:15II,C:15,DKM:23}. 

One of the key roadblocks in generalizing the above results to the non radial setting is the lack of understanding of the non radial nonlinear objects, such as spectral properties when linearizing around them, or their asymptotic behavior.

Our work here provides a first description in the non radial context, within a framework that encompass semilinear elliptic equations together with harmonic maps or the $H$-system. 

\bigskip

We address the question by recasting the elliptic equation in terms of an evolution equation on the sphere, where time is played by the radial variable. After performing a conformal change of variable, the equation is obviously (strongly) ill posed, but is amenable to resolution from infinity, for data \emph{without growing modes}.

We will now define the functional setting in the next paragraph, so as to state our results with the following ones.

\subsection{The functional setting}

The proof will be performed after a conformal change of coordinates from $\R^d$ to $\R\times \m S^{d-1}$ using spherical coordinates. The harmonical analysis on $\m S^{d-1}$ will play a crucial role. We begin by a few generalities.

Let $d \ge 2$. We denote $\Delta_{\mathbb S^{d-1}}$ the Laplace-Beltrami operator on the sphere $\m S^{d-1}$, and let $(\phi_{\ell,m})_{\ell \in \m N, m \le N_\ell}$ be an $L^2$ orthogonal basis of normalized spherical harmonics, so that $\phi_{\ell,m}$ is the restriction of a harmonic homogeneous polynomial of degree $\ell\in \N$. Recall that $\phi_{\ell,m}$ are eigenfunctions for $-\Delta_{\mathbb{S}^{d-1}}$ 
\[ -\Delta_{\mathbb{S}^{d-1}} \phi_{\ell,m} = \ell (\ell+d-2)  \phi_{\ell,m}, \]
and so $N_\ell$ is the dimension of the eigenspace of $-\Delta_{\mathbb{S}^{d-1}}$ for the eigenvalue $\ell (\ell+d-2)$. Let $P_\ell$ be the orthogonal projection onto this eigenspace: if $v$ is a function defined on $\m S^{d-1}$,
\begin{align} \label{def:Pell}
P_\ell v = \sum_{m=1}^{N_\ell}  \langle f, \phi_{\ell,m} \rangle \phi_{\ell,m}, \quad \| P_\ell v \|_{L^2(\m S^{d-1})}^2 = \sum_{m=1}^{N_\ell}  |\langle v, \phi_{\ell,m} \rangle|^2.
\end{align}
We consider the positive elliptic operator on the sphere $\m S^{d-1}$
\[ \opD = \sqrt{-\Delta_{\mathbb{S}^{d-1}}+\left(\frac{d-2}{2}\right)^2}, \]
so that for all $\ell \in \m N$ and $m \le N_\ell$,
\[ \opD  \phi_{\ell,m} = \left( \ell + \frac{d-2}{2} \right) \phi_{\ell,m}. \]
We denote $L^2_0(\m S^{d-1}) = \Span(\phi_{\ell,m}, \ell \in \m N, m \le N_\ell \}$ the space of (finite) linear combinations of eigenfunctions of $\gh D$; all normed spaces below will be meant as completion of $L^2_0$ for the underlying norm.

The space $H^s(\m S^{d-1})$ is the completion of $L^2_0(\m S^{d-1})$ for the $H^s$ norm defined as
\begin{align} \label{defSob}\| v \|_{H^s(\m S^{d-1})}^2 : = \sum_{\ell =0}^{+\infty} \langle \ell \rangle^{2s} \| P_\ell v \|_{L^2(\m S^{d-1})}^2,
\end{align}
where $\langle \ell \rangle = \sqrt{1+|\ell|^2}$ is the japanese bracket (here $\ell$ is an integer, but we also use the notation for a vector or a multi-index). There hold
\[ \| v \|_{H^s(\m S^{d-1})}\approx \| (1 -\Delta_{\mathbb{S}^{d-1}})^{s/2} v \|_{L^2(\m S^{d-1})}, \]
and both these norms are equivalent to the usual Sobolev norm considered (thanks to some partition of unity) in each coordinate charts,  \cite[Section 7.3]{LionsMagenes1}. Therefore, all usual Sobolev embeddings apply.

\bigskip

We will state our results using the spaces $Z_{s,r}^\infty$ and $Z_{s,r}^0$, made of functions on $\m S^{d-1}$, and which are the completion of $L^2_0(\m S^{d-1})$ for the respective norms:
\begin{align*}
     \nor{v}{Z_{s,r}^{\infty}} & :=r^{\frac{d-2}{2}}\nor{r^{\opD} v}{H^{s}(\mathbb{S}^{d-1})} = \left( \sum_{\ell=0}^{+\infty} \langle \ell \rangle^{2s} r^{2(\ell + d-2)} \| P_\ell v \|_{L^2(\m S^{d-1})}^2 \right)^{1/2}, \\
     \text{and} \quad \nor{v}{Z_{s,r}^{0}} & :=r^{\frac{d-2}{2}}\nor{r^{-\opD} v}{H^{s}(\mathbb{S}^{d-1})} = \left( \sum_{\ell=0}^{+\infty} \langle \ell \rangle^{2s} r^{-2\ell} \| P_\ell v \|_{L^2(\m S^{d-1})}^2 \right)^{1/2}.
\end{align*}

Notice that for $1\le r<r'$, we have the continuous embedding $Z^{\infty}_{s,r'}\subset Z^{\infty}_{s,r}\subset H^s$  together with $\nor{v}{Z_{s,r}^{\infty}}\le \nor{v}{Z_{s,r'}^{\infty}}$ and similarly for $0< r'<r\le 1$, there hold $Z^{0}_{s,r'}\subset Z^{0}_{s,r}\subset H^s$ and $\nor{v}{Z_{s,r}^{0}}\le \nor{v}{Z_{s,r'}^{0}}$.

\bigskip

For functions defined on $\R^{d}\setminus B(0,r_{0})$ or $B(0,r_{0})$ respectively, we will be interested in the $Z_{s,r}$ regularity on rescaled restrictions $u(r \cdot) :  (y \mapsto u(ry)$ (defined on the sphare $ \m S^{d-1}$) (for $r \le r_0$ or $r \ge r_0$ respectively: this dependence of the space on the radius prompt us to the following definitions, which are meaningful due to the continuous embedding mentioned above.

We say that $u \in \q Z_{s,r_{0}}^{\infty}$ if $u$ defined on  $\R^{d}\setminus B(0,r_{0})$ and such that, for all $r \ge r_0$, $u(r \cdot) \in Z_{s,r/r_0}^{\infty}$, $(\partial_r u)(r \cdot) \in Z_{s-1,r/r_0}^{\infty}$,
\[ (\rho\mapsto u(\rho \cdot)) \in \mc C([r,+\infty),Z_{s,r/r_0}^{\infty})\cap \mc C^1([r,+\infty),Z_{s-1,r/r_0}^{\infty}), \]
and such that the following norm is finite:
\begin{align*}
\nor{u}{\q Z_{s,r_{0}}^{\infty}} := r_{0}^{\frac{d-2}{2}}  {\sup_{r\ge r_{0}} }  \left(  \nor{u(r\cdot)}{Z_{s,r/r_{0}}^{\infty}}+ \nor{\left( \frac{d-2}{2} u +r\frac{\partial u}{\partial r} \right)(r\cdot)}{Z_{s-1,r/r_{0}}^{\infty}}\right).
\end{align*}
Similarly, $u \in \q Z_{s,r_{0}}^{0}$ if it is defined on $\overline{B(0,r_0)}$ so that for all $0 < r \le r_0$, $u(r \cdot) \in Z^0_{s,r/r_0}$, $(\partial_r u)(r \cdot) \in Z_{s-1,r/r_0}^{0}$,
\[ (\rho \mapsto u(\rho \cdot)) \in \mc C((0,r],Z_{s,r/r_0}^{0})\cap \mc C^1((0,r],Z_{s-1,r/r_0}^{0}), \]
and such that the related norm is finite:
\begin{align*}
\nor{u}{\q Z_{s,r_{0}}^{0}} :=r_{0}^{\frac{d-2}{2}}  \sup_{0<r\le r_{0}}  \left(  \nor{u(r\cdot)}{Z_{s,r/r_{0}}^{0}}+ \nor{\left( \frac{d-2}{2} u +r\frac{\partial u}{\partial r} \right)(r\cdot)}{Z_{s-1,r/r_{0}}^{0}}\right).
\end{align*}
The exponent $\infty$ reminds that we will be interested in the behavior for $|x| \to +\infty$, while the exponent $0$ denotes a space adapted to the vicinity of $0$ (or a point).

The regularity index $s$ appears as a fine tuning parameter: it plays an important role in the product laws, and so in the multi-linear estimates; one could simply fix for the rest of the article 
\begin{equation}   
s > \frac{d}{2}+\frac{3}{2}.
\end{equation}
For some purposes, we stated some intermediary results with more precision on $s$: except if a specific weaker bound is precised, all  the following results assume $s$ as above.

$u\in\q Z_{s,r_{0}}^{\infty}$ implies that $u$ has the same decay as a linear solution of $\Delta u_L=0$, that is $|u(x)|=\mathcal{O}(|x|^{-(d-2)})$ (see Lemma \ref{lminjectZ0}). It is actually more precise: when decomposing in spherical harmonics, each component decays as a linear solution, that is $\nor{P_\ell u(r\cdot)}{L^{\infty}(\m S^{d-1})}=\mathcal{O}_{\ell}(r^{-(d-2)-\ell})$ as $r \to +\infty$.

Note that in several results that we will prove, the functions we consider are vector valued. In order to keep reasonable  notations which are already heavy, we always mean that every coordinate belongs to the related spaces: for instance, we will write $\q Z^\infty_s$ instead of $(\q Z^\infty_s)^N$. This should not lead to any confusion.

Finally, we drop the index $r$ when $r=1$, for example $\q Z^\infty_s : =  \q Z^\infty_{s,1}$.

\bigskip

Now, we begin to present our main results for specific types of equations. The constructions of the nonlinear solutions with prescribed behavior at infinity are always consequence of the two general Theorem \ref{thmexistglobP} and \ref{thmexistglobPgain} presented in Section \ref{subsec:conf}. Theorem \ref{thmexistglobPgain} is a refinement of Theorem \ref{thmexistglobP}, needed in some critical cases where we have to use some better properties of the first iterate of the Picard iteration. This improved behavior is observed in the case of conformal equations in dimension $2$ where a ``null'' structure is observed.

\subsection{Main results on the semilinear equation}

In this section we state those of our results which are concerned with the case of the (non derivative) nonlinearity
\begin{equation}
    \label{eq:f_semilinear}
    {f}(y)=\sum_{p\in \N} a_p y^p,
\end{equation}
with a positive radius of convergence.

The first statement constructs, for a general nonlinearity, nonlinear solutions having a prescribed linear behavior at infinity. The second one, restricted to $H^1$-critical analytic nonlinearities, realizes the converse, that is, establishes that a finite energy solutions behaves as a linear solution at infinity. The combination of both results leads to a kind of scattering operator identifying linear and nonlinear germs of solutions.

\begin{thm}\label{thmexistglobPintro}
Let $d\ge 3$. Assume that $f$ as in \eqref{eq:f_semilinear} satisfies the supercriticality assumption
\begin{align}
\label{hypbpqinftyintro}
 a_{p} \ne 0 \Longrightarrow  (d-2) p-d \ge \nu_0>0. 
 \end{align}
Let $u_{0} \in H^s(\m S^{d-1})$, and  $u_{L} \in \q Z^{\infty}_{s}$ given in \eqref{eq:ell_lin_intro}.

Then, there exist $r_{0}\ge 1$ and a unique small solution $u \in \q Z^{\infty}_{s,r_{0}}$ of \eqref{eq:ell} on  $\{|x|\ge r_{0}\}$ and such that 
\begin{align}
\label{scattranslrdintro}
\nor{(u-u_{L})(r\cdot)}{Z^{\infty}_{s,r/r_{0}}}\lesssim r^{-\nu_{0}}\underset{r\to +\infty}{\longrightarrow}0.
\end{align}
Moreover, the map $u_{0}\mapsto u$ is injective; and if $\nor{u_{0} }{H^s(\m S^{d-1})}$ is small enough, we can take $r_{0}=1$.
\end{thm}

(Here and below, we say that there is a unique small solution $u$ in a Banach space $Z$ if there exists $\e >0$  such that $u$ is the unique solution in the ball centered at 0 and of radius $\e$ of $Z$.)

When we restrict to $\dot H^{1}$ critical exponents with analytic nonlinearities (which actually leaves the three possibilities mentioned below), we obtain a full classification of all possible solutions close to infinity.
\begin{thm}[Semilinear energy critical equation]
\label{thmH1reg}
We consider the equation
\begin{equation}
\label{eq:elliptic_nonlin}
\Delta u=\kappa u^{p}.
\end{equation} 
where $\kappa \in \m R$, and we assume to be in one of the following situations:
\[ (d,p) \in \{ (3,5), (4,3), (6,2) \}. \]

1) Let $u\in \dot{H}^{1}(\{|x|\ge 1\})$ be a solution of \eqref{eq:elliptic_nonlin} in the weak sense (see Definition \ref{defsoluext}). Then, there exist $r_{0}\ge 1$ so that $u\in \q Z_{s,r_{0}}^{\infty}$ and a unique $u_{L}\in \q Z_{s,r_{0}}^{\infty}$ solution of $\Delta u_{L}=0$ on $\{|x|\ge r_{0}\}$ so that  
\begin{align}
\label{scattranslrdinverseH1regpuiss}
\nor{u(r\cdot)-u_{L}(r\cdot)}{Z_{s,r/r_{0}}^{\infty}}\le Cr^{-2}\underset{r\to +\infty}{\longrightarrow}0.
\end{align}
2) Reciprocally, given  $u_{0} \in H^s(\m S^{d-1})$, and  $u_{L} \in \q Z^{\infty}_{s}$ as in \eqref{eq:ell_lin_intro}, there exists $r_{0}\ge 1$ and a unique small solution $u\in \q Z_{s,r_{0}}^{\infty}$ of \eqref{eq:elliptic_nonlin} on $\{|x|\ge r_{0}\}$ satisfying \eqref{scattranslrdinverseH1regpuiss}.
\end{thm}

To our knowledge, such classification did not appear elsewhere in the literature, for any elliptic type equation. It gives both a complete rigidity and a fine description for nonlinear solutions, concerning their behavior at infinity.

In particular, the previous theorem also implies a result of unique continuation at infinity.

\begin{cor}
\label{corUCPup}
In the situation of the previous Theorem \ref{thmH1reg} 1), if $u\in \dot{H}^{1}(\{|x|\ge 1\})$ is a solution of \eqref{eq:elliptic_nonlin} so that 
\[ \forall \ell \in \m N, \quad  r^{d-2+\ell}\nor{P_{\ell}u(r\cdot)}{H^{s}(\mathbb{S}^{d-1})} \to 0 \quad \text{as} \quad r\to +\infty, \]
then $u=0$ on $\m R^d \setminus B(0,1)$. In particular, if $ u(x)=\mathcal{O}(|x|^{-\beta})$ for any $\beta\in \R$, then $u=0$.
\end{cor}
The results in this direction we are aware of (see for instance \cite{BK:05,B:14}) would be obtained considering $u^p$ as $Vu$ for some potential $V=u^{p-1}$. They require exponential decay without distinction between the spherical harmonics.

For power nonlinearities $u^p$, $p>\frac{d}{d-2}$ ($p$ integer), Theorem \ref{thmexistglobPintro} constructs a lot of solutions with prescribed asymptotic linear behavior. 
We can perform a classification under further decay assumption.
\begin{thm}[Semilinear equation with decay]
\label{thmregdecay}
Let $d\ge 3$, $\kappa\in \R$ and $p\in \N^*$ with $p>\frac{d}{d-2}$ and consider the equation 
\begin{equation}
\label{eq:elliptic_nonlingen}
\Delta u=\kappa u^{p}.
\end{equation}

1) Let $u\in \dot{H}^{1}(\{|x|\ge 1\})$ be a solution of \eqref{eq:elliptic_nonlingen} in the weak sense (see Definition \ref{defsoluext}) so that  for some $\eta>0$ and $C>0$, we have 
\[ \forall |x| >1, \quad |u(x)|\le C |x|^{-\frac{2}{p-1}-\eta}. \]
Then, there exists $r_{0}\ge 1$ so that $u\in \q Z_{s,r_{0}}^{\infty}$ and there exists a unique $u_{L}\in \q Z_{s,r_{0}}^{\infty}$ solution of $\Delta u_{L}=0$ on $\{|x|\ge r_{0}\}$ so that  
\begin{align}
\label{scattranslrdinverseH1regpuissgene}
\nor{u(r\cdot)- u_{L}(r\cdot)}{Z_{s,r/r_{0}}^{\infty}}\le Cr^{-((d-2) p-d)}\underset{r\to +\infty}{\longrightarrow}0.
\end{align}
2) Reciprocally, given  $u_{0} \in H^s(\m S^{d-1})$, and  $u_{L} \in \q Z^{\infty}_{s}$ as in \eqref{eq:ell_lin_intro}, there exists $r_{0}\ge 1$ and a unique small solution $u\in \q Z_{s,r_{0}}^{\infty}$ of \eqref{eq:elliptic_nonlingen} on $\{|x|\ge r_{0}\}$ satisfying \eqref{scattranslrdinverseH1regpuissgene}.
\end{thm}
In the defocusing case, the decay can be obtained using results of V\'eron \cite{V:81} for solutions constructed by Benilan-Br\'ezis-Crandall in \cite{BBC:75}. 
\begin{cor}
\label{corfocsemi}
    Let $d\ge 3$, $p\in 2\N+1$ with $p>\frac{d}{d-2}$. Let $f \in L^1(\R^d)$ be real valued with compact support. Due to \cite{BBC:75}, there exist a unique real valued solution $u \in L^{\frac{d}{d-2}, \infty}(\m R^d)$ with $\Delta u\in L^1(\R^d)$ of
    \[ \Delta u=u^p+f. \]
    Then the conclusion of Theorem \ref{thmregdecay}~1) holds for $u$.
\end{cor}
Here $L^{q,\infty}(\R^d)$ are the weak-$L^{q}$ spaces, for $1<q<+\infty$, and are called spaces of Marcinkiewicz $M^{q}(\R^d)$ in the above reference, see \cite[Appendix]{BBC:75}  or \cite[Chap V.3]{StWeiBookFourierEucl}.

In particular, in the defocusing case, the assumption of additional decay in Theorem  \ref{thmregdecay} is not necessary and can be obtained under reasonable assumptions on the solution. Yet, this assumption is sometimes necessary with the power $\frac{2}{p-1}$ being optimal. For instance, for a $\dot{H}^1$-supercritical nonlinearity $p>\frac{d+2}{d-2}$, in the focusing case $\kappa<0 $, it is known  that there exist radial positive solutions that behave like $C|x|^{-\frac{2}{p-1}}$ at infinity. These solutions have a slower decay than the solutions we construct in $\q Z_{s}^{\infty}$ which decay as the linear solutions, that is $C|x|^{-(d-2)}$. We refer to \cite[Theorem 5.2]{KS:07} for a nice summary. We refer also to \cite[Theorem 3.3]{VV:91} for a dichotomy result in the case of positive solutions in the defocusing case and $p\neq \frac{d+2}{d-2}$, $p>1$.

The class of equations covered by our theorems for constructing solutions are quite general, either for prescribed behavior at infinity (see Theorem \ref{thmexistglobP} below) or prescribed Dirichlet value (see Theorem \ref{thmexistglobPDir} below). Yet, the regularity results and uniqueness of the Dirichlet boundary value problem has to be adapted to each equation. This is the reason why we only treated some examples for the classification; we nonetheless believe that the strategy can be applied in many more cases. Assuming that we are able to construct solution with prescribed behavior at infinity, the road map for the classification in Theorem \ref{thmH1reg} and \ref{thmregdecay} goes as follows:
\begin{itemize}
\item prove by scaling and regularity arguments that, for a rescaled version of the solution, the trace on $\m S^{d-1}$ is small in $H^s(\m S^{d-1})$ with $s$ large enough.
\item construct a solution in the space $\q Z_{s}^{\infty}$ with the same Dirichlet data on $\m S^{d-1}$. By construction, this solution has the correct decay and will ``scatter'' to a linear solution.
\item prove a uniqueness result for the Dirichlet value problem in some appropriate space containing the original solution and the solution we constructed.
\end{itemize}
The full classification as in Theorem \ref{thmH1reg} is not always true, but we believe that some modifications of the methods we introduce in this paper might lead to similar results. It would be natural to try to construct, by a modification of the space $\q Z_{s}^{\infty}$, other sets of nonlinear solutions with different asymptotic behavior.

\subsection{Main results on conformal equations in dimension \texorpdfstring{$2$}{2}}

 Let $(\mathcal{N}, h)$ be an analytic compact Riemannian manifold of dimension $N$. Without loss of generality, we assume that $\mathcal{N}$ is actually embedded in $\m R^M$ (for some large integer $M$) analytically and isometrically, see \cite{NashAnalytic, GJ:71, J:72}.
For $\Omega\subset \R^d$ open subset ($d=2$ in this section, but some definitions will also be used for any $d\ge 2$), we will consider maps in the space
\begin{align}
 \label{defH1MN}
     H^{1}_{\loc}(\Omega,\mathcal{N}) :=\left\{  u \in H^1_{\loc}(\Omega, \R^{M}) :  u(x)\in \mathcal{N} \textnormal{ for a.e. } x \in \Omega\right\}.
\end{align}
We define similarly the spaces $\mc C^r(\Omega,\mathcal{N})$ for $r\ge 0$.
We will say that $u = (u_1, \dots, u_M)$ is of \emph{finite energy} on $\Omega$ if $\nabla u$, defined in the distributional sense on $\Omega$, is in $L^2(\Omega)$, i.e., the following quantity $\mathcal{E}(u)$ is finite:
 \begin{align}
 \label{defNRJHM}
\mathcal{E}(u) :=\int_{\Omega}|\nabla u|^2dx <+\infty, \quad \text{where} \quad |\nabla u|^2=\sum_{i=1}^M \sum_{\alpha=1}^d |\partial_{\alpha} u_i|^2.
 \end{align}
We refer to the lecture notes \cite{R:cours} for a survey on the subject and appropriate references. Let $\omega$ be an analytic 2-form on $\mathcal{N}$. We denote $\widetilde{\omega}={\pi_{\mathcal{N}}}^{*}\omega$ the pullback of $\omega$ by $\pi_{\mathcal{N}}$, the orthogonal projection on $\mathcal{N}$, defined in a small tubular neighborhood of $\mathcal{N}$. 
For $u\in \mc C^2(\Omega,\mathcal{N})$, we are studying solutions of
\begin{align}
\label{systconformembed}
\tag{Conf-E}
\Delta u=-A(u)(\nabla u,\nabla u)-H(u)(\partial_x u,\partial_y u)
\end{align}
where $A$ is the second fundamental form\footnote{we denote $A(u)(\nabla u,\nabla u)=\sum_{i=1}^d A(u)(\partial_{x_i} u, \partial_{x_i} u)$} of the embedding $\q N \subset \m R^M$ and for $z \in \q N$, $H(z)$ is the $T_{z}\mathcal{N}$-valued alternating 2-form on $T_{z}\mathcal{N}$ defined by
\begin{align}
\label{formH}
\forall U,V,W \in T_{z}\mathcal{N}, \quad d\omega_{z}(U,V,W) = U\cdot H(z)(V,W)
\end{align}
Let  $(e_{i})_{i=1, \dots,M}$ be the canonical basis of $\R^{M}$. Denote for $y\in\mathcal{N}$, $H_{jk}^{i}(y)=d\widetilde{\omega}_{y}(e_{i},e_{j},e_{k})$. Note that we have $H_{jk}^{i}=-H_{ik}^{j}$.
The previous formulations is quite general and contains the following particular cases:
\begin{itemize}
\item Harmonic maps: $\Delta u=-A(u)(\nabla u,\nabla u)$
\item for $d=2$ and $\mathcal{N}=\R^3$ (or $\mathbb{T}^3$), the $H$-system (surfaces with prescribed mean curvature) : $\Delta u= H(u) u_x\land u_y$.
\end{itemize}

It was proved by Rivi\`ere \cite{R:07} that in dimension $d=2$,  weak solutions are actually smooth (see also H\'elein \cite{Helein:91} for the case $H=0$, that is, harmonic maps), so we won't distinguish between weak and smooth solutions in \eqref{systconformembed} in this case. Our main result on the system \eqref{formconf} is the following.

\begin{thm}
\label{thmH1regHarmon}
1) Let $u\in H^{1}_{\loc}(\R^{2}\setminus B(0,1),\mathcal{N})$ be a finite energy solution of \eqref{systconformembed}. Then, there exists $r_{0}\ge 1$ so that $u\in \q Z_{s,r_{0}}^{\infty}$. Moreover, there exists one unique $u_{\infty}\in \mathcal{N}$ and $u_{L}\in \q Z_{s,r_{0}}^{\infty}$ solution of $\Delta u_{L}=0$ on $\{|x|\ge r_{0}\}$ and with value in $T_{u_{\infty}}\mathcal{N}$ so that  
\begin{align}
\label{scattranslrdinverseHarmonMap}
\nor{\pi_{T_{u_{\infty}}\mathcal{N}}(u(r\cdot) - u_\infty)-u_{L}(r\cdot)}{Z_{s,r/r_{0}}^{\infty}}\le Cr^{-2}\underset{r\to +\infty}{\longrightarrow}0,
\end{align}
where $\pi_{T_{u_{\infty}}\mathcal{N}}$ is the orthogonal projection on $T_{u_{\infty}}\mathcal{N}$ (and $P_0 u_L =0$).

2) Reciprocally, for any $u_{\infty}\in \mathcal{N}$ and $u_{L} \in \q Z_{s}^{\infty}$  with value in $T_{u_{\infty}}\mathcal{N}$ and solution of $\Delta u_{L}=0$ on $\R^{2}\setminus B(0,1)$ with $P_{0}u_{L}=0$, there exists $r_{0}\ge 1$ and a unique small solution $u\in  \q Z_{s,r_{0}}^{\infty}\cap \mc C^{\infty}(\R^2\setminus B(0,r_0),N)$ solution of \eqref{systconformembed} on $\{|x|\ge r_{0}\}$ satisfying \eqref{scattranslrdinverseHarmonMap}.
\end{thm}
It turns out that conformal equations as \eqref{systconformembed} are critical with respect to the limit exponents in our general Theorem \ref{thmexistglobP}. So, we need a refined version, namely Theorem \ref{thmexistglobPgain}, that uses a better behavior of the first iteration of the Picard term. This improved decay is proved as a consequence of a general "null" condition of the form 
\begin{equation} \label{hypoNullintro}
\forall \xi \in \m C^2, \quad \left(\xi^2 =0 \imp f(\cdot,\xi) =0 \right).
\end{equation}
We have not seen this condition elsewhere. Since $\xi^2$ is the symbol of $\Delta$, this seems like an elliptic version of the celebrated null condition of Klainerman \cite{K:86} in the context of hyperbolic equations. Note that the null condition has to be checked for complex frequencies while the usual null condition for hyperbolic equations is checked for real $\xi$. We refer to Section \ref{s:null2} for more precisions and equivalent formulations in dimension $2$.



\subsection{Main results on Harmonic maps in dimension \texorpdfstring{$d\ge 3$}{d >=3}}
For $u\in \mc C^2(\Omega,\mathcal{N})$, we will say that $u$ is solution of the harmonic map equation if it satisfies 
\begin{align}
\label{eqnHarmonicMaps}
\tag{HM-E}
\Delta u=-A(u) (\nabla u,\nabla u)
\end{align}
where $A$ is the second fundamental form of the embedding of $\mathcal{N}$ in $\R^{M}$. This is exactly the previous equation with $H=0$. For $u\in H^{1}_{\loc}(\Omega,\mathcal{N})$, the equation \eqref{eqnHarmonicMaps} makes sense in the the distributional sense and we will say that it is a weak solution of the harmonic map when it is the case (see Definition \ref{defsoluextHM} for a more precise statement).

\begin{thm}
\label{thmH1Harmondimd}
Let $d\ge 3$. 

1) Let $u\in \mc C^{2}(\R^{d}\setminus B(0,1);\mathcal{N})$ be a finite energy solution of \eqref{eqnHarmonicMaps}. Then, there exists $u_{\infty}\in \mathcal{N}$ and $r_{0}\ge 1$ so that $u-u_{\infty}\in \q Z_{s,r_{0}}^{\infty}$. Moreover, there exists one unique $u_{L}\in \q Z_{s,r_{0}}^{\infty}$ solution of $\Delta u_{L}=0$ on $\{|x|\ge r_{0}\}$ and with value in $T_{u_{\infty}}\mathcal{N}$ so that  
\begin{align}
\label{scattranslrdinverseHMgeq3}
\nor{\pi_{T_{u_{\infty}}\mathcal{N}}(u(r\cdot)-u_{\infty})-u_{L}(r\cdot)}{Z_{s,r/r_{0}}^{\infty}}&\le Cr^{-2(d-2)}\underset{r\to +\infty}{\longrightarrow}0\\
\nor{\pi^{\perp}_{T_{u_{\infty}}\mathcal{N}}(u(r\cdot)-u_{\infty})}{Z_{s,r/r_{0}}^{\infty}}&\le Cr^{-\frac{d-2}{2}}\underset{r\to +\infty}{\longrightarrow}0
\end{align}
where $\pi_{T_{u_{\infty}}\mathcal{N}}$ and $\pi^{\perp}_{T_{u_{\infty}}\mathcal{N}}$ are the orthogonal projections on $T_{u_{\infty}}\mathcal{N}$ and $T_{u_{\infty}}\mathcal{N}^{\perp}$, respectively.

2) Reciprocally, for any $u_{\infty}\in \mathcal{N}$ and $u_{L} \in \q Z_{s}^{\infty}$ solution of $\Delta u_{L}=0$ on $\R^{d}\setminus B(0,1)$ and with value in $T_{u_{\infty}}\mathcal{N}$, there exists $r_{0}\ge 1$ and a unique small\,\footnote{See Theorem \ref{thmexistglobPgain} for a precise condition.} solution $u\in  \q Z_{s,r_{0}}^{\infty}$ of \eqref{eqnHarmonicMaps} on $\{|x|\ge r_{0}\}$ satisfying \eqref{scattranslrdinverseHMgeq3}.

Additionally, we denote $u_{L,1}$ the first iterate of the Duhamel formula, that is the only solution of \[ \Delta u_{L,1}=\Gamma(u_L)(\nabla u_L,\nabla u_L) \quad \text{so that} \quad  \nor{ u_{L,1}(r\cdot)}{Z^{\infty}_{s,r/r_0}}\underset{r\to +\infty}{\longrightarrow}0, \]
where $\Gamma$ are the Cristoffel symbols in coordinates given by $\pi_{T_{u_\infty} \q N}$. Then, we have the improved decay
\begin{align}
\label{addidecaythmHM}
\nor{\pi_{T_{u_{\infty}}\mathcal{N}}(u(r\cdot) - u_\infty)-(u_{L}(r\cdot)+ u_{L,1} (r\cdot))}{Z_{s,r/r_{0}}^{\infty}}&\le Cr^{-4(d-2)}.
\end{align}
\end{thm}
\begin{remark}
The regularity $\mc C^{2}$ is not optimal, but some assumption is necessary to avoid singular solutions that do not enter in our framework, as the ones constructed in \cite{R:95}. More precisely, in the proof, we needed enough regularity to apply Theorem \ref{thmSchoen}.

Yet, it could be replaced by other types of assumption implying some smoothness. We refer to the book \cite{LW:08} on the available regularity results.

For instance, it is proved in \cite{BG:80} (in the case $\mathcal{N}=\m S^{N}$) that $\mc C^{0}$ solutions are actually analytic.  Also, the theory of Schoen-Uhlenbeck \cite{SU:83} proves that small energy minimizing harmonic maps are smooth, which happens in our context for $R_{0}$ large enough.

\end{remark}
\begin{remark}
The convergence of the orthogonal component seems very bad with respect to the tangential part. Yet, since the manifold $\mathcal{N}$ can be locally written as a graph of the tangential part, the orthogonal component is completely computable (without referring to the PDE) once the tangential expansion is performed. So, with a Taylor expansion of the graph locally defined by $\mathcal{N}$, it might be possible to obtain the same precision as the formula for the tangential part.
\end{remark}

\begin{remark}
The analysis in \cite{ABLV22} computes an expansion of the (locally energy minimizing) solution of the Harmonic maps in dimension $3$ with target $\m S^{2}$ at the order $r^{-4}$. 
Theorem~\ref{thmH1Harmondimd} allows to obtain a similar expansion, see Remark~\ref{rkLamy} for further details.
\end{remark}


\subsection{Main results on semilinear equations close to a point}

We also obtain some result close to $0$. 
\begin{thm}\label{thmexistzero}Let $f: \R \to \R$ be an analytic function with positive radius of convergence and such that $f(0)=0$.

1) For any smooth solution $u$ of $\Delta u=f(u) $ on $B(0,1)$, there exist a solution $u_{L}$ of $\Delta u_{L}=0$ and $g$ analytic on $B(0,r_{0})$ for some $0\le r_{0}<1$ so that $u$ can be written
\begin{align}
\label{Fischerthm}
u=u_{L}+|x|^{2}g.
\end{align}

2) Reciprocally, for any $u_{L}$ bounded solution of $\Delta u_{L}=0$ on $B(0,1)$ with $u_L(0)=0$, there exist $0<r_{0}\le 1$  and a unique small analytic solution $u$ of $\Delta u=f(u) $ on $B(0,r_{0})$ so that \eqref{Fischerthm} holds for one $g$ analytic on $B(0,r_{0})$.

Moreover, the application $u_{L}\mapsto u$ is injective. 
\end{thm}

The decomposition \eqref{Fischerthm} is known in the literature as the Fischer decomposition of the function $u$ (see Section \ref{subsectFischer}). This decomposition is already known to hold for any analytic function, so the first part is not really new. The main part of our proof is the construction of the nonlinear solution. This result can be seen as a local solvability result for a semi linear elliptic equations with a prescribed behavior at a point. It seems that the available results in this context only prescribe the first $2$ derivatives at one point, see for instance \cite[Section 14.3, Proposition 3.3]{T:bookIII}. So, our result constructs much more local solutions, and actually all of them.

\bigskip

In the context of conformal maps in dimension 2, one can adapt in a straightforward way Theorem~\ref{thmH1regHarmon} to derive a statement close to a point in the spirit of Theorem~\ref{thmexistzero}.

For harmonic maps in dimension at least 3, it seems that a similar result should hold as well, but one would have to first prove an extra gain (for example due to a null condition as in \eqref{hypoNullintro}).

\subsection{Acknowledgments}

The  authors warmly thank Fabrice Bethuel and Didier Smets for many references and insights on harmonic maps.

RC acknowledges support from the University of Strasbourg Institute for Advanced Study (USIAS) for a Fellowship within the French national programme ``Investment for the future'' (IdEx-Unistra).

\section{The linear flow and Duhamel formulation} \label{sec:Duh}

The starting point of the analysis is the following. If $\Delta u =0$ then denoting for $(t,y) \in \m R \times \m S^{d-1}$ the conformal change of variable
\[ v(t,y) = e^{\frac{(d-2)t}{2}} u(e^t y),\]
$v$ solves 
\begin{align}\label{eq:D0} \partial_{tt} v - \gh D^2 v =0.\end{align}
This equation is not a well behaved evolution equation but can still be amenable to an analysis.

To make this more precise, we introduce suitable $Y_{s,t}$ spaces, intimately related to the $Z$ spaces (after a conformal change of variables) in which the results are stated.

\subsection{The $Y_{s,t}$ spaces}

For a function $u$ defined of $\m S^{d-1}$, let
\[ \| u \|_{Y_{s,t}} = \| e^{t \gh D} u \|_{H^s(\m S^{d-1})}, \]
or equivalently,
\[  \| u \|_{Y_{s,t}}^2 =  \sum_{\ell =0}^{+\infty}  \langle \ell \rangle^{2s} e^{ 2 \left( \ell + \frac{d-2}{2} \right) t} \| P_\ell u \|_{L^2(\m S^{d-1})}^2. \]
As before, the space $Y_{s,t}$ is defined as the completion of $L^2_0(\m S^{d-1})$ for the $\| \cdot \|_{Y_{s,t}}$ norm. Note that for $0\le t\le t'$, we have the inclusion $Y_{s,t'}\subset Y_{s,t}\subset H^s$  together with $\nor{u}{H^s}\le \nor{u}{Y_{s,t}}\le \nor{u}{Y_{s,t'}}$.
Given a regularity index $s \ge 1$ and a ``time'' $t \ge 0$, we also define the norm $\q Y_{s,t}$ by
\begin{equation} \label{def:Yst_norm}
\| (v,\vp) \|_{\q Y_{s,t}} = {\sup_{\tau \ge t} } \left(  \| v(\tau) \|_{Y_{s,\tau}} + \| \vp(\tau) \|_{Y_{s-1,\tau}} \right).
\end{equation}
for $(v,\vp)$ defined on $[t,+\infty) \times \m S^{d-1}.$
The purpose of the second variable $\vp$ is to take into account the time derivative $\partial_{t} v$ for a solution, as is usual for second order evolution equations. 

The space $\q Y_{s,t}$ is defined as the space of functions $(v,\vp)$ defined on $[t,+\infty) \times \m S^{d-1}$, so that for all $\tau\ge t$, $v|_{[\tau,+\infty)}\in \mc C ([\tau,+\infty), Y_{s,\tau})$ and $\dot v|_{[\tau,+\infty)}\in \mc C ([\tau,+\infty), Y_{s-1,\tau})$ and $\| (v,\vp) \|_{\q Y_{s,t}} < +\infty$. This is in the same spirit as was done for the $\q Z_{s,t}$ spaces; we make use that the $Y_{s,t}$ spaces are decreasing in $t$ (for the inclusion ordering).

We will also sometimes need the following translated version for $t_{0}\ge 0$: 
 \[ \| (v,\vp) \|_{\q Y_{s,t}^{t_{0}} }= {\sup_{\tau \ge t}}  \left( \| v(\tau) \|_{Y_{s,\tau-t_{0}}} + \| \vp(\tau) \|_{Y_{s-1,\tau-t_{0}}} \right), \]
and we say that $(v,\dot v) \in \q Y_{s,t}^{t_0}$ when for all $\tau\ge t$, $v|_{[\tau,+\infty)}\in \mc C ([\tau,+\infty), Y_{s,\tau-t_{0}})$ and $\dot v|_{[\tau,+\infty)}\in \mc C ([\tau,+\infty), Y_{s-1,\tau-t_{0}})$ and $\| (v,\vp) \|_{\q Y_{s,t}^{t_{0}} } < +\infty$.
We will essentially always consider these spaces for $t \ge t_0$: in which case it is a weaker space that $\q Y_{s,t}$. More precisely, there hold 
\begin{align}
\label{inegYtri}
\forall t \ge t_0 \ge 0, \quad \| \bs v \|_{\q Y_{s,t}^t}\le \| \bs v \|_{\q Y_{s,t}^{t_{0}} }\le \|\bs v \|_{\q Y_{s,t}} \le  \| \bs v \|_{\q Y_{s,t_0}}. \end{align}
for $\bs v=(v,\vp)\in \q Y_{s,t_0}$.
Note that the inequalities between $Y$ norms imply (given $t \ge 0$)
\begin{multline}
{\sup_{\tau \ge t}}  \| v(\tau) \|_{Y_{s,\tau-t_{0}}}={\sup_{t\le \tau }} \left( {\sup_{t\le \tau' \le \tau}}  \| v(\tau) \|_{Y_{s,\tau'-t_{0}}} \right) \\ ={\sup_{t\le \tau' }} \left( {\sup_{\tau' \le \tau}}  \| v(\tau) \|_{Y_{s,\tau'-t_{0}}} \right)
= {\sup_{\tau' \ge t}} \nor{u}{\mc C ([\tau',+\infty), Y_{s,\tau'-t_{0}})}.
\end{multline}
Similar equality holds for $\dot v$ and we easily get that for $t\ge t_0\ge 0$, the spaces $\q Y_{s,t}^{t_{0}}$ are Banach spaces as intersection of Banach spaces.

\bigskip

The spaces $\q Y_{s,t}$ are well suited to linear solutions of the Laplace equation, in conformal variables, as it is shown in the next paragraphs. In particular, we have $\nor{(v_L,\partial_t v_L)}{\q Y_{s,t} }\approx \| v_0 \|_{H^s(\m S^{d-1})}$ where $v_L$ is the decaying linear solution with Dirichlet data $v_0$ at $t=0$ (see Lemma \ref{lmlienlinY} for a more precise statement). 

The $Z$ spaces are in fact the $Y$ spaces after conformal transform. More precisely, for $r>0$ and $s\in \R$ and a function $u$ defined on $\m S^{d-1}$, then
\begin{align}
\label{equiYZ}
  \nor{u}{Z_{s,r}^{\infty}} = r^{\frac{d-2}{2}}\| u \|_{Y_{s,\log(r)}} \quad \text{and} \quad
   \nor{u}{Z_{s,r}^{0}} = r^{\frac{d-2}{2}}\| u \|_{Y_{s,-\log(r)}}.
   \end{align} 
Similarly,  for a function defined on $\R^{d}\setminus B(0,r_{0})$ or $B(0,r_{0})$ respectively, if we denote for $(t,y) \in \m R \times \m S^{d-1}$
\begin{align} 
v^{\infty}(t,y) &=e^{\frac{(d-2)t}{2}}u(e^ty) & &  \hspace{-10mm} \text{close to infinity and } \label{def:conf_uv_infty} \\ 
v^{0}(t,y) & =e^{-\frac{(d-2)t}{2}}u(e^{-t}y) & &  \hspace{-10mm} \text{close to zero.} \label{def:conf_uv_zero}
\end{align}
one can relate the $\q Y$ spaces and the $\q Z$ spaces:
\begin{gather}
  \label{equivnorm}
      \nor{u}{\q Z_{s,r}^{\infty}} =\| (v^{\infty}, \partial_t v^\infty) \|_{\q Y_{s,\log(r)}^{\log(r)}} \quad \text{and} \quad 
   \nor{u}{\q Z_{s,r}^{0}} =\| (v^{0}, \partial_t v^0 )\|_{\q Y_{s,-\log(r)}^{-\log(r)}}. 
\end{gather}
 
Finally, as for $Z$ spaces, we drop the index $t$ when $t=0$, for example $\q Y_s := \q Y_{s,0}$. Observe that \eqref{equivnorm} also implies that the space  $\q Z_{s,r_{0}}^{\infty}$ and $\q Z_{s,r_{0}}^{0}$ are Banach spaces with their defined norm.

\subsection{The linear flow}
Consider $u$, defined on $\R_+ \times \m S^{d-1}$, solution to
\begin{align} \label{eq:D}
\partial_{tt} u - \gh D^2 u = F.
\end{align}
Equivalently, $\bs u = (u, \partial_t u)$ solves
\begin{align} \label{eq:sysD} \partial_t \bs u = \begin{pmatrix}
0 & 1 \\
\gh D^2 & 0 
\end{pmatrix} \bs u + \begin{pmatrix}
0 \\ F 
\end{pmatrix}
\end{align}
Notice that the resolvent operator writes
\[ \q S(t) := \exp \begin{pmatrix}
0 & t \\
t \gh D^2 & 0 
\end{pmatrix} = \begin{pmatrix}
\cosh(t \gh D) & \ds \frac{\sinh(t \gh D)}{\gh D} \\
\gh D \sinh(t \gh D) &\ds  \cosh(t \gh D) \vphantom{\int} 
\end{pmatrix} \]
Although $\q S(t)$ is well defined on $(L^2_0)^2$, the growing modes prevent it from defining a semi-group on any reasonable space like  $H^s(\m S^{d-1}) \times H^{s-1}(\m S^{d-1})$. 
We will however show that one can construct a wave operator at $+\infty$ in $\q Y_s$, for well chosen final data with no growing modes.

We therefore consider a linear solution with no growing modes, that is of the form
\begin{equation} \label{def:u+}
\bs u_0 = (u_0, - \gh D u_0) \quad \text{for some } u_0 \in H^s(\m S^{d-1}),
\end{equation}
so that
\[ \q S(t) \bs u_0 = (e^{-t \gh D} u_0, - \gh D e^{-t \gh D} u_0). \]
Observe that for any non zero $v \in H^{s-1}(\m S^{d-1})$ , $\| \q S(t)(0,v) \|_{\q Y_{s,t}} \to +\infty$ as $t \to +\infty$ or it is infinite. Hence, given $u_0 \in H^s(\m S^{d-1})$, $\q S(t)\bs u_0$ is the only bounded solution (in $\q Y_{s,t})$ of \eqref{eq:sysD} with $F=0$ and initial data  $u_0$ at time $0$. If we denote 
\begin{align}
\label{def:u_L}
u_L(x) =  |x|^{-\frac{d-2}{2}} (\q S(t)\bs u_0)(\ln |x|, x/|x|), \quad \text{for } |x| \ge 1,
\end{align}
then $u_L \in \q Z^\infty_{s,1}$ and solves
\[ \Delta u_L =0 \quad \text{on} \quad \m R^d \setminus B(0,1), \quad u_L|_{\m S^{d-1}} = u_0.\]

\bigskip

We first measure our solutions in our norms. The following Lemma explains that $\q Y_{s,t}$ is the natural space for linear solutions with initial datum in $H^s(\m S^{d-1})$ at $t=0$ while $\q Y_{s,t}^{t_{0}}$ is adapted when the initial datum is given at $t=t_{0}$.

\begin{lem} 
\label{lmlienlinY}
Let $u_0 \in H^s(\m S^{d-1})$. Then,  for any $t _0 \ge 0$, $\q S(\cdot-t_0) \bs u_0 \in \q Y_{s,t_0} $ together with the estimates uniform in $t \ge t_0$,
\begin{gather} \label{est:Hs_Yst0}
\| u_0 \|_{H^s(\m S^{d-1})}\le \| \q S(\cdot -t_{0}) \bs u_0\|_{\q Y_{s,t}^{t_{0}}}\le \frac{d+2}{2} \| u_0 \|_{H^s(\m S^{d-1})},  \\
\| \q S(\cdot) \bs u_0\|_{\q Y_{s,t}^{t_{0}}}\le C \| e^{-t_{0} \gh D}u_0 \|_{H^s(\m S^{d-1})}\le C e^{-t_{0} \frac{d-2}{2}}\| u_0 \|_{H^s(\m S^{d-1})} \label{est:Hs_Yst0_exp}
\end{gather}
If $\ds P_{\ell} u_0 = 0$ for all $\ell< \ell_0$, then we have also  
\begin{gather} \label{est:Hs_Yst0_exp_2d}  
\| \q S(\cdot) \bs u_0\|_{\q Y_{s,t}^{t_{0}}}\le  C e^{-\left(\ell_0+\frac{d-2}{2}\right)t_{0}}\| u_0 \|_{H^s(\m S^{d-1})}.
\end{gather}

\end{lem}

We will often use \eqref{est:Hs_Yst0} with $t_0=0$.
\bnp
From the definitions,
\begin{align*} 
\| \q S(\cdot -t_0) \bs u_0 \|_{\q Y_{s,t}^{t_0}} & = \| \q S(\cdot ) \bs u_0 \|_{\q Y_{s,t}}  = \sup_{\tau \ge t} \left(  \| e^{-\tau \gh D} u_0 \|_{Y_{s,\tau}} + \| \gh D e^{-t \gh D} u_0 \|_{Y_{s-1,\tau}} \right) \\
& = \| u_0 \|_{H^s(\m S^{d-1})} + \| \gh D u_0 \|_{H^{s-1}(\m S^{d-1})},
\end{align*}
and \eqref{est:Hs_Yst0} follows as $\ell+ \frac{d-2}{2} \le\frac{d}{2} \langle \ell \rangle$ for all $\ell \ge 0$.  
For \eqref{est:Hs_Yst0_exp}, we write for $t\ge t_{0}$,
\begin{align*}
 \| \q S( \cdot ) \bs u_0 \|_{\q Y_{s,t}^{t_{0}}} & =\sup_{\tau \ge t}  \| e^{-\tau \gh D} u_0 \|_{Y_{s,\tau-t_{0}}} + \sup_{\tau \ge t} \| \gh D e^{- \tau \gh D} u_0 \|_{Y_{s-1,\tau-t_{0}}}\\
  & = \sup_{z\ge t-t_{0}}  \| e^{-(z+t_{0}) \gh D} u_0 \|_{Y_{s,z}} + \sup_{z\ge t-t_{0}} \| \gh D e^{-(z+t_{0}) \gh D} u_0 \|_{Y_{s-1,z}}\\
& \le C \| e^{-t_{0} \gh D}u_0 \|_{H^s(\m S^{d-1})}.
\end{align*}
For the second inequality, notice that the first eigenvalue of the operator $\gh D$ is $\frac{d-2}{2}$. 

Finally, for \eqref{est:Hs_Yst0_exp_2d}: if $\ds P_{\ell}u_0=0$ for all $\ell< \ell_0$, then $u_0\in \mathop{\mathrm{Span}}(\phi_{\ell,m}, \ell\ge \ell_0,m\le N_{\ell})$. The restriction of $\gh D$ to this space has first eigenvalue $\ell_0+\frac{d-2}{2}$, which gives the bound. It remains to discuss continuity: for this, fix $\tau\ge t_0$ and $\tau_1 \ge \tau$. Denote $u(\tau')=e^{-(\tau' -t_0)\opD }u_0$ for $\tau'\in [\tau,+\infty)$. Then
\begin{align*} 
\nor{u(\tau')-u(\tau_1)}{Y_{s,\tau-t_0}} & =\nor{e^{-(\tau' -t_0)\opD }u_0-e^{-(\tau_1 -t_0)\opD }u_0}{Y_{s,\tau-t_0}} \\
& = \nor{e^{-(\tau' -\tau)\opD }u_0-e^{-(\tau_1 -\tau)\opD }u_0}{H^s}. \end{align*}
This clearly converges to zero as $\tau'\to \tau_1$, $\tau'\ge \tau$ (for instance by approximating $u_0$ in $L^2_0$). The same holds for the derivative.
\enp

\subsection{Bounds on the Duhamel term}

Given a function $\bs u_+$ as in \eqref{def:u+}, we would like to construct nonlinear solutions $\bs u$ to
\[ \partial_{tt} u - \gh D^2 u = F(\bs u), \quad \bs u(t) \simeq \q S(t) \bs u_+ \quad \text{as} \quad t \to +\infty.\]
More precisely, on the difference
$\bs v(t) = \bs u(t) - \q S(t) \bs u_+$, we are looking for a solution to the problem
\[ \begin{cases}
 \partial_{tt} v - \gh D^2 v = F(\bs v(t) + \q S(t) \bs u_+) \\
 \bs v(t) \to 0 \quad \text{as } t \to +\infty
 \end{cases} \]
($F$ is a given nonlinear functional). The Duhamel formulation between times $t \le s$ writes, at least formally,
 \[ \q S(t-s) \bs v(s) =  \bs v(t) + \int_t^s \q S(t-\tau)  \begin{pmatrix}
0 \\ F (\bs v(\tau) + \q S(\tau) \bs u_+)
\end{pmatrix} d\tau. \]
Letting $s \to +\infty$, it is reasonable to assume that the left hand side tends to $0$, and we want to solve
 \begin{equation} \label{eq:conf_ex}
 \bs v(t) = -  \int_t^{+\infty} \q S(t-\tau)  \begin{pmatrix}
0 \\ F (\bs v(\tau) + \q S(\tau) \bs u_+)
\end{pmatrix} ds.
\end{equation}
This leads us to consider the map
\begin{align} \label{def:Phi}
\Phi: F \mapsto  \left( t \mapsto -  \int_t^{+\infty} \q S(t-\tau)  \begin{pmatrix}
0 \\ F (\tau)
\end{pmatrix} d\tau \right)
\end{align}

\begin{lem}
\label{lmDuhamelinfty}
Let $s\ge 1$ and $t_0 \ge 0$. Let $F\in \mc C([t,+\infty),Y_{s-1,t-t_0})$ for any $t\ge t_0$ and assume
\begin{align*}
\text{if } d \ge 3: \quad &   
\int_{t_0}^{+\infty} \| F(\tau) \|_{Y_{s-1,\tau-t_0}} d\tau <+\infty, \\
\text{if } d =2: \quad & \int_{t_0}^{+\infty} \left( \| F(\tau) \|_{Y_{s-1,\tau-t_0}} + (\tau-t_0) \| P_0 F(\tau) \|_{L^2(\m S^{d-1})} \right) d\tau <+\infty.
\end{align*}
Then $\Phi(F) \in \q Y_{s,t_0}^{t_0}$ and there hold for $t \ge t_0 \ge 0$:
\begin{align} \label{est:lin_Duhamel}
\text{if } d \ge 3: \quad \| \Phi(F) \|_{\q Y_{s,t}^{t_0}} & \lesssim \int_{t}^{+\infty} \| F(\tau) \|_{Y_{s-1,\tau-t_0}} d\tau. \\
\text{if } d =2: \quad \| \Phi(F) \|_{\q Y_{s,t}^{t_0}} & \lesssim \int_{t}^{+\infty} \left( \| F(\tau) \|_{Y_{s-1,\tau-t_0}} + (\tau-t) \| P_0 F(\tau) \|_{L^2(\m S^{d-1})} \right) d\tau. \nonumber
\end{align}
In particular, $\| \Phi(F) \|_{\q Y_{s,t}^{t_0}} \to 0$ as $t \to +\infty$.

Furthermore, $\bs v=(v,\vp)= \Phi(F)$ satisfies the equation
\[ \partial_{tt} v - \gh D^2 v = F,\quad \vp=\partial_t v \quad \text{on } (t_0,+\infty) \times \mathbb{S}^{d-1}, \]
in the sense of distributions.

\end{lem}

\begin{proof}
Let us start with a few preliminary estimates. Notice that (uniformly) for $\ell \ge 1$, $m \le N_\ell$, $t \ge 0$ and $\tau \in \R$,
\begin{align*}
 \left\|  \frac{\sinh\left(\tau \gh D \right)}{\gh D }  \phi_{\ell,m} \right\|_{Y_{s,t}} & = e^{t (\ell+\frac{d-2}{2})} \left< \ell \right>^s \frac{\sinh( \tau( \ell+ \frac{d-2}{2}))}{ \ell+ \frac{d-2}{2}}  \\
& \lesssim  \left< \ell \right>^{s-1} \left( e^{(t +\tau)( \ell+\frac{d-2}{2})} +   e^{(t -\tau)( \ell+\frac{d-2}{2})} \right).
 \end{align*}
If $d \ge 3$, this is true also for $\ell=0$, and we infer
\begin{equation}
\label{est:D_3d}
\left\|  \frac{\sinh\left(\tau \gh D \right)}{\gh D }  f \right\|_{Y_{s,t}} \lesssim \left\| \left( e^{(t +\tau) \gh D} +   e^{(t -\tau) \gh D} \right) f \right\|_{H^{s-1}}.
\end{equation}
If $d=2$, then for $\ell=0$, $\ds \left\|  \frac{\sinh\left(\tau \gh D \right)}{\gh D }  \phi_{0,m} \right\|_{Y_{s,t}} = |\tau|$,
and so
\[ \left\|  \frac{\sinh\left(\tau \gh D \right)}{\gh D }  P_0 f \right\|_{Y_{s,t}} = |\tau |\| P_0 f \|_{L^2(\m S^{d-1})}. \]
Therefore, 
\begin{align}
\left\|  \frac{\sinh\left(\tau \gh D \right)}{\gh D }  f \right\|_{Y_{s,t}} & \lesssim \left\| \left( e^{(t +\tau) \gh D} +   e^{(t -\tau) \gh D} \right) (\Id-P_0) f \right\|_{H^{s-1}} +  \tau \| P_0 f \|_{L^2(\m S^{d-1})} \nonumber \\
& \lesssim \left\| \left( e^{(t +\tau) \gh D} +   e^{(t -\tau) \gh D} \right) f \right\|_{H^{s-1}} +  \tau \| P_0 f \|_{L^2(\m S^{d-1})}. \label{est:D_2d}
\end{align}
Similarly, for any $d\ge 2$, $\ell \ge 0$, $t \ge 0$ and $\tau \in \R$,
\[
 \left\| \cosh \left(\tau \gh D \right)  \phi_{\ell,m} \right\|_{Y_{s-1,t}} \lesssim  \left< \ell \right>^{s-1} \left( e^{(t +\tau)( \ell+\frac{d-2}{2})} +   e^{(t -\tau)( \ell+\frac{d-2}{2})} \right).
\]
so that
 \begin{align} \label{est:Dd} \left\| \cosh \left(\tau \gh D \right)  f \right\|_{Y_{s-1,t}} \lesssim  \left\| \left( e^{(t +\tau) \gh D} +   e^{(t -\tau) \gh D} \right) f \right\|_{H^{s-1}}.
 \end{align}
In the same way, if $\phi,\psi$ are two real functions such that $|\phi| \le \psi$ on $[0,+\infty)$, then for any $\ell,m$, $\| \phi(\gh D) \phi_{\ell,m} \|_{H^s} \le \| \psi(\gh D) \phi_{\ell,m} \|_{H^s}$ and so 
\[ \| \phi(\gh D) f \|_{H^s} \le \| \psi(\gh D) f \|_{H^s}. \]
for any $f$ such that the right hand side is finite.

\bigskip

We can now proceed with the proof of \eqref{est:lin_Duhamel}; we will only do the case $d=2$ (for $d\ge 3$, the proof is similar but simpler). Denote $\Phi(F) = (v,\vp)$, we first estimate $v$ using \eqref{est:D_2d} (exchanging letters):
\begin{align*}
\MoveEqLeft \| v(\tau) \|_{Y_{s,\tau-t_0}} = \left\| \int_\tau^{+\infty} \frac{\sinh\left((\tau- \sigma) \gh D \right)}{\gh D } F(\sigma) d \sigma \right\|_{Y_{s,\tau-t_0}} \\
& \lesssim \int_{\tau}^{+\infty} \left( \nor{e^{(2\tau-t_0 -\sigma) \gh D} F(\sigma)}{H^{s-1}} + \nor{e^{(\sigma-t_0) \gh D }F(\sigma)}{H^{s-1}} + (\sigma-\tau) \| P_0 F(\sigma) \|_{L^2} \right) d\sigma\\
& \lesssim \int_{\tau}^{+\infty} \left( \nor{F(\sigma)}{Y_{s-1, \sigma-t_0}} + (\sigma-\tau) \| P_0 F(\sigma) \|_{L^2} \right) d \sigma
\end{align*}
where we have used $2\tau- \sigma\le \sigma$ when $\sigma \ge \tau$, so that $e^{(2\tau-\sigma-t_0) \cdot} \le e^{(\sigma-t_0) \cdot}$ on $[0,+\infty)$. Similar arguments give that $v\in \mc C([\tau,+\infty),Y_{s,\tau-t_0})$ for any $\tau\ge t_0$.

Using \eqref{est:Dd}, we conclude also that $\vp\in \mc C([\tau,+\infty),Y_{s-1,\tau-t_0})$ for any $\tau\ge t_0$ and
\begin{align*}
\| \vp(\tau) \|_{Y_{s-1,\tau-t_0}} & =  \left\| \int_{\tau}^{+\infty}\cosh\left((\tau-\sigma) \gh D\right)F(\sigma) d\sigma \right\|_{Y_{s-1,\sigma-t_0}} \\
& \lesssim \int_{\tau}^{+\infty} \nor{F(\sigma)}{Y_{s-1,\sigma-t_0}} d\sigma.
\end{align*}
Summing up, we obtain 
\begin{align*}
\| \Phi(F) \|_{\q Y_{s,t}^{t_0}}& \lesssim \sup_{\tau\ge t}\int_{\tau}^{+\infty} \left( \| F(\sigma) \|_{Y_{s-1,\sigma-t_0}} + (\sigma-\tau) \| P_0 F(\sigma) \|_{L^2(\m S^{d-1})} \right) d\sigma \\
&\lesssim \sup_{\tau\ge t}\int_{t}^{+\infty} \left( \| F(\sigma) \|_{Y_{s-1,\sigma-t_0}} + (\sigma-t) \| P_0 F(\sigma) \|_{L^2(\m S^{d-1})} \right) d\sigma \\
&= \int_{t}^{+\infty} \left( \| F(\sigma) \|_{Y_{s-1,\sigma-t_0}} + (\sigma-t) \| P_0 F(\sigma) \|_{L^2(\m S^{d-1})} \right) d\sigma.
\end{align*}

which is \eqref{est:lin_Duhamel}. The convergence to zero  is immediate by dominated convergence.

Now observe that
\begin{align*}
\partial_t v(t) & = -\partial_{t}\left(\int_t^{+\infty} \frac{\sinh\left((t-\tau) \gh D \right)}{\gh D } F(\tau)d\tau \right) \\ & =-\int_{t}^{+\infty}\cosh\left((t-\tau) \gh D\right)F(\tau)d\tau = \vp(t).
\end{align*}

 To check that $v$ is solution of the equation, let $\psi(t,x)=h(t)\phi_{\ell,m}(x)$ where $h\in \mc C^{\infty}_{0}((t_0,+\infty))$. Denote $\Omega = (t_0,+\infty) \times \m S^{d-1}$, we easily check (using integration by part where needed) that 
\begin{align*}
&\left< (\partial_{tt}-\gh D^2) v, \psi\right>_{\mathcal{D}'(\Omega),\mathcal{D}'(\Omega) }=\left< v, (\partial_{tt}-\gh D^2) \psi\right>_{\mathcal{D}'(\Omega), \mathcal{D}(\Omega) }\\
& = \left< v, \left(h''(t)- \left(\ell+\frac{d-2}{2} \right)^{2}h(t) \right) \phi_{\ell,m}(x)\right>_{\mathcal{D}'(\Omega),\mathcal{D}(\Omega) }\\
& = -\int_{t_0}^{+\infty} \left(h''(t)- \left(\ell+\frac{d-2}{2} \right)^{2}h(t) \right) \\
& \qquad \quad \cdot \left( \int_{t}^{+\infty}\frac{\sinh\left((t-\tau) \left(\ell+ \frac{d-2}{2} \right)\right)}{(\ell+\frac{d-2}{2})}\left<F(\tau),\phi_{\ell,m}\right>_{\mathcal{D}'( \mathbb{S}^{d-1}),\mathcal{D}(\mathbb{S}^{d-1}) }d\tau \right) dt\\
&=\int_{t_0}^{+\infty}h(t)\left<F(t),\phi_{\ell,m}\right>_{\mathcal{D}'( \mathbb{S}^{d-1}),\mathcal{D}'(\mathbb{S}^{d-1}) }d\tau dt=\left<F,\psi\right>_{\mathcal{D}'(\Omega),\mathcal{D}(\Omega)}.
\end{align*}
This gives the result by density of linear combinations of such functions.
\end{proof}

For the uniqueness statements, the problem \eqref{eq:conf_ex} and the functional \eqref{def:Phi} are too demanding. We are led to consider a small variant of it,  related to the Dirichlet boundary condition, and we will obtain similar bound. 

The main difference is that we are prescribing some datum at zero and imposing some decay at $+\infty$. The model example is the ODE, $\ddot{x}-x=f$ and we can check that for $f$ sufficiently decaying, there is a unique solution exponentially decaying so that $x(0)=0$\footnote{More precisely, if $R(t) = \begin{pmatrix} \cosh t & \sinh t \\ - \sinh t & - \cosh t \end{pmatrix}$ is the resolvent, and if $e^t f(t) \in L^1$ then $\ds F: = \int_0^{+\infty} R(-s) \begin{pmatrix} 0 \\ f(s) \end{pmatrix} ds$ is convergent, and the sought for solution is  
\[ (x,\dot x) = R(t) (X_0 + F) -   \int_t^{+\infty } R(t -s) \begin{pmatrix} 0 \\ f(s) \end{pmatrix}  ds, \] where  $X_0 = (0,x_1)$ is defined by $x_1 + f_0+f_1 =0$ with $(f_0,f_1)=F$ being the coordinates of $F$.}. We use a similar fact for our operator, which leads to a kind of modified Duhamel formula. The first guess would be to consider 
\[ (v(t),\partial_{t}v)=\int_0^{t} \q S(t-\tau)  \begin{pmatrix}
0 \\ F (\tau)
\end{pmatrix} d\tau \]
which is well defined and has value $(0,0)$ at $t=0$. Yet, this would be to impose the Dirichlet and Neumann value and this expression might contain some exponentially growing modes. So, it is more natural to constrain only the Dirichlet value $v(0)=0$. So, this leads us to consider 
\[ u_{+,F} \text{ the first component of } \int_0^{+\infty} \q S(-\tau) \begin{pmatrix}
0 \\ F (\tau)
\end{pmatrix} d\tau \quad \text{and} \quad \bs u_{+,F}= (u_{+,F}, - \gh D u_{+,F}), \]
and the map
\begin{align} \label{def:PhiD}
\Phi^{D}: F \mapsto \q S(t)\bs u_{+,F}  -  \int_t^{+\infty} \q S(t-\tau)  \begin{pmatrix}
0 \\ F (\tau)
\end{pmatrix} d\tau
\end{align}
Observe that this expression is not local: the value close to zero of $\Phi^{D}(F) $ is influenced by all the value of $F$ everywhere. This was not the case at infinity: $\Phi(F)$ for large $t$ only depends on larger times.

The expression giving $\Phi^D(F)$ is well defined if $u_{+,F} \in H^s(\m S^{d-1})$, and a condition for this is the purpose of the following lemma.

\begin{lem}
\label{lmDuhamelzero}
Let $s\ge 1$. Assuming that the right-hand side in the estimates below is finite, $\Phi^D(F) \in \q Y_s$ and there hold:
\begin{align} \label{est:lin_DuhamelDir}
\text{if } d \ge 3: \quad \| \Phi^{D}(F) \|_{\q Y_{s}} & \lesssim \int_{0}^{+\infty} \| F(\tau) \|_{Y_{s-1,\tau}} d\tau. \\
\text{if } d =2: \quad \| \Phi^{D}(F) \|_{\q Y_{s}} & \lesssim \int_{0}^{+\infty} \left( \| F(\tau) \|_{Y_{s-1,\tau}} + \tau \| P_0 F(\tau) \|_{L^2(\m S^{d-1})} \right) d\tau. \nonumber
\end{align}
Furthermore, $\bs v= \Phi^{D}(F)$ satisfies the equation
\[
\left\{ \begin{aligned}
\partial_{tt} v - \gh D^2 v & = F,\quad & &\textnormal{in }\R^{*}_{+}\times \mathbb{S}^{d-1}\\
 v(0) & =0, \quad  & &  \textnormal{in } \mathbb{S}^{d-1}.
 \end{aligned} \right. \]
where the first equation is meant in the sense of distributions.

Moreover, $\bs u_{+,F}$ is the unique initial datum in $H^{s}\times H^{s-1}(\mathbb{S}^{d-1})$ so that
\begin{align}
\label{cvceinftysolDirDuham}
 \| \bs v - \q S(\cdot) \bs u_{+,F}  \|_{\q Y_{s,t}} \to 0 \quad \text{as} \quad t \to +\infty. 
\end{align}
\end{lem}
\bnp
\eqref{est:D_2d} with $t=0$ gives when integrated in $\tau$
\begin{align*}
\nor{u_{+,F}}{H^s(\m S^{d-1})}&\lesssim  \int_{0}^{+\infty} \left( \nor{e^{-\tau \gh D} F(\tau)}{H^{s-1}} + \nor{e^{\tau \gh D }F(\tau)}{H^{s-1}} + \tau \| P_0 F(\tau) \|_{L^2} \right) d\tau\\
& \lesssim \int_{0}^{+\infty} \left( \nor{F(\tau)}{Y^{s-1}_\tau} + \tau \| P_0 F(\tau) \|_{L^2} \right) d\tau 
\end{align*}
Now, due to \eqref{est:Hs_Yst0} in Lemma \ref{lmlienlinY} (with $t=t_0=0$), we conclude that \eqref{est:lin_DuhamelDir} holds for the first term $\q S(\cdot) \bs u_{+,F}$ of $\Phi^{D}(F)$. The second term is $\Phi(F)$, which satisfies similar estimates as seen in Lemma \ref{est:lin_Duhamel}. The last statement is direct due to the convergence $\| \Phi(F) \|_{\q Y_{s,t}} \to 0$ given in Lemma \ref{est:lin_Duhamel}. The uniqueness is also direct in view of \eqref{est:Hs_Yst0}.
\enp

\section{Some properties of \texorpdfstring{$Y_{s,t}$}{Yst}  and \texorpdfstring{$Z_{s,t}$}{Zst} spaces }

We start by recalling the following result by Sogge for eigenfunctions of the Laplace-Beltrami operator on a compact manifold, see for instance \cite[Corollary 5.1.2]{SoggeBookFIO}.
\begin{lem}[Sogge]
\label{lmSogge}
Let $M$ be a compact Riemannian manifold without boundary of dimension $n$. Then, there exists $C>0$ so that we have 
\begin{align*}
\nor{\phi_{\lambda}}{L^{\infty}(M)} \le C \lambda^{\frac{n-1}{2}}\nor{\phi_{\lambda}}{L^{2}(M)}
\end{align*}
for any $\phi_{\lambda}\in L^{2}(M)$ satisfying $-\Delta_{g}\phi_{\lambda}=\lambda^{2}\phi_{\lambda}$.
\end{lem}
For $n=d-1$, this proves in particular that for all $\ell \in  \m N$, and $u \in L^2(\m S^{d-1})$, as $P_\ell u$ is an eigenfunction for $-\Delta_{\m S^{d-1}}$ with eigenvalue $\ell(\ell+d-2)$: 
\begin{align} \label{est:sogge_sphere}
 \| P_\ell u \|_{L^{\infty}(\m S^{d-1})} \le C \left< \ell \right>^{\frac{d}{2}-1} \| P_\ell u \|_{L^{2}(\m S^{d-1})}.
\end{align}

\subsection{Product law in the \texorpdfstring{$Y_{s,t}$}{Ys,t} spaces}

Our goal in this paragraph is to prove the following result regarding the product of two functions in $Y_{s,t}$.

\begin{prop}
\label{lmprodanalytique}
Let $d \ge 2$ and $s> \frac{d}{2} + \frac{1}{2}$. There exists $C>0$ so that 
\begin{align}
\label{estimprodYst}
\forall t \ge 0, \ \forall u,v \in Y_{s,t}, \quad \nor{uv}{Y_{s,t}}\le  Ce^{- \frac{d-2}{2} t}\nor{u}{Y_{s,t}}\nor{v}{Y_{s,t}}.
\end{align}

\end{prop}
We will actually prove a slightly more general version of Proposition \ref{lmprodanalytique}.
Recall $P_{\ell}$ is the projector defined in \eqref{def:Pell}. One important property will be the following result on product of spherical harmonics. Similar statements have been used in the context of nonlinear Schr\"odinger equations on  $\m S^{d}$ (see \cite{InventionesBGT}). 
\begin{lem}
\label{prodHarmo}
Let $\ell_1$ and $\ell_2\in \N$ and $Y_{\ell_{i}}$ be two spherical harmonics of degree $\ell_{i}$. The product $\phi_{\ell_1} \phi_{\ell_2}$ can be written as a sum of spherical harmonics of degree $\ell$ with $|\ell_1-\ell_2| \le \ell \le \ell_1+\ell_2$. Equivalently, if $\ell < |\ell_1-\ell_2|$ or if $\ell > \ell_1+\ell_2$, then
\[ P_{\ell}(\phi_{\ell_1} \phi_{\ell_2}) =0. \]
\end{lem}

\begin{proof}
We refer for instance to \cite[Lemma 4.50, Section 4.E.3]{GallotHulinLaf} for the upper bound on $\ell$. For the second part, we assume 
\[ \int_{\omega\in \mathbb{S}^{d-1}} \phi_{\ell_1} \phi_{\ell_2} \phi_{\ell_{3}}d\omega\ne 0. \]
Without loss of generality, we can furthermore assume $\ell_2\ge \ell_1$. We apply the first part of the Lemma to $\ell_1$ and $\ell_{3}$ to get $\ell_2 \le \ell_1+\ell_{3}$, that is $\ell_{3}\ge \ell_2- \ell_1=| \ell_2- \ell_1|$. 
\end{proof}

\begin{definition}
A sequence ${\underline \beta} = (\beta_\ell)_{\ell \in \m N}$ such that $\beta_\ell >0$ is said to be an \emph{easing sequence} with \emph{factor} $\kappa >0$, if 
\begin{equation} \label{def:beta_kappa}
\forall i,j,k \in \m N, \quad  k \le i+j \imp \beta_k \le \kappa \beta_i \beta_j.
\end{equation}
Given such a sequence, we define the norm for functions defined on $\m S^{d-1}$,
\[ \| v \|_{N(s,\underline \beta)}^2 = \sum_{\ell \in \m N} \left< \ell \right>^{2s} \beta_\ell^2  \| P_\ell v \|_{L^2}^2. \]
\end{definition}

\begin{lem}
Let $\underline \beta$ be an easing sequence with factor $\kappa$, and $s > \frac{d}{2} + \frac{1}{2}$. If $u,v$ have finite $N(s, \underline \beta)$ norm, then so does $uv$ and there hold 
\[ \| uv \|_{N(s,\underline \beta)} \le C \kappa \| u \|_{N(s,\underline \beta)} \| v \|_{N(s,\underline \beta)}. \]
for some constant $C$ depending only on $s$ and $d$ (not on $\underline \beta$ or $\kappa$).
\end{lem}

\begin{proof}
Throughout this proof, the implicit constant in $\lesssim$ is allowed to depend on $d$ and $s$ only. We assume $u,v\in L^2_0$ and conclude by density.

Denote $u_i = P_i u$ and $v_j = P_j v$ so that $u = \sum_{i \in \m N} u_i$, $v = \sum_{j \in \m N} v_j $. 
Due to Lemma \ref{prodHarmo}, we know that $P_{\ell}(u_{i}u_{j}) =0$ unless $|j-i| \le \ell \le i+j$, and so:
\[ P_\ell(uv) = \sum_{i,j} P_\ell(u_i v_j) =  \sum_{\substack{i,j \\ |i-j| \le \ell \le i+j}} P_{\ell} (u_i v_j). \]
We can split this sum depending on $i \le j$ or $j < i$, and using that $P_\ell$ is an $L^2$-orthogonal projection, we bound
\begin{gather*} \| P_\ell(uv) \|_{L^2} \le  \sum_{\substack{i,j \\ |i-j| \le \ell \le i+j}} \| P_{\ell} (u_i v_j) \|_{L^2} \le S_{\ell}(u,v) + S_{\ell}(v,u), \quad \text{where}  \\
S_{\ell}(u,v) = \sum_{\substack{i,j \\ i \le j, \ |\ell-j| \le i}} \| u_i v_j \|_{L^2}.
\end{gather*}
Now, by  Lemma \ref{lmSogge}, 
 \begin{align*}
 \nor{u_{i}}{L^{\infty}(\mathbb{S}^{d-1})} \le C \left<i\right>^{\frac{d}{2}-1} \nor{u_{i}}{L^{2}(\mathbb{S}^{d-1})}.
 \end{align*}
so that if $\ell \le i+j$,
\begin{align*}
\nor{u_{i}v_{j}}{L^{2}(\mathbb{S}^{d-1})} & \le \nor{u_{i}}{L^{\infty}(\mathbb{S}^{d-1})}\nor{v_{j}}{L^{2}(\mathbb{S}^{d-1})}\\
& \lesssim  \frac{\left<i\right>^{\frac{d}{2}-1}}{\beta_i \beta_j} \beta_i \nor{u_{i}}{L^{2}(\mathbb{S}^{d-1})} \beta_j \nor{v_{j}}{L^{2}(\mathbb{S}^{d-1})} \\
& \lesssim \kappa \frac{\left<i\right>^{\frac{d}{2}-1}}{\beta_\ell} \beta_i \nor{u_{i}}{L^{2}(\mathbb{S}^{d-1})} \beta_j \nor{v_{j}}{L^{2}(\mathbb{S}^{d-1})}.
\end{align*}
Also observe that if $|\ell-j| \le i \le j$, then $\ell \le 2j$. Together with the above estimate, we can bound
\begin{align*}
S_{\ell}(u,v) & \lesssim  \kappa \beta_{\ell}^{-1} \sum_{\substack{i \le j \\  |\ell-j| \le i }}  \left<i\right>^{\frac{d}{2}-1-s}  (\left<i\right>^{s} \beta_i \nor{u_{i}}{L^{2}(\mathbb{S}^{d-1})})  \beta_j \nor{v_{j}}{L^{2}(\mathbb{S}^{d-1})} \\
& \lesssim  \kappa \beta_{\ell}^{-1} \sum_{j \ge \ell/2} \beta_j \nor{v_{j}}{L^{2}(\mathbb{S}^{d-1})} \sum_{i,\  |\ell-j| \le i \le j} \left<i\right>^{\frac{d}{2} -1 - s } (\left<i\right>^{s}\beta_i \nor{u_{i}}{L^{2}(\mathbb{S}^{d-1})} ) \\
& \lesssim  \kappa \beta_{\ell}^{-1} \sum_{j \ge \ell/2} \beta_j \nor{v_{j}}{L^{2}(\mathbb{S}^{d-1})} \left( \sum_{|\ell-j| \le i} \left<i\right>^{d-2-2s} \right)^{1/2} \left( \sum_i \left< i \right>^{2s}\beta_i^2 \nor{u_{i}}{L^{2}(\mathbb{S}^{d-1})}^2 \right)^{1/2} \\
& \lesssim   \kappa \beta_{\ell}^{-1} \| u \|_{N(s,\underline \beta)} \sum_{j \ge \ell/2} \left< \ell-j \right>^{\frac{d-1}{2} -s} \beta_j \nor{v_{j}}{L^{2}(\mathbb{S}^{d-1})} .
\end{align*}
We used the Cauchy-Schwarz inequality and the fact that $d-2-2s <-1$. When $j \ge \ell/2$, $\left< \ell \right>^s \lesssim \left< j \right>^s$, so that
\begin{align*}
\left< \ell \right>^s \beta_\ell S_{\ell}(u,v) & \lesssim  \kappa  \| u \|_{N(s,\underline \beta)} \sum_{j} \left< \ell-j \right>^{\frac{d-1}{2} -s} (\left< j \right>^s \beta_j \nor{v_{j}}{L^{2}(\mathbb{S}^{d-1})}).
\end{align*}
We recognize a convolution: as $s > \frac{d}{2} + \frac{1}{2}$, $(\left< j \right>^{\frac{d-1}{2}-s})_j \in \ell^1$ and
\[ \| (\left< j \right>^s \beta_j \nor{v_{j}}{L^{2}(\mathbb{S}^{d-1})})_j \|_{\ell^2} = \| v \|_{N(s,\underline \beta)}, \]
we get
\begin{align*}
\| \left< \ell \right>^s \beta_\ell S_{\ell}(u,v) \|_{\ell^2} \lesssim  \kappa  \| u \|_{N(s,\underline \beta)} \| v \|_{N(s,\underline \beta)}.
\end{align*}
The same equality hold when replacing $S_\ell(u,v)$ by $\| P_\ell(u,v) \|_{L^2}$, and so
\[ \| uv \|_{N(s,\underline \beta)} = \| \left< \ell \right>^s \beta_\ell \| P_\ell(uv) \|_{L^2} \|_{\ell^2} \lesssim \kappa  \| u \|_{N(s,\underline \beta)} \| v \|_{N(s,\underline \beta)}. \qedhere
\]
\end{proof}

\begin{proof}[Proof of Proposition \ref{lmprodanalytique}]

The main observation is the folllowing:

\begin{claim}
Given $c,\ell_0 \in \m R$ and $t\ge 0$, the sequence defined by $\beta_\ell = e^{(c+ \max(\ell,\ell_0)) t}$ is easing with factor $\kappa =  e^{- c t}$.
\end{claim}

\begin{proof}
Let $k \le i+j$, then $\max(k,\ell_0) \le  \max(i,\ell_0) + \max(j,\ell_0) $: indeed, $k \le i+j$ so that $k \le  \max(i,\ell_0) + \max(j,\ell_0)$, and obviously, $\ell_0 \le  \max(i,\ell_0) + \max(j,\ell_0)$.

(Equality holds for $i,j \ge \ell_0$)

Hence, $\beta_k \le e^{(c +  \max(i,\ell_0) + \max(j,\ell_0))t} \le e^{-c t} \beta_i \beta_j$.
\end{proof}

We apply this claim to $\ell_0=0$ and $c = \frac{d-2}{2}$, for which $\| \cdot \|_{N(s,\underline \beta)} = \| \cdot \|_{Y_{s,t}}$.
\end{proof}
Proposition \ref{lmprodanalytique} yields that the space $Y_{s,t}$ is a Banach algebra (up to a multiplication of the norm by a constant). So, we easily get the following corollary.
\begin{cor}
\label{coranalyt}
Let $d \ge 2$ and $s> \frac{d}{2} + \frac{3}{2}$. Let $f$ be an analytic function of $\m C$ of positive radius $\rho>0$ with $f(0)=0$. Then, there exists $\varepsilon>0$ and $C>0$ so that for any $t_0\ge0$ and $u$ function on $(t_0,+\infty)\times \m S^{d-1}$ with  $\nor{(u,\partial_t u)}{ \q Y_{s,t_0}^{t_0}}\le \varepsilon $, then $(f(u),\partial_t (f(u)))\in \q Y_{s,t_0}^{t_0}$. Moreover, if $f$ can be written $f(z)=z^n g(z)$ with $n\in \N^*$ and $g$ analytic, then we have additionally, for any $t \ge t_0$
\[ \nor{f(u)}{ Y_{s,t-t_0}}\le C e^{-\frac{d-2}{2}(n-1)(t-t_0)}\nor{u}{Y_{s,t-t_0}}^n. \]
\end{cor}

\bnp
We write $g(z)=\sum_{i=0}^{+\infty}a_i z^i$ with $|a_i|\leq C \rho^{-i}$ (up to changing $\rho$ by a smaller one). For $t\ge t_0$, We have, by definition, $\nor{u}{Y_{s,t-t_{0}}}\le \varepsilon$ for $t\ge t_0$. By the algebra property of Proposition \ref{lmprodanalytique}, we get for $t\ge t_0$
\begin{align*}
   \nor{f(u)}{ Y_{s,t-t_0}} & \le \sum_{i=0}^{+\infty} C^{n+i-1}|a_i|e^{- \frac{d-2}{2} (t-t_0)(n+i-1)}\nor{u}{Y_{s,t-t_0}}^{n+i} \\
   & \le C^n e^{- \frac{d-2}{2} (t-t_0)(n-1)} \nor{u}{Y_{s,t-t_0}}^n \sum_{i=0}^{+\infty}(C\varepsilon/\rho)^{i}.
\end{align*}
This is convergent if $\varepsilon$ is small enough and gives the last announced result. Writing $\partial_t (f(u))=(\partial_t u) f'(u)$, we get similar result and prove the rest of the Corollary.
\enp
\subsection{The elliptic null condition in dimension 2}
\label{s:null2}
Here we focus on the dimension $d=2$. In that case, the rate in the exponential in \eqref{estimprodYst} is zero, and there is a priori no extra decay on non linear terms. However, if $u$ and $v$ are linear solutions of 
\[ \partial_{tt} u - \opD^2 u =0 \]
and when the product satisfies a special ``null condition'', extra cancellations occur and one derives improved estimates. This is particularly relevant for critical cases, as are conformal equations, which we detail in Section \ref{s:conf2}.

We work in radial coordinates: let $z=(t,\theta) \in \m R\times \R/2\pi \Z\approx \R \times \m S^{1}$ be the running point: observe that
\[ |\partial_\theta =  \opD. \]
We will denote $\xi = (\xi_t,\xi_\theta)$ the coordinates of a vector $\xi \in \m C^2$. 

\begin{definition}
\label{defNull}

Let $A$ be a $\m C$-valued bilinear form on $\m C^2$, which we represent by a $2 \times 2$ matrix
\[ A = \begin{pmatrix}
a_{t t} & a_{t \theta } \\
a_{\theta t} & a_{\theta \theta}
\end{pmatrix} \in \mathscr M_2(\m C),\]
so that for $\xi,\eta \in \m C^2$, $A(\xi,\eta) = \xi^\top A \eta$.

We say that $A$ satisfies the \emph{elliptic null condition} if
\begin{equation} \label{hypoNull}
\forall \xi \in \m C^2, \quad \left( p(\xi) =0 \imp A(\xi,\xi) =0 \right),
\end{equation}
where $p(\xi) := - (\xi_{t}^{2} + \xi_{\theta}^{2})$ is the symbol of the Laplace operator, defined on $\m C^{2}$. 

If $A: \q V \to Bil_{\m R}(\m C^2, \m C)$, $z \mapsto A(z)$ is a map defined on a neigborhood $\q V$ of $0 \in \m R^2$, with values in $\m C$-bilinear forms on $\m C^2$, we say that $A$ satisfies the elliptic null condition \emph{near $0$} if for all $z \in \q V$, $A(z)$ satisfies the elliptic null condition.
\end{definition}

\begin{lem}
\label{lmequiNull}
Let $A \in \mathscr M_2(\m C)$, we denote 
$A^\flat = \begin{pmatrix} 
-1 & i \\
-1 & -i 
\end{pmatrix} A  \begin{pmatrix} 
-1 & -1 \\
i & -i 
\end{pmatrix}$. 

 We have the equivalence:
\begin{enumerate} 
\item  \label{equiv1} $A^\flat$ has null diagonal terms.
\item \label{equiv2} For all $j,k \in \m Z$ such that $jk \ge 0$, $A(\zeta(j),\zeta(k))=0$, where we denoted $\zeta(k) : = (-|k|, ik) \in \m C^2$.
\item  \label{equiv3} $A$ satisfies the elliptic null condition.
\end{enumerate}
\end{lem}

\begin{proof}
Notice that the entries of $A^\flat$ are (for $i,j \in \{1,2 \}$)
\[ (A^\flat)_{i,j} = A(\zeta((-1)^{i-1},\zeta((-1)^{j-1}) \]
By homogeneity, we see that
\begin{equation} \label{eq:A_zeta_homogene}
\forall p,q \in \m R, \quad A(\zeta(p),\zeta(q)) = |p||q| A(\zeta(\sgn p),\zeta(\sgn q)),
\end{equation}
so that \eqref{equiv1} and \eqref{equiv2} are equivalent.

Second, \eqref{equiv3} implies \eqref{equiv1}: it suffices to test with $\zeta(\pm 1)$, for which we have 
\[ p(\zeta(\pm 1)) = -1 - i^2 =0. \]

Last, for the converse, we notice that $p(\xi)=0$ is equivalent to $\xi_{\theta}=\pm i\xi_{t}$. So, in particular, $\xi=(\xi_{t},\xi_{\theta})=(\xi_{t},\pm i\xi_{t})=-\xi_{t}(-1,\mp i)=-\xi_{t} \zeta(\mp 1)$ and 
\[ A(\xi,\xi)=\xi_{t}^{2}A( \zeta(\mp 1),\zeta(\mp 1))=0 \]
by assumption.
\end{proof}
\begin{remark}
\label{rk:nullmat}
    Two important examples of matrices satisfying the null conditions will be the matrices $A=I_2$ and $J = \begin{pmatrix} 
0 & 1 \\
-1 & 0 
\end{pmatrix}$, corresponding to terms of the form $\nabla v\cdot \nabla u$ and $\nabla^{\perp} v\cdot \nabla u=\{v,u\}=\partial_t v\partial_\theta u-\partial_\theta v\partial_tu$ respectively. We compute the corresponding matrices $(I_2)^\flat= \begin{pmatrix}
0& 2 \\
2 & 0
\end{pmatrix}$ and $J^\flat= \begin{pmatrix}
0& 2i \\
-2i & 0
\end{pmatrix}$ that have null diagonal indeed.
\end{remark}
\begin{prop}
\label{propNullgain}
Let $s > 3/2$, there exists $C>0$ so that the following holds. For $u_0, v_0 \in H^s(\m S^1)$,  denote $u (t)$ and $v(t)$  the associated linear solutions to \eqref{eq:D0} with data $(u_0, - \gh D u_0)$ and $(v_0, - \gh D v_0)$ at $t=0$ (with no growing modes);
denote also $\nabla u = (\partial_t u, \partial_\theta u)$ and similarly $\nabla v = (\partial_t u, \partial_\theta u)$.

Let $t \ge 0$ and $A: \m S^1 \to \mathscr M_2(\m C)$ such that $A(\theta)$ satisfies the elliptic null condition of Definition \ref{hypoNull} for all $\theta \in \m S^1$. 

Then, we have
\begin{align} \label{est:null_ell_2d}
\nor{A(\cdot) (\nabla u(t), \nabla v(t)) }{Y_{s-1,t}}\le C e^{-2t} \| A \|_{Y_{s-1,t}}  \|  \partial_\theta u_0 \|_{H^{s-1}(\m S^1)} \|  \partial_\theta v_0 \|_{H^{s-1}(\m S^1)}.
\end{align}
\end{prop}

\begin{proof}
Due to the choice of the initial data for $u(t), v(t)$, they can be decomposed:
\begin{align*}
u(t,\theta) & = \sum_{j\in \Z} \alpha_j e^{ij\theta} e^{-|j|t}, \qquad v(t,\theta) = \sum_{k \in \Z} \beta_k e^{ik \theta} e^{-|k|t}.
\end{align*}
for some complex coefficients $\alpha_j,\beta_k$. Then
\begin{align*}
 \nabla u (t,\theta)& = \sum_{j \in \Z} \alpha_j e^{i j\theta} e^{-|j| t} \zeta(j)
\end{align*}
and similarly for $\nabla v$, so that 
\[ A(\nabla u, \nabla v) (t,\theta) = \sum_{j,k}\alpha_j \beta_k A (\theta)(\zeta(j),\zeta(k)) e^{i(j+k) \theta} e^{-(|j|+|k|)t}.
 \]
Denoting $\q R = \{ (j,k) \in \m Z^2 : jk < 0 \}$, condition \eqref{equiv2} writes that for all $(j,k) \in \m Z^2 \setminus \q R$, $A(\theta) (\zeta(j),\zeta(k)) =0$, so
\begin{align*}
A(\theta)(\nabla u(t,\theta), \nabla v(t,\theta))  & = \sum_{(j,k) \in \q R}\alpha_j \beta_k  |j||k| A(\theta)^\flat_{\sgn(j), \sgn(k)} e^{i(j+k) \theta} e^{-(|j|+|k|)t} \\
& = \sum_{\sigma \in \{ \pm 1 \}^2}  \sum_{(j,k) \in \q R_\sigma} \sum_{\ell \in \m Z} |j| \alpha_j |k| \beta_k  a_{\sigma,\ell} e^{i(j+k+\ell) \theta} e^{-(|j|+|k|+|\ell|)t},
\end{align*}
where $\q R_\sigma = \{ (j,k) \in \q R : (\sgn(j), \sgn(k)) = \sigma \}$ and we decomposed the component $A(\theta)_\sigma^\flat$ of the matrix $A(\theta)^\flat$ (introduced in Lemma \ref{lmequiNull}, recall \eqref{eq:A_zeta_homogene}) in Fourier modes in $\theta$ :
\[ A(\theta)_\sigma^\flat = \sum_{\ell \in \m Z}  a_{\sigma,\ell}  e^{i \ell \theta} e^{-|\ell|t}. \]
The choice of the renormalization factor $e^{-|\ell|t}$ is consistent with the equality 
\begin{align}
\label{norAY}\nor{ A(\cdot)_\sigma^\flat}{Y_{s-1,t}}^2=\sum_{\ell \in \Z}\langle \ell \rangle^{2s-2)} |a_{\sigma,\ell}|^2.
\end{align}
Therefore, for $m \in \m N^*$,
\begin{align*}
\MoveEqLeft \| P_m( A(\cdot)(\nabla u (t), \nabla v (t)) \|_{L^2(\m S^1_\theta)}^2 \\
& \le \sum_{\epsilon \in \{ \pm \}} \bigg| \sum_{\sigma \in \{ \pm 1 \}^2} \sum_{\begin{subarray}{c} j+k+\ell = \epsilon m \\ (j,k) \in \q R_\sigma \end{subarray}}  
|j \alpha_j| |k \beta_k|  |a_{\sigma,\ell}| e^{-(|j|+|k|+|\ell|)t} \bigg|^2.
\end{align*}
(and the corresponding equation for $m=0$, without the $\epsilon$ sum). Now, the key  property that we use is the following statement:

\medskip

\textbf{Claim:} if $(j,k) \in \q R$, then $|j+k| \le |j|+|k|-2$.

Let us prove the claim. Let $(j,k) \in \q R$, then $j,k \ne 0$ and have opposite signs. In the case $|k|\ge |j|$ and $k>0$, $j<0$, we have $|j+k|= |k|-|j| = |j|+|k|-2|j|\le  |j|+|k|-2$. The case $|k|\ge |j|$ and $k<0$, $j>0$ gives the same result and the claim is proved by symmetry.

\medskip

Now that the claim is proved, we can get back to the proof of the Proposition and get
\[ \forall (j,k) \in \q R, \quad  e^{-(|j|+|k|)t} \le e^{-2t} e^{-|j+k| t}. \]
 Hence, if $j+k+\ell = \epsilon m$ and $(j,k) \in \q R$, then $m\le |\ell|+|\epsilon m-\ell|=|\ell|+|j+k|$ and 
\[ e^{-(|j|+|k|+|\ell|)t} \le e^{-2t} e^{- m t}. \]
We obtained the bound
\[ \| P_m( A(\nabla u, \nabla v) (t,\theta)) \|_{L^2(\m S^1_\theta)} \le e^{- 2t} e^{- m t} \sum_{\sigma \in \{ \pm 1 \}^2} \sum_{j+k+\ell = \pm m}  
| j\alpha_j | |k\beta_k|  |a_{\sigma,\ell}|. \]
As $j+k+\ell =\epsilon m$,
\[\left< m\right>^{s-1}= \left< j+k+\ell \right>^{s-1} \lesssim  \left< j \right>^{s-1} +  \left< k \right>^{s-1} +  \left< \ell \right>^{s-1}, \]
and we get for $m \in \m N$,
\begin{align*}
\MoveEqLeft e^{mt}  \left< m \right>^{s-1}  \| P_m( A(\nabla u, \nabla v) (t,\theta)) \|_{L^2(\m S^1_\theta)} \\
& \le e^{- 2t}  \left( \sum_{\sigma \in \{ \pm 1 \}^2} \sum_{j+k+\ell = \pm m}  
\left< j \right>^{s-1}  |j \alpha_j | |k \beta_k|  |a_{\sigma,\ell}|  + \text{symmetric terms} \right). 
\end{align*}
Squaring and summing over $m \in \m N$, we recognize a trilinear convolution: due to Young's inequality, the continuous embedding $\ell^2 * \ell^1 * \ell^1 \to \ell^2$ holds. This gives
\begin{align*}
\| A(\nabla u, \nabla v) (t) \|_{Y_{s-1,t}}^2 & \lesssim e^{-4t} \left( \sum_{j \in \m Z} \left< j \right>^{2s-2}  |j \alpha_j |^2 \right) \left( \sum_{k \in \m Z} |k \beta_k | \right)  \left( \sum_{\ell \in \m Z} |a_{\sigma,\ell}|  \right) \\
& \qquad + \text{symmetric terms}.
\end{align*}
We now recall \eqref{norAY} and that
\begin{gather*}
\| \partial_\theta u_0 \|_{H^{s-1}(\m S^1)}^2  = \sum_{j \in \m Z} |j \alpha_j|^2 \left< j \right>^{2s-2}, \quad \|  \partial_\theta v_0 \|_{{H^{s-1}(\m S^1)}}^2  = \sum_{k \in \m Z} |k \beta_k|^2 \left< k \right>^{2s-2}.
\end{gather*}
By Cauchy-Schwarz inequality, we infer
\[  \sum_{k \in \m Z} |k \alpha_k | \lesssim \|  \partial_\theta u_0 \|_{H^{s-1}(\m S^1)} \left( \sum_{k \in \m Z} \left< k \right>^{2-2s} \right)^{1/2} \lesssim  \|  \partial_\theta u_0 \|_{H^{s-1}(\m S^1)}, \]
because $2-2s < -1$. The same gives
\[ \sum_{k \in \m Z} |k \beta_k | \lesssim \|  \partial_\theta v_0 \|_{H^{s-1}(\m S^1)}  \quad \text{and} \quad 
\sum_{\ell \in \m Z} |a_{\sigma,\ell}| \lesssim \| A_\sigma^\flat \|_{Y_{s-1,t}}. \]
This allows us to conclude to
\[ \| A(\nabla u, \nabla v) (t) \|_{Y_{s-1,t}} \lesssim e^{-2t} \|  \partial_\theta u_0 \|_{H^{s-1}(\m S^1)} \|  \partial_\theta v_0 \|_{H^{s-1}(\m S^1)} \| A^\flat \|_{Y_{s-1,t}}, \]
and from there, to \eqref{est:null_ell_2d}.
\end{proof}

\subsection{Polynomials in \texorpdfstring{$Y_{s,t}$}{Ys,t} spaces}

Product with monomials appear naturally when performing the conformal transform. We first recall the following classical Lemma that will be used several times in the article:
\begin{lem}[{\cite[Lemma 4.50, Section 4.E.3]{GallotHulinLaf}}]
\label{lmdecompHarmo}
Let $\widetilde{P}_{k}$ denote the space of homogeneous polynomials of degree $k$ and $\widetilde{H}_{k}$ the space of harmonic polynomials of degree $k$ on $\R^{d}$. Denote $P_{k}$ and $H_{k}$ the spaces obtained by restricting these polynomials to $\mathbb{S}^{d-1}$. Then, we have 
\begin{align*}
\widetilde{P}_{k} = \bigoplus_{i=0}^{\lceil k/2\rceil}r^{2i}\widetilde{H}_{k-2i} \quad \text{and} \quad
P_{k} = \bigoplus_{i=0}^{\lceil k/2\rceil}H_{k-2i}.
\end{align*}
\end{lem}

\begin{lem}
\label{lm:polyYst}
Let $s\ge 0$. Then, there exists $C=C(d,s)>0$ so that for any $\alpha \in \m N^d$ multi index and $t\ge 0$, we have
\begin{equation} \label{est:ya}
\| y^\alpha \|_{Y_{s,t}} \le C e^{(|\alpha|+ \frac{d-2}{2}) t} \left< \alpha \right>^{s+1}. 
\end{equation}
In particular, 
$\nor{x^{\alpha}}{\q Z_{s}^{0}}\le C\left< \alpha \right>^{s+2}$.
\end{lem}
Here, we have written $y^{\alpha}$ for the restriction to $\m S^{d-1}$ of the function defined on $\R^d$ by $x\mapsto x^{\alpha}$ while we have written $x^{\alpha}$ for the function defined on $B(0,1)\subset \R^d$.
Combining \eqref{est:ya} and \eqref{estimprodYst}, we get that for $s > \frac{d}{2}+\frac{1}{2}$ and any $\alpha \in \m N^d$, $u \in Y_{s,t}$,
\begin{equation} \label{est:prod_ya}
\forall t \ge 0, \quad \| y^\alpha u \|_{Y_{s,t}} \le C  \left< \alpha \right>^{s+1} e^{|\alpha| t} \| u \|_{Y_{s,t}}.
\end{equation}

\begin{proof}
$x^\alpha = \prod_{i=1}^d x_i^{\alpha_i}$ is a homogeneous polynomial, it decomposes into
\[ x^\alpha = \sum_{j \le |\alpha|/2} |x|^{2j} h_{|\alpha| -2 j}, \]
where $h_{|\alpha| -2 j}$ is a harmonic polynomial of degree $|\alpha|-2j$. When restricted to the sphere, we get 
\[ y^\alpha = \sum_{j \le |\alpha|/2} h_{|\alpha| -2 j}. \] 
Now $h_{|\alpha| -2 j}$ is an eigenfunction of $\gh D$ with eigenvalue $|\alpha|-2j+ \frac{d-2}{2}$, so that
\[ e^{t \gh D} y^\alpha =  \sum_{j \le |\alpha|/2}  e^{t \gh D} h_{|\alpha| -2 j} = \sum_{j \le |\alpha|/2} e^{(|\alpha|-2j+ \frac{d-2}{2} )t } h_{|\alpha| -2 j}. \]
Also, the decomposition is orthogonal so that
\[ \|  y^\alpha \|_{H^s(\m S^{d-1})}^2 = \sum_{j \le |\alpha|/2} \| h_{|\alpha| -2 j}  \|_{H^s(\m S^{d-1})}^2, \]
and therefore
\begin{align*}
\| y^\alpha \|_{Y_{s,t}}^2 & = \| e^{t \gh D}  y^\alpha \|_{H^s(\m S^{d-1})}^2 = \sum_j e^{2(|\alpha|-2j+ \frac{d-2}{2})t } \| h_{|\alpha| -2 j} \|_{H^s(\m S^{d-1})}^2 \\
& \le e^{2(|\alpha|+ \frac{d-2}{2} ) t} \sum_j  \| h_{|\alpha| -2 j} \|_{H^s(\m S^{d-1})}^2 = e^{2(|\alpha|+ \frac{d-2}{2} ) t} \|  y^\alpha \|_{H^s(\m S^{d-1})}^2.
\end{align*}
To conclude, it suffices to finally notice that
\begin{align*}
\|  y^\alpha \|_{H^s(\m S^{d-1})} & \lesssim  \|  y^\alpha \|_{\mc C^{\lceil s \rceil} (\m S^{d-1})} 
\le \nor{x^{\alpha}}{\mc C^{\lceil s \rceil}(B_{\R^d}(0,1))} =    \sum_{|\beta|\le \lceil s \rceil }\nor{\partial^{\beta}x^{\alpha}}{L^{\infty}(B_{\R^d}(0,1))} \\
& \lesssim \left< \alpha \right>^{ \lceil s \rceil} \lesssim \left< \alpha \right>^{s+1}. 
\end{align*} 
Concerning the second part, since for $u(x)=x^{\alpha}$, $u(r y)=r^{|\alpha|} y^{\alpha}$, while for $w=\frac{d-2}{2} u +r\frac{\partial u}{\partial r}$, $w(ry)=r^{|\alpha|} \left(\frac{d-2}{2}+|\alpha|\right) y^{\alpha}$  so we estimate
\begin{align}
\nor{x^{\alpha}}{\q Z_{s}^{0}} & =\sup_{0<1\le 1} r^{\frac{d-2}{2}+|\alpha|}\| y^{\alpha} \|_{Y_{s,-\log(r)}}+\left(\frac{d-2}{2}+|\alpha|\right) \sup_{0<1\le 1}  r^{\frac{d-2}{2}+|\alpha|} \| y^{\alpha} \|_{Y_{s,-\log(r)}}\\
 & \le C \left< \alpha\right> \sup_{0<1\le 1}r^{\frac{d-2}{2}+|\alpha|}r^{-\frac{d-2}{2}-|\alpha|} \le C \left< \alpha \right>^{s+2}. \qedhere
 \end{align}
\end{proof}
\subsection{Embeddings in usual spaces }
In this section, we first describe the notations concerning usual spaces and then describe their link to our spaces $\q Z$.

In all what follows, $\dot{H}^1(\R^d)$ denotes the completion of $\mc C^{\infty}_c(\R^d)$ for the norm $ \nor{\nabla u}{L^2(\R^d)}$. For $d\ge 3$, this is isomorphic to the functions $u\in L^{2^{*}}$ with $\nabla u\in L^2$ with $2^{*}=\frac{2d}{d-2}$ the critical Sobolev exponent for the Sobolev embedding $L^{2^{*}}\subset \dot{H}^{1}$. We will sometimes use a localized version
\begin{align}
\label{Soblocal}
  \nor{u}{L^{2^{*}}(\{|x|\ge 1\})}&\le C_{d} \nor{\nabla u}{L^{2}(|x|\ge 1)}
\end{align}
 valid for any $u\in \dot{H}^1(\{|x|\ge 1\})$. Here, we denoted $
\dot{H}^1(\{|x|\ge 1\})$ the restriction of such functions to $\{|x|\ge 1\}$. $
\dot{H}_0^1(\{|x|\ge 1\}$) denotes the completion of $\mc C^{\infty}_c(\{|x|> 1\})$ for the norm $ \nor{\nabla u}{L^2(\{|x|> 1\})}$. This is isomorphic to the functions in $
\dot{H}^1(\{|x|\ge 1\})$ with trace zero on $\m S^{d-1}$. We will avoid the use of $\dot{H}^1$ in unbounded domains of dimension $2$ since the definition contains subtleties that are not necessary here.

We have the following embedding of the $\q Z$ spaces into the usual homogeneous Sobolev space $\dot H^s$.
\begin{lem}
\label{lminjectZ}
1) Let $d\ge 3$, $s \ge 1$ and $u\in \q Z_{s}^{\infty}$. Then  $ u\in \dot{H}^1(|x|\ge 1)$. If furthermore $s>(d-1)/2$, then $ u\in L^p(|x|\ge 1)$ for any $p>\frac{d}{d-2}$.

2) Let $d=2$, $s \ge 1$ and $u \in \q Z_{s}^{\infty}$ such that for some $\nu >0$,
\[ \sup_{r\ge 1}  r^{\nu}\left(\nor{u(r\cdot)}{Z_{s,r}^{\infty}}+ \nor{r\frac{\partial u}{\partial r} (r\cdot)}{Z_{s-1,r}^{\infty}}\right)<+\infty. \]
Then, $\nabla u\in L^{2}(|x|\ge 1)$.
\end{lem}

\bnp
We decompose 
\[  u = \sum_{\ell,m}f_{\ell,m}(|x|) \phi_{\ell,m}\left(\frac{x}{|x|}\right). \]
As 
\[ \nabla (f(|x|) g(x/|x|)) = f'(|x|) g \left( \frac{x}{|x|} \right) \frac{x}{|x|} + \frac{f(|x|)}{|x|} \nabla_{\m S^{d-1}} g \left( \frac{x}{|x|} \right), \] 
we infer, using $y\cdot \nabla_{\m S^{d-1}} g(y)=0$ for any $g: \m S^{d-1} \to \m R$ and $y \in \m S^{d-1}$, the orthogonality of $(\phi_{\ell,m})_{\ell,m}$ and $\nor{\nabla_{\m S^{d-1}}\phi_{\ell,m}}{L^2(\m S^{d-1})}^2=\ell(\ell +d-2)$,
\begin{align*}
\| \nabla u \|_{ L^2(|x|\ge 1)}^2 & = C_d \sum_{\ell,m} \int_1^{+\infty} \left( |f_{\ell,m}'(r)|^2 + \frac{|f_{\ell,m}(r)|^2}{r^2} \ell(\ell +d-2) \right) r^{d-1} dr \\
& \le C_d \int_1^{+\infty}  \sum_{\ell,m} \left( |r f_{\ell,m}'(r)|^2 + \langle \ell \rangle^2 |f_{\ell,m}(r)|^2 \right) r^{d-3} dr.
\end{align*}
On the other hand, for $r \ge 1$,
\[ \nor{u(r\cdot)}{Z_{s,r}^{\infty}}^2 = r^{d-2} \sum_{\ell,m} |f_{\ell,m}(r)|^2 \langle \ell \rangle^{2s} r^{2\ell + d-2} \ge r^{2d-4} \sum_{\ell,m} |f_{\ell,m}(r)|^2  \langle \ell \rangle^{2s}. \]
Similarly, for $s\ge 1$,
\begin{align*} 
\left\| \left( \frac{d-2}{2}u + r \frac{\partial u}{\partial r} u \right)(r \cdot) \right\|_{Z_{s-1,r}^\infty}^2 & \ge r^{2d-4} \sum_{\ell,m} \left| \frac{d-2}{2} f_{\ell,m} (r) + rf_{\ell,m}'(r) \right|^2 \langle \ell \rangle^{2(s-1)} \\
& \ge  r^{2d-4} \sum_{\ell,m} \left| \frac{d-2}{2} f_{\ell,m} (r) + r f_{\ell,m}'(r) \right|^2.
\end{align*}
Finally notice that
\[ |rf_{\ell,m}'(r)|^2 \le \left| \frac{d-2}{2} f_{\ell,m} (r) + rf_{\ell,m}'(r) \right|^2 + \left(\frac{d-2}{2} \right)^2 |f_{\ell,m} (r) |^2. \]

From these computations, if $d \ge 3$, then
\begin{align*}
\| \nabla u \|_{ L^2(|x|\ge 1)}^2 & \lesssim \int_1^{+\infty} \left( \| u(r \cdot) \|_{Z_{s,r}^{\infty}}^2  + \left\| \left( \frac{d-2}{2}u + r \frac{\partial u}{\partial r} u \right)(r \cdot) \right\|_{Z_{s-1,r}^{\infty}}^2 \right) r^{d-3- (2d-4)} dr \\
& \lesssim \| u \|_{\q Z_s^\infty}^2 \int_1^{\infty} r^{1-d} dr \lesssim  \| u \|_{\q Z_s^\infty}^2.
\end{align*}
Moreover, for $2^* = \ds \frac{d+2}{d-2}$ and $s\ge 1$ and $r\ge 1$, we have by Sobolev embedding $\nor{g}{L^{2^*}(\m S^{d-1})}\lesssim \nor{g}{H^1(\m S^{d-1})}\lesssim  r^{-(d-2)}\nor{g}{Z_{s,r}^{\infty}}$ where we have used $\opD\ge \frac{d-2}{2}$ in the sense of operators of $H^s(\m S^{d-1})$. In particular, 
\begin{align*}
\| u \|_{ L^{2^*}(|x|\ge 1)}^{2^*} & \lesssim \int_1^{+\infty}r^{d-1}\| u(r\cdot) \|_{ L^{2^*}(\m S^{d-1})}^{2^*}dr\lesssim \int_1^{+\infty} r^{d-1-2^{*}(d-2)}\| u(r \cdot) \|_{Z_{s,r}^{\infty}}^{2^*} dr \\
& \lesssim \| u \|_{\q Z_s^\infty}^{2^{*}} \int_1^{\infty} r^{-3} dr \lesssim  \| u \|_{\q Z_s^\infty}^{2^{*}},
\end{align*}
so that $ u\in \dot{H}^1(|x|\ge 1)$. If $s>(d-1)/2$, we similarly have 
\[ \nor{g}{L^{p}(\m S^{d-1})}\lesssim \nor{g}{L^\infty(\m S^{d-1})} \lesssim \nor{g}{H^s(\m S^{d-1})} \lesssim r^{-(d-2)}\nor{g}{Z_{s,r}^{\infty}}, \]
and the same argument gives for any $p> \frac{d}{d-2}$,
\begin{align*}
\| u \|_{ L^{p}(|x|\ge 1)}^{p} & \lesssim \int_1^{+\infty} r^{d-1-p(d-2)}\| u(r \cdot) \|_{Z_{s,r}^{\infty}}^{p} dr \lesssim_p  \| u \|_{\q Z_s^\infty}^p.
\end{align*}

If $d=2$, under the extra assumption we bound similarly
\[ \| u \|_{\dot H^1(|x|\ge 1)}^2  \lesssim \sup_{r\ge 1}  r^{-\nu}\left(\nor{u(r\cdot)}{Z_{s,r}^{\infty}}+ \nor{r\frac{\partial u}{\partial r} (r\cdot)}{Z_{s-1,r}^{\infty}}\right) \int_1^{+\infty} r^{-\nu -1} dr < +\infty. \qedhere \]
\enp

\begin{lem}
\label{lminjectZ0}
1) Let $u \in \q Z_s^0$.
If $s> \frac{d}{2} - \frac{1}{2} $, then $u \in L^\infty(B(0,1))$. If $s> \frac{d}{2}+ \frac{1}{2}$, then $u \in W^{1,\infty}_{\loc}(B(0,1) \setminus \{ 0 \})$ and more precisely, provided the right hand side is finite, there hold
\begin{align} \label{est:Z_W1infty}
\nor{u}{W^{1,\infty}(B(0,1))} \lesssim \sup_{0 < r \le 1} \left( \| u(r \cdot) \|_{Z_{s,r}^0} + \frac{1}{r} \left\| r\frac{\partial u}{\partial r} (r \cdot) \right\|_{Z_{s-1,r}^0} \right).
\end{align}

2) Let $u \in \q Z_s^\infty$.
If $s> \frac{d}{2} - \frac{1}{2} $, then $u \in L^\infty(|x| \ge 1)$ with the decay
\begin{align}
|u(x)|\lesssim |x|^{-(d-2)}\nor{u}{\q Z_s^\infty},\quad 
\nor{P_{\ell}(u(r\cdot))}{L^{\infty}(\m S^{d-1})}\lesssim_{\ell} r^{-(d-2-\ell)}\nor{u}{\q Z_s^\infty}.
\end{align}If $s> \frac{d}{2}+ \frac{1}{2}$, then $u \in W^{1,\infty}(|x| \ge 1)$ and more precisely, there hold
\begin{align} \label{est:Zinfty_W1infty}
\nor{u}{W^{1,\infty}(|x| \ge 1)} \lesssim \sup_{r \ge 1} \left( \frac{1}{r^d} \| u(r \cdot) \|_{Z_{s,r}^\infty} + \frac{1}{r^{d-1}} \left\| r\frac{\partial u}{\partial r} (r \cdot) \right\|_{Z_{s-1,r}^\infty} \right) \lesssim \| u \|_{\q Z^\infty_s}.
\end{align}
\end{lem}

Functions in $\q Z_{s}^{0}$ are actually not well defined at $0$ as there can be oscillations: $\sin(\ln |x|) \in \bigcap_{s \ge 0}  \q Z_{s}^{0}$, but is not continuous at $0$. So, the above Lemma should be understood as the existence of an extension to $B(0,r_{0})$ with the expected properties.

For $\q Z^\infty_s$, there is actually a $1/r^{d-1}$ gain, which we will not really exploit.

To study smoothness issues of function with an emphasis  on spherical regularity, it is convenient to define the following differential operator: for $u$ defined on $B(0,R)$ or on its complement in $\m R^d$
\begin{equation} \label{def:Lambda u} \Lambda u(x) : = \left( \nabla_{\m S^{d-1}} u_{|x|} \right) \left( \frac{x}{|x|} \right) \quad \text{where } u_r \text{ is defined on } \m S^{d-1} \text{ by } u_r(y)= u(ry).
\end{equation}
In particular, it allows to express
\begin{equation} \label{eq:grad_lambda}
\nabla u(x) = \frac{\partial u }{\partial r}(x) \frac{x}{|x|} + \frac{1}{|x|} \Lambda u(x).
\end{equation}

\bnp
a) First let us observe that if $f \in Z^0_s$, then for $r \le 1$,
\begin{multline} 
\| f \|_{Z^0_{s,r}}^2 = \sum_{\ell=0}^{+\infty} r^{-2\ell} \langle \ell \rangle^{2s} \| P_\ell f \|_{L^2(\m S^{d-1})}^2 \\ \ge \sum_{\ell=0}^{+\infty}  \langle \ell \rangle^{2s} \| P_\ell f \|_{L^2(\m S^{d-1})}^2 = \| f \|_{H^s(\m S^{d-1})}^2. \label{est:Z_Hs}
\end{multline}
Similarly, if $g \in Z^\infty_s$, then for $r \ge 1$,
\begin{equation} \label{est:Zi_Hs}
\| g \|_{Z^{\infty}_{s,r}}^2 = \sum_{\ell=0}^{+\infty} r^{d-2} r^{2\ell+d-2} \langle \ell \rangle^{2s} \| P_\ell g \|_{L^2(\m S^{d-1})}^2 \ge r^{2(d-2)}  \| g \|_{H^s(\m S^{d-1})}^2.
\end{equation}
b) Let $s>\frac{d}{2} - \frac{1}{2}$ and $u \in \q Z^0_s$. The Sobolev embedding $L^{\infty}(\m S^{d-1})\subset H^{s}(\m S^{d-1})$ writes, for some $C$ independent of $r >0$,
\begin{align*}
\| u(r \cdot) \|_{L^\infty(\m S^{d-1})} \le C \| u(r \cdot) \|_{H^s(\m S^{d-1})}.
\end{align*}
Hence taking the supremum in $r<1$, we get
\[ \| u \|_{L^{\infty}(B(0,1))} = \sup_{0 < r \le 1} \| u(r \cdot) \|_{L^\infty(\m S^{d-1})} \le C  \| u \|_{\q Z^0_{s}}. \]
(We used \eqref{est:Z_Hs} on the last inequality).

d) Let now $u \in \q Z^\infty_s$ with $s >\frac{d}{2} - \frac{1}{2}$. As in c), 
\[ \| u \|_{L^{\infty}(|x| \ge 1)} \le \sup_{ r \ge 1} r^{d-2}\| u(r \cdot) \|_{L^\infty(\m S^{d-1})} \le C  \| u \|_{\q Z^\infty_{s}}. \]
(We used \eqref{est:Zi_Hs} on the last inequality). This gives the first part of 2). For the second part, using \eqref{est:sogge_sphere},
\begin{align*}
\MoveEqLeft r^{2(\ell+d-2)} \nor{P_{\ell}(u(r\cdot))}{L^{\infty}(\m S^{d-1})}^2 \le C \left< \ell \right>^{d-2}
 r^{2(\ell+d-2)} \nor{P_{\ell}(u(r\cdot))}{L^{2}(\m S^{d-1})}^2 \\
 &\le C\langle \ell \rangle^{d-2-2s} r^{d-2} \sum_{ m=0}^{+\infty} r^{2(m+\frac{d-2}{2})} \langle m \rangle^{2s} \| P_m (u(r\cdot)) \|_{L^2(\m S^{d-1})}^2\\
 & \le C\langle \ell \rangle^{d-2-2s}\| u(r\cdot) \|_{Z^{\infty}_{s,r}}^2 \le C\langle \ell \rangle^{d-2-2s} \| u \|_{\q Z^\infty_{s}}^2.
\end{align*}

e) 
Now assume that $s > \frac{d}{2}+ \frac{1}{2}$ and let $f$ be defined on $\m S^{d-1}$. We have $\nabla_{\m S^{d-1}} f=\nabla_{\m S^{d-1}} P_0^{\perp }f$ where $P_0^{\perp }= \text{Id}-P_0$ is the projection orthogonal to the constants on $\m S^{d-1}$. Then, the Sobolev embedding $W^{1,\infty}(\m S^{d-1})\subset H^{s}(\m S^{d-1})$, writes
\begin{align*}
\nor{ \nabla_{\m S^{d-1}} f}{L^\infty(\m S^{d-1})} & \lesssim  \nor{ P_0^{\perp}f }{H^s(\m S^{d-1})}\lesssim \left( \sum_{\ell =1}^{+\infty} \langle \ell \rangle^{2s} \| P_{\ell} f \|_{L^2(\m S^{d-1})}^2 \right)^{1/2}.
\end{align*}

f) Let $s \ge \frac{d}{2}+ \frac{1}{2}$ and $u \in \q Z^0_s$. We recall \eqref{eq:grad_lambda}:
\[ \nabla u(x)  =   \frac{\partial u}{\partial r} \frac{x}{|x|}  + \frac{1}{|x|} \Lambda u (x). \]
In view of the first equality in \eqref{est:Z_Hs}, for $0 < r \le 1$
\[ \| u (r \cdot) \|_{Z_{s,r}^0}^2 \ge r^{-2}\sum_{\ell=1}^{+\infty} \langle \ell \rangle^{2s}  \| P_\ell (u(r \cdot)) \|_{L^2(\m S^{d-1})}^2, \]
and so, using e),
\[ \| \Lambda u \|_{L^\infty(r \m S^{d-1})} = \| \nabla_{\m S^{d-1}} u_r \|_{L^\infty(\m S^{d-1})} \lesssim r \| u (r \cdot) \|_{Z_{s,r}^0}. \]
Hence
\begin{align*}
\| \nabla u \|_{L^\infty(B(0,1))} & \le \sup_{0 < r \le 1} \left( \left\| \frac{\partial u}{\partial r} \right\|_{L^\infty(r \m S^{d-1})} + \frac{1}{r} \| \Lambda u \|_{L^\infty(r \m S^{d-1})} \right) \\
& \lesssim \sup_{0 < r \le 1} \left( \| u (r \cdot) \|_{Z_{s,r}^0} + \left\| \frac{\partial u}{\partial r} (r \cdot) \right\|_{Z_{s-1,r}^0} \right).
\end{align*}
The statement regarding $W^{1,\infty}_{\loc}(B(0,1) \setminus \{ 0 \})$ is similar, working on $B(0,1) \setminus B(0,r_0)$ for any $r_0 >0$. This gives 1).

g) Similarly let $u \in \q Z^\infty_s$ with $s \ge \frac{d}{2}+ \frac{1}{2}$. Then for $r \ge 1$,
\[ \| u (r \cdot) \|_{Z_{s,r}^{\infty}}^2 = \sum_{\ell=0}^{+\infty} r^{d-2} r^{2\ell+d-2} \langle \ell \rangle^{2s} \| P_\ell f \|_{L^2(\m S^{d-1})}^2 \ge r^{2(d-1)}\sum_{\ell=1}^{+\infty} \langle \ell \rangle^{2s}  \| P_\ell (u(r \cdot)) \|_{L^2(\m S^{d-1})}^2, \]
and so from e)
\[ \| \Lambda u \|_{L^\infty(r \m S^{d-1})} = \| (\nabla_{\m S^{d-1}} u)(r \cdot) \|_{L^\infty(\m S^{d-1})} \lesssim r^{-d+1} \| u (r \cdot) \|_{Z_{s,r}^\infty}. \]
Hence
\begin{align*}
\| \nabla u \|_{L^\infty(|x| \ge 1)} & \le \sup_{r \ge 1} \left( \left\| \frac{\partial u}{\partial r}  \right\|_{L^\infty(r \m S^{d-1})} + \frac{1}{r} \| \Lambda u \|_{L^\infty(r \m S^{d-1})} \right) \\
& \lesssim \sup_{r \ge 1} \left( \frac{1}{r^d} \| u (r \cdot) \|_{Z_{s,r}^\infty} + \frac{1}{r^{d-2}} \left\| \frac{\partial u}{\partial r} (r \cdot) \right\|_{Z_{s-1,r}^\infty} \right) \lesssim \| u \|_{\q Z^\infty_s}. \qedhere
\end{align*}
\enp

\begin{lem}
\label{lmisom0infty}
Assume $d=2$, and define $\widetilde{u}(x)=u\left(\ds\frac{x}{|x|^{2}}\right)$. Then $\nor{\widetilde{u}}{\q Z_{s}^{\infty}}=\nor{u}{\q Z_{s}^{0}}$.
\end{lem}
\bnp
By definition, for $r >0$, we have the equality for functions defined on the sphere $\m S^{d-1}$
\[ \widetilde{u}(r \cdot) =  u(  \cdot /r). \]
Hence, as $d-2=0$,
\[ \| \widetilde{u} (r \cdot) \|_{Z^\infty_{s,r}} = \| u ( \cdot/r) \|_{Z^\infty_{s,r}} = \| r^{\gh D} u( \cdot/r) \|_{H^s(\m S^{d-1})} = \| u( \cdot /r ) \|_{Z^0_{s,1/r}}. \]
Similarly, as
\begin{align*} \left(r \frac{\partial \widetilde u}{\partial r} \right) (x) & = |x| \nabla \widetilde u(x) \cdot \frac{x}{|x|} =  \nabla u \left( \frac{x}{|x|^2} \right) \left( \frac{1}{|x|^2} - \frac{2 x \cdot x}{|x|^4} \right) \cdot x \\
& = -\frac{1}{|x|^2} \nabla u \left( \frac{x}{|x|^2} \right)  \cdot x = - \frac{1}{|x|}   \nabla u \left( \frac{x}{|x|^2} \right)  \cdot \frac{x/|x|^2}{1/|x|} = - \left(r \frac{\partial  u}{\partial r} \right) \left( \frac{x}{|x|^2} \right).
\end{align*}
As before, we infer that 
\[ \left\| \left(r \frac{\partial \widetilde u}{\partial r} \right)(r \cdot) \right\|_{Z^\infty_{s-1,r}} = \left\| \left(r \frac{\partial u}{\partial r} \right)(\cdot/r) \right\|_{Z^\infty_{s-1,1/r}}, \]
and the conclusion follows from taking the supremum in $r \ge 1$.
\enp

  \subsection{Action of some operators}  
\label{s:actionder}
  In the case of nonlinearities with derivatives, we will need to understand the effect of several operators on the spaces $Y_{s,t}$ and $\q Y_{s,t}$. One of the problems will come from the fact that some functions are defined on the manifold $\m S^{d-1}$ and the gradient is therefore in $T \m S^{d-1}$, while we will need to consider power series. 
  
  We will see $\m S^{d-1}$ as embedded in $\R^d$ so that we can consider $T_y \m S^{d-1}\subset\R^{d}$ for $y\in \m S^{d-1}$. For any $i \in \llbracket 1, d \rrbracket$, one natural operator that we will use, is the following operator, defined for $v$ function on $\m S^{d-1}$, 
  \begin{align}
  \label{defDi}
D_{i} v=e_i\cdot \nabla_{\m S^{d-1}}v.
  \end{align}
  where $\cdot$ is the usual scalar product in $\R^d$ and $(e_i)_{i=1,\dots,d}$ is the canonical basis of $\R^d$. $D_i$ comes naturally when we want to consider the operator $\partial_i $ on $\R^d$, written in polar coordinates. It turns out that for some nonlinearities that have some structure, we will need to decompose $D_i$ with a ``main order term'' $-y_i\left(\opD  - \frac{d-2}{2} \right)$. That is why, or any $i \in \llbracket 1, d \rrbracket$, we define the operator 
    \begin{align}
  \label{defRi}
\opR_{i} v=D_i v+y_i\left(\opD  - \frac{d-2}{2} \right) v.
  \end{align}

\begin{lem}
\label{lmderiv}
Let $s \in \R$.  There exists $C_{s,d}$ so that for any $u\in Y_{s,t}$ there hold, for all $i \in \llbracket 1, d \rrbracket$,
\begin{align} \label{est:Ytt0}
\forall t \ge 0, \quad \nor{\opD u}{Y_{s-1,t}}+ e^{-t}\nor{D_{i}u}{Y_{s-1,t}} +e^{t}\nor{\opR_i u}{Y_{s-1,t}}\le C_{s,d}\nor{ u}{ Y_{s,t}}.
\end{align}
Finally, if $d=2$ and $\m S^{1}\approx \R/2\pi\Z$ is parameterized by $\theta$, then we have 
\begin{align}
\label{actdtheta}
\nor{\partial_{\theta}u}{Y_{s-1,t}}\le C_{s}\nor{ u}{Y_{s,t}}.
\end{align}
\end{lem}

We emphasize, in \eqref{est:Ytt0}, the loss for $D_i$ and the gain for $\opR_i$ of an exponential factor $e^t$.

\bnp
The part about $\opD u$ is direct from the definitions. Decompose $u=\sum_{k\in \N}P_k u$, so that 
\[ P_\ell D_{i}u=\sum_{k\in \N}P_\ell D_{i}P_k u. \]
If $k=0$, $D_{i}P_k u=0$. Otherwise, $P_k u$ is the restriction to $\mathbb{S}^{d-1}$ of a homogeneous harmonic polynomial in $\R^d$ of degree $k \ge 1$, say $H_k$. We have for $x$ in $\R^d$,
\begin{align*}
\frac{\partial H_k}{\partial x_{i}} (x)=e_i\cdot \nabla H_k =\frac{1}{|x|} e_i \cdot \Lambda H_k(x) + e_i\cdot\frac{x}{|x|}\frac{\partial H_k}{\partial r} (x).
\end{align*}

When restricted to $\mathbb{S}^{d-1}$, we get by homogeneity of $H_k$, for all $y \in \m S^{d-1}$,
\begin{align*}
\frac{\partial H_k}{\partial x_{i}} (y) = (D_i H_k|_{\m S^{d-1}})(y) + k y_i H_k(y)= (D_i P_k u)(y) + y_i\left(\opD  - \frac{d-2}{2} \right)P_k u(y).
\end{align*}
Now $\widetilde{H}_{k-1} := \frac{\partial H_k}{\partial x_{i}}$ is a harmonic polynomial of degree $k-1$, so that its restriction to $\m S^{d-1}$ is an eigenfunction of $\Delta_{\m S^{d-1}}$. On $\m S^{d-1}$, using the decomposition of $u$, we have 
\[  D_{i}u (y) =-y_i \left(\opD  - \frac{d-2}{2} \right) u(y)+\sum_{k\in \N^*}\widetilde{H}_{k-1}(y) =: -y_i \left(\opD  - \frac{d-2}{2} \right) u(y) + \opR_i u (y) . \]
 $H_k$ and $\widetilde{H}_k$ are homogeneous harmonic polynomials of degree $k$. Multiply the equation $\Delta \widetilde{H}_k=0$ by $\widetilde{H}_k$ and integrate by parts on $B(0,1)$, we get 
 \[ \int_{\m S^{d-1}} \partial_{\nu}\widetilde{H}_k\widetilde{H}_k=\int_{B(0,1)} |\nabla \widetilde{H}_k|^2. \]
 Also, by homogeneity, $ \partial_{\nu}\widetilde{H}_k=k \widetilde{H}_k $. So, we obtain for any $k \in \N^*$, 
 \[ \| \widetilde{H}_{k}|_{\m S^{d-1}} \|_{L^2(\m S^{d-1})}^2=\frac{1}{k} \| \nabla \widetilde{H}_{k} \|_{L^2(B(0,1))}^2. \]
 Since $\widetilde{H}_{k}=\frac{\partial H_{k+1}}{\partial x_{i}} $, by elliptic regularity and uniqueness for the Dirichlet boundary value problem $\Delta H_{k+1}=0$ on $B(0,1)$ (see for instance \cite[Chapter 2, Theorem 8.3]{LionsMagenes1}), we have 
\begin{align*}
\| \nabla \widetilde{H}_{k} \|_{L^2(B(0,1))}^2 &\le C \| H_{k+1} \|_{H^2(B(0,1))}^2\le C \| H_{k+1}|_{\m S^{d-1}} \|_{H^{3/2}(\m S^{d-1}))}^2 \\
& \le C\langle k \rangle^{3} \| H_{k+1}|_{\m S^{d-1}} \|_{L^{2}(\m S^{d-1}))}^2.     
\end{align*} 
Hence, uniformly in $k\in \N$,  
\[ \| \widetilde{H}_{k}|_{\m S^{d-1}} \|_{L^2(\m S^{d-1})}\le C\langle k \rangle\| H_{k+1}|_{\m S^{d-1}}  \|_{L^{2}(\m S^{d-1}))}. \]
We have $P_\ell \widetilde{H}_{k-1}=0$ if $\ell\neq k-1$ so that, for $\ell \in \m N$, $P_{\ell}\opR_i u= \widetilde{H}_{\ell}$. Therefore, we can estimate, uniformly in $\ell\in \N$,
\begin{align*}
\| P_\ell \opR_i u \|_{L^2(\m S^{d-1})}&\le C  \langle \ell \rangle \|  (H_{\ell+1}  )_{\left|\m S^{d-1}\right.}\|_{L^{2}(\m S^{d-1})}=\langle \ell \rangle \|  P_{\ell+1} u \|_{L^{2}(\m S^{d-1})},
\end{align*}
and we can bound
\begin{align*}
\nor{\opR_i u}{Y_{s-1,t}}^{2} &= \sum_{\ell =0}^{+\infty}  \langle \ell \rangle^{2(s-1)} e^{ 2 \left( \ell + \frac{d-2}{2} \right) t} \| P_\ell \opR_i u \|_{L^2(\m S^{d-1})}^2 \\
& \le C \sum_{\ell =0}^{+\infty}  \langle \ell \rangle^{2s} e^{ 2 \left( \ell + \frac{d-2}{2} \right) t} \|  P_{\ell+1} u \|_{L^{2}(\m S^{d-1})}^{2} \\
& \le C \sum_{\ell =1}^{+\infty}  \langle \ell-1 \rangle^{2s} e^{ 2 \left( \ell-1 + \frac{d-2}{2} \right) t} \|  P_{\ell} u \|_{L^{2}(\m S^{d-1})}^{2}\\
&\le Ce^{-2t} \nor{u}{Y_{s,t}}^{2}.
\end{align*}
Finally, thanks to \eqref{est:prod_ya}, we get
\begin{equation*} 
\left\| y_i\left(\opD  - \frac{d-2}{2} \right) u \right\|_{Y_{s-1,t}} \le C   e^{t} \| u \|_{Y_{s,t}}.
\end{equation*} 
The last statement in dimension $2$ is immediate taking the orthonormal basis $(e^{ik\theta})_{k\in \Z}$ and seeing that $P_{\ell}$ is the orthogonal projection on $\text{Span} (e^{i \ell \theta}, e^{- i \ell \theta})$.   
\enp

\begin{remark}
It might be instructive to see with one example the effect of the operator $D_i$ and the operator $\opR_i$, for instance in dimension $2$ where $\m S^{1}\approx \R/2\pi\Z$. Take $u=\sin (n\theta)$ with $n\in \m N^*$. Since (at least for small $\theta$), $\theta=\arctan(y/x)$ where the typical variable is $(x,y)\in \m  S^1\subset\R^2$ with $(x,y)=(\cos(\theta),\sin(\theta))$, we have $ \nabla_{\m S^{1}}u=(-\partial_x u,\partial _y u)=n\cos(n\theta)(- \sin (\theta),\cos(\theta))$. In particular, 
\[ D_x u= -n\cos(n\theta)\sin(\theta)=-n \sin(n\theta)\cos(\theta)+n\sin((n-1)\theta)=-x\opD u+\opR_x u. \]
Then
\[ \nor{\opR_x u}{Y_{s-1,t}} =  n (1+ (n-1)^2)^{(s-1)/2} e^{(n-1)t} \simeq e^{-t} n^s e^{nt} \simeq e^{-t} \nor{ u}{Y_{s,t}}, \]
as expected. 
The simplifications coming from some structure of the nonlinearity will for instance be consequences of identities like $x^2+y^2=1$, that is $\cos^2(\theta)+\sin^2(\theta)=1$. This allows to obtain,from a trigonometric polynomial of order $2$, a trigonometric polynomial of order $0$, and so, improves the estimates in the norms $Y_{s,t}$.
\end{remark}

\section{Scattering in conformal variables} \label{subsec:conf}
\subsection{A first general result}
Let $g$ be the nonlinearity after performing the conformal transform, that is 
\begin{gather*}
 f(u(x),\nabla u(x)) = g(t,y,v(t,y), \partial_t v(t,y), \nabla_{y} v(t,y)), \\  \text{where } e^t = |x| \text{ and } y = \frac{x}{|x|} \in \m S^{d-1}.
\end{gather*}

We state our result for system of equations on the unknowns $v = ( v_1, \dots, v_N)$, for ulterior purposes, in particular when studying harmonic maps. However to present the proofs, we try to limit notational inconvenience and we will assume $N=1$; the scalar case contains already the essence of the result. We separate the time variable $t$ (corresponding to the radial variable) because we will exploit the fact that it is better behaved.

So, we are interested in solving the system on $v = ( v_1, \dots, v_N)$, given by
\begin{equation} \label{eq:sys_conf}
\partial_{tt} v - \gh D^2 v = g (t, y, v, \partial_t v, \nabla_{y}  v)
\end{equation}
where $\gh D$ acts component by component, that is,
\begin{align} 
\forall i \in \llbracket 1, N \rrbracket, \quad \partial_{tt} v_i - \gh D^2 v_i = g_i (t, y, v,  \partial_t v, \nabla_{y}  v),
\end{align}
for a smooth function $g= (g_1, \dots, g_N)$ where $t \in \m R$, $v \in \m R^N$, $\partial_t v = (\partial_t v_1, \dots, \partial_t v_N)$, $\nabla_{y}  v = \nabla_{\m S^{d-1}}  v = (\nabla_{y}  v_1, \dots, \nabla_{y}  v_N)$ and for $i \in   \llbracket 1, N \rrbracket$,  $\partial_t v_i \in \m R$ and $(y, \nabla_y v_i) \in T \m S^{d-1}$.

As $f$ will be analytic (see Section \ref{sec:Duh}), we assume that for $i \in \llbracket 1, N \rrbracket$, the functions $g_i$ are, 
in variable $(t,y, v  , w, z)$, of the form of a series indexed by the parameters $\alpha \in \m N^d$, $\beta, \gamma \in \m N^N$, $\delta \in \mc M_{N,d}(\m N)$:
\begin{gather} \label{Def:g_series}
g_i (t,y,  v, w, z) = \sum_{\alpha,\beta,\gamma,\delta}  b_{i,\alpha,\beta,\gamma,\delta} (t) y^\alpha v^\beta w^\gamma z^\delta.
\end{gather}
In the above sum, we use the standard convention for multi-index powers of a vector:
\[ v^\beta = \prod_{i=1}^N v_i^{\beta_i}, \quad w^\gamma = \prod_{i=1}^N w_i^{\gamma_i},\quad z^\delta = \prod_{i=1}^N \prod_{j=1}^d z_{ij}^{\delta_{ij}}. \]

The $w$ and $z$ variable of $g_i$ are meant for the derivatives of $v$: $w_i$ will have the place of $\partial_t v_i$ and $z_{ij}$ that of $D_j v_i$ so that $(z_{ij})_{1 \le j \le d}$ describes $\nabla_y v_i \in T_y\m S^{d-1}\subset \R^d$ (recall the definition of $D_i$ in \eqref{defDi}).

In the various sums below, we use latin letters for index for which the sum is on finite sets and greek letters where the sum might be infinite.

For technical purpose, we will assume that each $b_{i,\alpha,\beta,\gamma,\delta}$ can be written 
\[
b_{i,\alpha,\beta,\gamma,\delta}(t)=\sum_{\iota\in \N} b_{i,\vartheta}(t),
\]
where, to simplify notations, we gather the parameters into one index
\[ \vartheta : = (\alpha, \beta, \gamma, \delta, \iota) \in \Theta := \m N^d \times \m N^{N} \times \m N^{N} \times \mc M_{N,d}(\m R) \times \N, \]
and for any $\vartheta \in \Theta$, there exists $B_\vartheta \ge 0$ and $\kappa_\vartheta \in \m R$ so that we have
\begin{align}
\label{estimaabetaiota}
\forall i \in \llbracket 1, N \rrbracket, \ \forall t\ge 0, \quad \left|b_{i,\vartheta} (t)\right|\le B_{\vartheta }e^{-\kappa_{\vartheta} t}.
\end{align}
This assumption will naturally fit $f$ being analytic, as it will be clear from paragraph \ref{s:scattinfty}. In many cases, it will be enough to consider that only the terms with $\iota=0$ are not zero; we will drop the index $\iota$ in this case. We will use the same convention if we have always $\alpha=0$.

We denote
\begin{equation} \label{def:nu_a}
    \nu_{\vartheta}=\kappa_{\vartheta}- |\alpha| + (|\beta|+ |\gamma| -1) \frac{d-2}{2} +|\delta|\frac{d-4}{2}.
\end{equation} 

(see Remark \ref{rkxplaincoeff} for explanations) and define the series
\begin{align}
\label{defhh1}
h(\sigma,\rho,\varsigma)&:=
 \sum_{\vartheta \in \Theta} B_{\vartheta} \left< \alpha \right>^{s+1}(|\beta| + |\gamma| + |\delta|) \sigma^{|\beta| + |\gamma| + |\delta| -1}  \rho^{\nu_{\vartheta}} \varsigma^{\kappa_{\vartheta}},  \\
h_1(\sigma ) &:=h(\sigma,1,1) =\sum_{\vartheta \in \Theta} B_{\vartheta} \left< \alpha \right>^{s+1}(|\beta|+ |\gamma| + |\delta|)  \sigma^{|\beta| + |\gamma| + |\delta|-1} . 
\end{align}
We can assume without loss of generality (due to the $(|\beta| + |\gamma| + |\delta|)$ factor) that 
\begin{align}
\label{hypokabc}
B_{\vartheta} \ne 0 \implies 
 |\beta|+|\gamma|+|\delta| \ge 1,
 \end{align}
and we will always assume that for sufficiently small $\sigma$, the series defining $h_1$ is convergent.

Denote
\begin{equation} \label{def:nu0}
\nu_0 = \inf \{  \nu_{\vartheta} : \vartheta \in \Theta, \  B_{\vartheta} \ne 0 \}.
\end{equation}
We will always assume
\begin{equation} \label{hypokabcnu}
 \nu_0 \ge 0.
\end{equation}
\begin{remark}
 \label{rkxplaincoeff}  
 The definition of the exponent $\nu_{\vartheta}$ in \eqref{def:nu_a} might seem a bit mysterious at first, but it just reflects the exponential decay given by any term $ b_{\vartheta} (t) y^\alpha v^\beta w^\gamma z^\delta$ which corresponds (in the scalar case) to $ b_{\vartheta} (t) y^\alpha v(t)^\beta (\partial_t v)^\gamma (D_i v)^{\delta_i}$.
 \begin{itemize}
\item $\kappa_{\vartheta}$ is the exponential decay of the constant (in $y$) $b_{\vartheta}$.
\item $|\alpha|$ comes from the loss described by \eqref{est:prod_ya} in Lemma \ref{lm:polyYst}.
\item $|\delta|$ comes from the loss of $e^t$ for the action of $D_i$ due to \eqref{est:Ytt0} in Lemma \ref{lmderiv}.
\item We have a multiplication of $ |\beta|+ |\gamma| +|\delta|$ functions in $Y_{s-1,t}$, so, it creates an exponential gain of factor $ (|\beta|+ |\gamma| +|\delta|-1) \frac{d-2}{2} $ due to \eqref{estimprodYst} in Lemma \ref{lmprodanalytique}.
 \end{itemize}
Adding all the above yield the rate $\nu_{\vartheta}$ in \eqref{def:nu_a}.
These exponents will be crucial in Lemma \ref{lminegG} below.
\end{remark}

The purpose of the following two results is to construct solutions of \eqref{eq:sys_conf}
defined for large times and with a prescribed linear behavior 
\[ \bs v_L(t) : = \q S(t)(\bs v_0, - \gh D \bs v_0) = (\q S(t)(\bs v_{i,0}, - \gh D \bs v_{i,0}))_{1 \le i \le N}, \]
 as $t \to +\infty$ (which is non growing in the $Y_{s,t}$ norms), where $\bs v_0 = (v_{1,0}, \dots, v_{1,N})$ is given. The idea is to perform a fixed point argument in ${\bs v} = (v,\vp) \in \q Y_{s,t}$ (that is, each component of $\bs v$ lies in $ \q Y_{s,t}$) on the map
\begin{gather} \label{def:Psi}
\Psi: {\bs v} \mapsto \Phi( g(t,y, \tilde v, \dot{\tilde{v}}, \nabla_{y} \tilde v)), \quad \text{with} \quad \tilde{\bs v} :={\bs v}+{\bs v}_L
\end{gather}
where  the map $\Phi$ is defined on \eqref{def:Phi} and acts component by component. 

Here is our first result.

\begin{thm}[Conformal variables] \label{th:conf}
Let $s > \frac{d}{2}+ \frac{3}{2}$. 
We assume a stronger version of  \eqref{hypokabcnu}, namely that $\nu_{0}>0$. 
%
Then, there exists $C>0$ and $\eta>0$ so that the following holds. Let 
\[  v_0 =( v_{i,0})_{1 \le i \le N} \in H^s(\m S^{d-1}), \quad \bs v_{L} := \q S(\cdot) ( v_{0} ,-\gh D  v_{0}), \]
and $t_{0}\ge t_{1}\ge 0$ such that
\begin{align}
\label{asumhvLt1}
h(C \| {\bs v}_L \|_{\q Y_{s,t_{0}}^{t_{1}}},e^{-(t_{0}-t_{1})},e^{-t_1})\le \eta.
\end{align}
Then there exists a unique (with $\bs v-\bs v_L$ small) solution $\bs v = (\bs v_1, \dots, \bs v_N) \in  \q Y_{s,t_0}^{t_1}$ (defined for times $t \ge t_0$) to the integral formulation of the system \eqref{eq:sys_conf}, with final condition
\begin{align}
\label{cvceinftysol}
 \| \bs v - \bs v_L \|_{\q Y_{s,t}} \lesssim e^{-\nu_{0}t}\to 0 \quad \text{as} \quad t \to +\infty. 
\end{align}
\end{thm}

We emphasize that \eqref{asumhvLt1} assumes implicitly that the quantity 
\[ h(C \| {\bs v}_L \|_{\q Y_{s,t_{0}}^{t_{1}}},e^{-(t_{0}-t_{1})},e^{-t_1}) \] is finite, which is not always the case. To discuss this, given $A, \e, D \in \m R$, we consider the following property on $h$:
\begin{align}
\label{infkvartheta}
\forall \vartheta \in \Theta \text{ with } B_\vartheta \ne 0, \quad \kappa_{\vartheta}\ge (-A+\e) (|\beta|+ |\gamma|+|\delta|-1)-D.
\end{align}
This is convenient because it allows to relate $h$ and $h_1$, and in particular ensure convergence of the former.

\begin{claim} \label{cl:h_h1}
Assume that $h$ satisfies \eqref{infkvartheta} with $A, \e \ge 0$ and $D \in \m R$. Then there holds
\begin{equation} \label{est:h_h_1} 
\forall \sigma \ge 0, \ \forall \rho, \lambda \in (0,1], \quad h(\lambda^{A} \sigma, \rho, \lambda) \le \lambda^{-D} \rho^{\nu_0} h_1(\lambda^\e \sigma).
\end{equation}
\end{claim}

\begin{proof}
We have
\begin{align*}
h(\lambda^{A} \sigma, \rho, \lambda) & \le \sum_{\vartheta \in \Theta} B_{\vartheta} \left< \alpha \right>^{{s+1}}(|\beta| + |\gamma| + |\delta|)  \sigma^{|\beta| + |\gamma| + |\delta| -1} \rho^{\nu_\vartheta} \lambda^{ m_{\vartheta}}
\end{align*}
where
\[ m_{\vartheta} := \kappa_{\vartheta}+ A \left(|\beta| + |\gamma| + |\delta| -1\right)\ge \e \left(|\beta| + |\gamma| + |\delta| -1\right)-D. \]
(We used the assumption \eqref{infkvartheta} for the inequality).
Now, $\nu_\vartheta \ge \nu_0$ and as $0 \le \rho, \lambda \le 1$, we infer
\begin{align*}
h(\lambda^{A} \sigma, \rho, \lambda) & \le    \sum_{\vartheta \in \Theta} B_{\vartheta} \left< \alpha \right>^{{s+1}}(|\beta| + |\gamma| + |\delta|) (\lambda^\e \sigma)^{|\beta| + |\gamma| + |\delta| -1} \rho^{\nu_0}  \lambda^{ -D} \\
& \le \lambda^{ -D} \rho^{\nu_0} h_1(\lambda^\e \sigma). \qedhere
\end{align*}
\end{proof}

We are now in a position to give some conditions under which Theorem \ref{th:conf} applies (proven after its proof).

\begin{lem} \label{lem:conf}
Under the reinforced condition $\nu_0>0$, the assumption \eqref{asumhvLt1} holds if one of the following assumptions is satisfied
\begin{enumerate}
\item \label{hypott0} 
$h_1(0)=0$ and we  have $t_{0}=t_{1}=0$ and $\| v_{0}\|_{H^s(\m S^{d-1})}$ is small enough.
\item \label{hyp3nu}
$t_1=0$ and $t_0$ is large enough (depending on $v_0$), and there exist $\e>0$ and $D\in \R$ and so that, for all $\vartheta \in \Theta$ such that $B_\vartheta \ne 0$,
\begin{align}
\label{infnuvartheta}
\nu_{\vartheta}\ge \e (|\beta|+ |\gamma|+|\delta|-1)-D.
\end{align}
\item there exists $\ell_0 \in \m N$ such that $P_\ell v_0 =0$ for all $\ell < \ell_0$, and $h$ satisfies \eqref{infkvartheta} with $A =  \ell_0 + \frac{d-2}{2}$, $\e >0$ and $D \in \m R$; and $t_0$ and $t_0-t_1$ are large enough (depending on $v_0$).
\end{enumerate}
\end{lem}

\begin{remark}
Assumptions \eqref{infnuvartheta} and \eqref{infkvartheta} (for any $\eta$) are obvious if $B_{\vartheta}\neq 0$ only for a finite number of $\vartheta$: this corresponds to a polynomial nonlinearity in the original variable. They are mainly made to ensure the convergence of the series. The condition 3) is actually used only for $\ell_0 =0$ or $1$.
\end{remark}

The proof of Theorem \ref{th:conf} follows from the following technical but crucial estimates. As mentioned above, we will assume for the purpose of the proofs that we are in the scalar case $N=1$.

\begin{lem}
\label{lminegG}
Let $s > \frac{d}{2} + \frac{3}{2}$.
There exists a universal constant $C>0$ so that for every $t\ge t_{1}\ge0$, $ {\bs v}$, ${\bs w}\in \q Y_{s,t}^{t_{1}}$, and denoting  
\begin{align}\label{defMt} M (t)=   \max\left( \| {\bs v} \|_{\q Y_{s,t}^{t_{1}}},  \| {\bs w} \|_{\q Y_{s,t}^{t_{1}} },\| {\bs v}_L \|_{\q Y_{s,t}^{t_{1}} } \right), 
\end{align}
then we have  
\begin{align}
\| \Psi ({\bs v}) - \Psi ({\bs w})  \|_{\q Y_{s,t}^{t_{1}}} & \lesssim  \sum_{\vartheta \in \Theta} e^{-\kappa_{\vartheta}t_1} B_{\vartheta} \left< \alpha \right>^{s+1}(|\beta| + |\gamma|+|\delta|) (CM(t))^{|\beta| + |\gamma|+|\delta|-1} \\
& \qquad \qquad \times \int_t^{+\infty}\| {\bs v} - {\bs w} \|_{\q Y_{s,\tau}^{t_{1}}}(1+ \tau-t )  e^{- \nu_{\vartheta}(\tau-t_{1})} d\tau.
\end{align}
Similarly,
\begin{align}
\| \Psi ({\bs v})  \|_{\q Y_{s,t}^{t_{1}}} & \lesssim  \sum_{\vartheta \in \Theta}  e^{-\kappa_{\vartheta}t_1 } B_{\vartheta} \left< \alpha \right>^{s+1}(|\beta| + |\gamma|+|\delta|) (CM(t))^{|\beta| + |\gamma|+|\delta|-1} \\
& \qquad \qquad \times \int_t^{+\infty}(\| {\bs v} \|_{\q Y_{s,\tau}^{t_{1}}} +\| {\bs v_{L}} \|_{\q Y_{s,\tau}^{t_{1}}} )(1+ \tau-t ) e^{- \nu_{\vartheta}(\tau-t_{1})} d\tau.
\end{align}
\end{lem}

\bnp
We do the difference estimate only, as the other one follows in a similar fashion. First, note that due to $\eqref{defDi}$, \[ (\nabla_{y} v(\tau))_{i,j}= e_{j}\cdot \nabla_{y} v_i(\tau)=D_j v_i(\tau). \]
From  Proposition \ref{lmprodanalytique} and Lemma \ref{lmderiv}, as $s-1 > \frac{d}{2} + \frac{1}{2}$, for any $\beta, \gamma \in \m N$, $\delta \in \m N^{d}$ (recall we do the proof for $N=1$), and time $\tau \ge t \ge t_{1}$, denoting
\[ \bs{ \tilde v} = \bs v +\bs v_L, \quad \bs{\tilde w} = \bs w + \bs v_L, \]
there hold
\begin{align}
\MoveEqLeft \| \tilde v(\tau)^\beta \dot{\tilde v}(\tau)^\gamma (\nabla_{y} \tilde v(\tau)) ^\delta - \tilde w(\tau)^\beta \dot{\tilde w}(\tau)^\gamma (\nabla_{y} \tilde w(\tau)) ^\delta  \|_{Y_{s-1,\tau-t_{1}}} \\
& \lesssim  (|\beta| + |\gamma| + |\delta|) C_0^{|\beta| + |\gamma| + |\delta|-1} e^{- ((|\beta|+ |\gamma|+|\delta| -1) \frac{d-2}{2} - {|\delta|} ) (\tau-t_{1})} \\
& \qquad \times (\| {\bs v} \|_{\q Y_{s,\tau}^{t_{1}}} + \| {\bs w} \|_{\q Y_{s,\tau}^{t_{1}}} +\| {\bs v}_L \|_{\q Y_{s,\tau}^{t_{1}}})^{|\beta|+|\gamma|+|\delta|-1} \| {\bs v} - {\bs w} \|_{\q Y_{s,\tau}^{t_{1}}} \\
& \lesssim  (|\beta| + |\gamma|+|\delta|)(C M(t))^{|\beta| + |\gamma|+|\delta|-1}  e^{- ((|\beta|+ |\gamma| -1) \frac{d-2}{2} + |\delta| \frac{d - 4}{2} )(\tau-t_{1})} \| {\bs v} - {\bs w} \|_{\q Y_{s,\tau}^{t_{1}}}. \label{est:mult_Psi}
\end{align}
(We denoted $C_0$ the maximum of the constants appearing in estimates \eqref{est:Ytt0} and \eqref{estimprodYst}, and one can pick $C = 3C_0$. We also used $\| {\bs v} \|_{\q Y_{s,\tau}^{t_{1}}}\le M(t)$ for $\tau \ge t$ and the same for $\bs v$ and ${\bs v}_L $.)

Assume first $d\ge 3$. Using the above inequality together with the Duhamel bound \eqref{est:lin_Duhamel}, the product law \eqref{est:prod_ya} and the decay \eqref{estimaabetaiota}, we get
\begin{align*}
\MoveEqLeft \| \Psi ({\bs v}) - \Psi ({\bs w})  \|_{\q Y_{s,t}^{t_{1}}} \\
& \le \sum_{\vartheta \in \Theta} \int_t^{\infty} |b_{\vartheta}(\tau)| \| y^\alpha \left(  \tilde v(\tau)^\beta  \dot{\tilde v}(\tau)^\gamma (\nabla_{y} \tilde v(\tau)) ^\delta \right. \\
& \qquad \qquad \qquad \left. - \tilde w(\tau)^\beta \dot{\tilde w}(\tau)^\gamma (\nabla_{y} \tilde w(\tau)) ^\delta \right)  \|_{Y_{s-1,\tau-t_{1}}} d\tau \\
& \lesssim  \sum_{\vartheta \in \Theta} B_{\vartheta}\left< \alpha \right>^{s + 1}(|\beta| + |\gamma|+|\delta|) (CM(t))^{|\beta| + |\gamma| + |\delta| -1} \\
& \qquad \qquad \qquad \times \int_t^{+\infty}e^{-\kappa_{\vartheta}\tau } e^{|\alpha|(\tau-t_1)}e^{- ((|\beta|+ |\gamma| -1) \frac{d-2}{2} + |\delta| \frac{d -4}{2} )(\tau-t_{1})} \| {\bs v} - {\bs w} \|_{\q Y_{s,\tau}^{t_{1}}} d\tau\\
& \lesssim  \sum_{\vartheta \in \Theta} B_{\vartheta} \left< \alpha \right>^{s +1} e^{-\kappa_{\vartheta}t_1 }(|\beta| + |\gamma|+|\delta|) (CM(t))^{|\beta| + |\gamma| + |\delta| -1} \\
& \qquad \qquad \qquad \times \int_t^{+\infty} e^{-\nu_{\vartheta}(\tau-t_{1})} \| {\bs v} - {\bs w} \|_{\q Y_{s,\tau}^{t_{1}}} d\tau.
\end{align*}
(In the last inequality, we have used $e^{-\kappa_{\vartheta}\tau }= e^{-\kappa_{\vartheta}t_1 }e^{-\kappa_{\vartheta}(\tau-t_1) }$ and the definition \eqref{def:nu_a} of $\nu_{\vartheta}$.) This gives the expected estimates in this case.

If $d=2$, we rely on estimate \eqref{est:lin_Duhamel}. So, we need to bound the term with $P_{0}$, which we can estimate as before (because $s-1 > 1= d/2$): for any $\tau \ge t_1$,
\[ \| P_0 v \|_{L^2(\m S^1)} \le \| P_0 v \|_{H^{s-1}(\m S^1)} \le \| v \|_{Y_{s-1,\tau-t_1} }, \]
and so
\begin{align*} 
\int_{t}^{+\infty} (\tau-t) \| P_0 F(\tau) \|_{L^2(\m S^{1})}d\tau & \lesssim  \sum_{\vartheta \in \Theta} B_{\vartheta} \left< \alpha \right>^{s+1} e^{-\kappa_{\vartheta}t_1} (|\beta| + |\gamma|+|\delta|) \\ 
 & \qquad  \times (CM(t))^{|\beta| + |\gamma|+|\delta|-1} \\
 \int_t^{+\infty} (\tau-t)e^{-\nu_{\vartheta}(\tau-t_{1})} \| {\bs v} - {\bs w} \|_{\q Y_{s,\tau}^{t_{1}}} d\tau.  \qedhere
\end{align*}
\enp

\begin{proof}[Proof of Theorem \ref{th:conf}]
We consider
\begin{align}\label{defY} 
Y = \left\{ {\bs w} \in \q Y_{s, t_0}^{t_{1}} : \| {\bs w} \|_{\q Y_{s, t_0}^{t_{1}}} \le \| {\bs v}_L \|_{\q Y_{s, t_0}^{t_{1}} } \right\}. 
\end{align}
Observe that for $t \ge t_1$, and $\nu \ge \nu_0 >0$,
\begin{align}
\int_t^{+\infty} (1+\tau-t) e^{- \nu(\tau-t_{1})} d\tau & = e^{- \nu(t-t_{1})}\int_t^{+\infty} (1+\tau-t) e^{- \nu(\tau-t)} d\tau \nonumber\\ 
\le  C_{\nu_0} e^{- \nu (t-t_{1})}. \label{est:exp_tail}
\end{align}
Let ${\bs v}, {\bs w}\in Y$ and $t\ge t_0$. We have $ \| {\bs v} \|_{\q Y_{s,t}^{t_{1}}}\le \| {\bs v} \|_{\q Y_{s,t_0}^{t_{1}}} \le \|{\bs v}_L \|_{\q Y_{s,t_0}^{t_{1}}}$ and similarly $\| {\bs v} \|_{\q Y_{s,t}^{t_{1}}} \le \|{\bs v}_L \|_{\q Y_{s,t_0}^{t_{1}}}$. In particular, 
\[ M (t)\le \|{\bs v}_L \|_{\q Y_{s,t_0}^{t_{1}}}, \]
where $M(t)$ was defined in \eqref{defMt}. Hence, using Lemma \ref{lminegG}, we get for ${\bs v}, {\bs w}\in Y$ and $t\ge t_0\ge t_1\ge 0$,
\begin{align}
\label{est:Psi1diff}
\MoveEqLeft \| \Psi ({\bs v}) - \Psi ({\bs w}) \|_{\q Y_{s,t}^{t_{1}}} \\
& \lesssim  \sum_{\vartheta \in \Theta} B_{\vartheta} \left< \alpha \right>^{s +1 }  e^{-\kappa_{\vartheta}t_1 }(|\beta| + |\gamma|+|\delta|) (C \| {\bs v}_L \|_{\q Y_{s, t_0}^{t_{1}}})^{|\beta| + |\gamma|+|\delta|-1} \\
& \qquad \qquad \times \int_t^{+\infty}\| {\bs v} - {\bs w} \|_{\q Y_{s,\tau}^{t_1}} (1+\tau-t) e^{- \nu_{\vartheta}(\tau-t_{1})} d\tau \\
& \lesssim  \sum_{\vartheta \in \Theta} B_{\vartheta} \left< \alpha \right>^{s+1} e^{-\kappa_{\vartheta}t_1 }(|\beta| + |\gamma|+|\delta|) (C \| {\bs v}_L \|_{\q Y_{s,t_0}^{t_{1}}})^{|\beta| + |\gamma|+|\delta|-1} \\
& \qquad \qquad \times \| {\bs v} - {\bs w} \|_{\q Y_{s,t}^{t_{1}}}  e^{- \nu_{\vartheta}(t-t_1)}  \\
&\lesssim h(C \| {\bs v}_L \|_{\q Y_{s,t_0}}^{t_{1}}, e^{-(t-t_1)},e^{-t_1}) \| {\bs v} - {\bs w} \|_{\q Y_{s,t}^{t_{1}}}.  
\end{align}
And similarly, there holds
\begin{align}
\label{est:Psi1}
 \| \Psi ({\bs v})   \|_{\q Y_{s,t}^{t_{1}}} & \lesssim h(C \| {\bs v}_L \|_{\q Y_{s, t_0}^{t_{1}}},e^{-(t-t_1)}, e^{-t_1 })( \| {\bs v}  \|_{\q Y_{s,t}^{t_{1}}}+ \| {\bs v_{L}}  \|_{\q Y_{s,t}^{t_{1}}} ).
\end{align}
Denote $C_1$ the maximum of the implicit constants appearing in \eqref{est:Psi1diff}  and \eqref{est:Psi1}, and choose $\eta= 1/(2C_1)$. The above computations, applied with $t =t_0$, prove that $\Psi$ maps the closed $Y$ into itself and is contracting in the Banach space $\q Y_{s,t_0}^{t_{1}}$. Therefore, $\Psi$ admits a unique fixed point ${\bs r}=(r,\dot{r})$ in $Y$. Then from definition \eqref{def:Psi} and Lemma \ref{lmDuhamelinfty}, we have $ \dot{r}=\partial_t r$ and 
\begin{align*}
\partial_{tt} r - \gh D^2 r 
& =\Phi( g(t,y, r+v_L, \partial_t( r+v_L), \nabla_{y}(r+v_L))).
\end{align*}
In particular, ${\bs v_{L}}+{\bs r}$ is the desired solution.  For \eqref{cvceinftysol}, we combine \eqref{est:Psi1} with the bound $h(\sigma, \rho_1 \rho_2,\varsigma) \le \rho_1^{\nu_0} h(\sigma,\rho_2,\varsigma)$ for $\rho_1\le 1$ and $\varsigma >0$, so that 
\begin{align*}
h(C \| {\bs v}_L \|_{\q Y_{s,t}},e^{-(t-t_1)},e^{-t_1}) & \le h(C \| {\bs v}_L \|_{\q Y_{s,t_{0}}},e^{-(t_{0}-t_1)}, e^{-t_1 })e^{- \nu_{0}(t-t_0)} \\
& \le \eta e^{- \nu_{0}(t-t_0)}.
\end{align*}
Finally, the uniqueness of solutions with $\bs v-\bs v_L$ small in  $\q Y_{s,t_0}^{t_1}$ is a consequence of the uniqueness of the fixed point.
\end{proof}

\begin{proof}[Proof of Lemma \ref{lem:conf}]
For (1), recall that from Lemma \ref{lmlienlinY}, $\| {\bs v}_L \|_{\q Y_{s}} \lesssim \| v_0 \|_{H^s(\m S^{d-1})}$ is small, and actually can be made smaller that the radius of convergence of $h_1$; as $h_1 \to 0$ at $0$, $ h(C \| {\bs v}_L \|_{\q Y_{s}},1,1) = h_1(C \| {\bs v}_L \|_{\q Y_{s}})$ can be made small.

For (2), we observe that the hypothesis together with $\nu_0 >0$ implies that  there exist $\eta>0$ such that
\[ \nu_{\vartheta} \ge 2\eta (|\beta|+|\gamma|+|\delta|-1) + \frac{\nu_0}{2}. \]
For example,  $\eta=\frac{1}{4}\frac{\e\nu_0}{|D|+\nu_0}$ fits using
\[ (|D|+\nu_0)\nu_{\vartheta}\ge |D|\nu_0+ \nu_0 \e (|\beta|+|\gamma|+|\delta|-1)- \nu_0 D\ge \nu_0 \e (|\beta|+|\gamma|+|\delta|-1). \]
Then we decompose
\[ \nu_{\vartheta} \ge \frac{\nu_{\vartheta}}{2} + \frac{\nu_0}{4} + \frac{\nu_{\vartheta}}{4} \ge  \eta (|\beta|+|\gamma|+|\delta|-1) + \frac{\nu_0}{2} + \frac{\nu_{\vartheta}}{4}, \]
which leads to 
\begin{align*}
h(\sigma,e^{-t_0},1) & \le 
 \sum_{\vartheta} B_{\vartheta} \left< \alpha \right>^{s+1}(|\beta| + |\gamma|+|\delta|) \\
 & \qquad \times e^{-\nu_0 t_0/2} (\sigma e^{-\eta t_{0}})^{|\beta| + |\gamma| + |\delta| -1} e^{-\nu_{\vartheta} t_0/4} \\
 & \le e^{-\nu_0t_0/2} h(\sigma e^{-\eta t_{0}},e^{-t_0/4},1).
 \end{align*}
 Since by Lemma \ref{lmlienlinY}, we have $\| {\bs v}_L \|_{\q Y_{s,t_{0}}^{0}}\le C\| v_0 \|_{H^s(\m S^{d-1})}$ uniformly on $t_0\ge 0$, applying the previous estimate with $\sigma=C \| {\bs v}_L \|_{\q Y_{s,t_{0}}^{0}}$ gives  $h(C \| {\bs v}_L \|_{\q Y_{s,t_{0}}^{0}},e^{-t_0},1)  \le e^{-\nu_0t_0/2} h(C \| v_0 \|_{H^s(\m S^{d-1})}e^{-\eta t_{0}},e^{-t_0/4},1).$
As $h(\cdot,\cdot,1)$ is defined and bounded on a neighbourhood of $(0,0)$, this last expression is finite and arbitrary small for large $t_0$.

For (3), recall the estimate \eqref{est:Hs_Yst0_exp_2d} of Lemma \ref{lmlienlinY}, which is uniform for $t_0\ge t_1\ge 0$:
\begin{gather} 
\| {\bs v}_L\|_{\q Y_{s,t_0}^{t_{1}}}\le  C e^{-\left(\ell_0+\frac{d-2}{2}\right)t_{1}}\| v_0 \|_{H^s(\m S^{d-1})}.
\end{gather}
Since $h_1$ has a positive radius of convergence, we can  fix $t_1 \ge 0$ large enough so that $h_1 (e^{-\e t_1} \| v_0 \|_{H^s(\m S^{d-1})})$ is finite.

Using Claim \ref{cl:h_h1}, for $t_1$ as above and any $t_0 \ge t_1$, and as all coefficients are positive, there hold
\begin{align*}
h(C \| {\bs v}_L \|_{\q Y_{s,t_{0}}^{t_{1}}},e^{-(t_0-t_1)},e^{-t_1}) &\le h(C e^{-(\ell_0 + \frac{d-2}{2}) t_1} \| {\bs v}_L \|_{\q Y_{s,t_{0}}^{t_{1}}},e^{-(t_0-t_1)},e^{-t_1}) \\
& \le  e^{Dt_1} e^{-\nu_0(t_0-t_1)} h_1 (C e^{-\e t_1} \| v_0 \|_{H^s(\m S^{d-1})}).
\end{align*}
In the current case, $\nu_0 >0$: hence it suffice to choose $t_0-t_1$ so large that the $e^{-\nu_0(t_0-t_1)}$ factor absorbs the $e^{Dt_1}$ factor and make the right-hand side small.
%
\end{proof}

There are some limit situations where the assumptions of the previous theorem are not fulfilled, but we can still build a solution. This is the case for example if the first iterate of the Duhamel formula (that is $\Psi(0)$) is better than expected and decays in time: a convenient space is given by the norm
\begin{align} \label{def:Z}
 \| {\bs v} \|_{\q X_{\nu,t_0}^{t_1}} = \sup_{t \ge t_0} e^{\nu (t-t_{1})}  \| {\bs v} \|_{\q Y_{s,t}^{t_{1}}}, \end{align}
given $\nu >0$ and $t_0 \ge t_1 \ge 0$. We also simply denote $\q X_{\nu,t_0}^{t_0} = \q X_{\nu,t_0}$. $\q X_{\nu,t_0}^{t_1}$ defines also a Banach space (as we done for the other space times spaces like $\q Z_s^\infty$ etc.). Here is our result.

\begin{thm}[Conformal variables 2] \label{th:conf2}
Let $s > \frac{d}{2} + \frac{3}{2}$ and $\nu>0$. Then, there exists $C>0$ and $\eta>0$ so that the following holds. 

Let $ v_0 = (v_{i,0})_{1 \le i \le N} \in H^s(\m S^{d-1})$ and $\bs v_{L} := \q S(\cdot) ( v_{0} ,-\gh D  v_{0})$. Recall that we assumed \eqref{hypokabcnu}; also assume that for $t_0 \ge 0$, 
\begin{equation} \label{est:psi(0)_exp}
\Psi(0) \in \q X_{\nu,t_0}.
\end{equation}
($\Psi$ is defined in \eqref{def:Psi}; $\nu$ may depend on ${\bs v}_0$). We finally  assume $t_{0}\ge t_{1}\ge 0$ are such that 
\begin{gather} \label{est:h1_small}
h(C N,e^{-(t_0-t_1)},e^{-t_1}) \le  \eta, \quad \text{where} \\
N: = \max\left( e^{-\nu(t_0-t_1)} \| \Psi(0) \|_{\q X_{\nu,t_0}^{t_1}}, \| \bs v_L \|_{\q Y_{s,t_{0}}^{t_1}} \right). \nonumber
\end{gather}
Then, there exists a solution $\bs v \in \q Y_{s,t_0}^{t_1}$ (defined for times $t \ge t_0$) to the integral formulation of the system \eqref{eq:sys_conf}, with final condition
\begin{align}\label{cvceinftysolnu} 
\| \bs v - \bs v_{L}  \|_{\q Y_{s,t}^{t_1}}\lesssim e^{-\nu t} \to 0 \quad \text{as} \quad t \to +\infty. 
\end{align}
Furthermore, one has the more precise convergence
\begin{align}\label{est:sol_conf_nu+nu_0}
\| \bs v - \bs v_{L}  - \Psi(0) \|_{\q Y_{s,t}^{t_1}}\lesssim e^{-(\nu+\nu_0) t}.
\end{align}
Uniqueness holds for $\bs v - \bs v_{L}  - \Psi(0)$ small in $\q X_{\nu,t_0}^{t_1}$.
\end{thm}
 
As before, \eqref{est:h1_small} assumes implicitly that the quantity $h(C N,e^{-(t_0-t_1)},e^{-t_1})$ has a finite value. For instance, it happens in the following situations.

\begin{lem} \label{lem:th_conf2_hyp}
Here, we assume $h_1(0)=0$ and $\nu_0\ge0 $.
The assumption \eqref{est:h1_small} is satisfied for instance if either
\begin{enumerate}
\item $\| v_{0}\|_{H^s(\m S^{d-1})}$ and $\| \Psi(0) \|_{{\q X_{\nu,0}} }$ are small enough, and $t_{0}=t_{1}=0$.

\item $\Psi(0) \in \q X_{\nu,0}$, and there exists $\ell_0 \in \m N$ such that $\ds P_{\ell} v_0 = 0$ for $\ell< \ell_0\in \N$ and  $h$ satisfies \eqref{infkvartheta} with $A = \ell_0 + \frac{d-2}{2}$, $\e >0$ and $D =0$; and $t_1$ and $t_0-t_1$ are large enough (depending on $v_0$).
\end{enumerate}
\end{lem}


\begin{proof}
Let $t_{0}\ge t_{1}\ge 0$ as in the assumption. The idea is to perform a fixed point argument on the map $\tilde \Psi$ defined by 
\begin{align}\label{defpsitilde} \tilde \Psi(\bs v) = \Psi(\Psi(0)+\bs v) - \Psi(0), 
\end{align}
in a ball of $\q X_{\nu,t_0}^{t_1}$. We apply Lemma \ref{lminegG}, and we get, for $t \ge t_0$ and denoting $M(t)=\max\left( \| {\bs v} + \Psi(0) \|_{\q Y_{s,t}^{t_{1}} },\| {\bs v}_L \|_{\q Y_{s,t}^{t_{1}} } \right)$:
\begin{align*}
\| \tilde \Psi ({\bs v})  \|_{\q Y_{s,t}^{t_{1}}} & = \| \Psi (\Psi(0)+{\bs v})- \Psi(0)  \|_{\q Y_{s,t}^{t_{1}}}\\ 
& \lesssim  \sum_{\vartheta\in \Theta} B_{\vartheta} \left< \alpha \right>^{s+1} e^{-\kappa_{\vartheta}t_1 } (|\beta| + |\gamma|+|\delta|) (C M(t))^{|\beta| + |\gamma|+|\delta|-1} \\
& \qquad \qquad \times \int_t^{+\infty}\| \Psi(0)+{\bs v} \|_{\q Y_{s,\tau}^{t_{1}}}  (1+ \tau-t) e^{- \nu_{\vartheta}(\tau-t_{1})} d\tau \\
& \lesssim  \sum_{\vartheta \in \Theta} B_{\vartheta} \left< \alpha \right>^{s+1} e^{-\kappa_{\vartheta}t_1} (|\beta| + |\gamma| +|\delta|) (C M(t))^{|\beta| + |\gamma| +|\delta| -1} \\
&  \qquad \qquad \times \| \Psi(0)+ {\bs v}\|_{\q X_{\nu,t}^{t_1}} \int_t^{+\infty}  (1+\tau-t)e^{- (\nu_{\vartheta}+\nu)(\tau-t_{1})} d\tau \\
& \lesssim  \sum_{\vartheta \in \Theta} B_{\vartheta} \left< \alpha \right>^{s+1} e^{-\kappa_{\vartheta}t_1} (|\beta| + |\gamma| +|\delta|) (C M(t))^{|\beta| + |\gamma| +|\delta|-1} \\
& \qquad \qquad  \times \| \Psi(0)+{\bs v} \|_{\q X_{\nu,t}^{t_1}} e^{- (\nu_{\vartheta}+\nu)(t-t_{1})}.
\end{align*} 
We use used \eqref{est:exp_tail},  with an implicit constant dependent on $\nu>0$. This yields, when applied to $t=t_0$,
\begin{align}
\| \tilde \Psi ({\bs v})  \|_{\q X_{\nu,t_0}^{t_{1}}} & \lesssim  \sum_{\vartheta \in \Theta} B_{\vartheta} \left< \alpha \right>^{s+1}(|\beta| + |\gamma|+|\delta|) e^{-\kappa_{\vartheta}t_1 } e^{- \nu_{\vartheta}(t_0-t_{1})}  (C M(t_0))^{|\beta| + |\gamma|+|\delta|-1} \nonumber \\
& \qquad \qquad \times \left(\| {\bs v} \|_{\q X_{\nu,t_0}^{t_1}}+\| \Psi(0) \|_{\q X_{\nu,t_0}^{t_1}}\right) \nonumber \\
&  \lesssim  h(C M(t_0),e^{-(t_0-t_1)},e^{-t_1})\left(\| {\bs v} \|_{\q X_{\nu,t_0}^{t_1}}+\| \Psi(0) \|_{\q X_{\nu,t_0}^{t_1}}\right). \label{est:tPsi_vtheta}
\end{align} 
This useful for the fixed point argument. For the more precise convergence, we go back to the next to last bound, and derive the sharper bound (recall $\nu_{\vartheta}\ge \nu_0$), for all $t \ge t_0$,
\begin{align}
\| \tilde \Psi ({\bs v})  \|_{\q X_{\nu+\nu_0,t}^{t_{1}}} & \lesssim  \sum_{\vartheta \in \Theta} B_{\vartheta} \left< \alpha \right>^{s+1}(|\beta| + |\gamma|+|\delta|) e^{-\kappa_{\vartheta}t_1 } (C M(t))^{|\beta| + |\gamma|+|\delta|-1} \\
& \qquad \qquad \times \left(\| {\bs v} \|_{\q X_{\nu,t}^{t_1}}+\| \Psi(0) \|_{\q X_{\nu,t}^{t_1}}\right) \nonumber \\
&  \lesssim  h(C M(t),1,e^{-t_1}) \left(\| {\bs v} \|_{\q X_{\nu,t_0}^{t_1}}+\| \Psi(0) \|_{\q X_{\nu,t_0}^{t_1}}\right). \label{est:tPsi_v}
\end{align} 
We now turn to the difference estimate: using again Lemma \ref{lminegG} and denoting 
\[ N(t)= \max\left( \| \Psi(0)+ {\bs v} \|_{\q Y_{s,t}^{t_{1}}},  \| \Psi(0)+ {\bs w} \|_{\q Y_{s,t}^{t_{1}} }, \| {\bs v}_L \|_{\q Y_{s,t}^{t_{1}} } \right) \]
there hold:
\begin{align*}
\| \tilde \Psi ({\bs v}) -  \tilde \Psi ({\bs w}) \|_{\q Y_{s,t}^{t_{1}}} & = \| \Psi (\Psi(0)+{\bs v}) - \Psi(\Psi(0)+{\bs w}) \|_{\q Y_{s,t}^{t_{1}}}\\ 
& \lesssim  \sum_{\vartheta \in \Theta} B_{\vartheta}\left< \alpha \right>^{s +1} e^{-\kappa_{\vartheta}t_1} (|\beta| + |\gamma|+|\delta|) (C N(t))^{|\beta| + |\gamma|+|\delta|-1} \\
& \qquad \qquad \times \int_t^{+\infty}\| {\bs v} - {\bs w} \|_{\q Y_{s,\tau}^{t_{1}}}( 1+ \tau-t)  e^{- \nu_{\vartheta}(\tau-t_{1})} d\tau.
\end{align*}
As before, we infer
\begin{align*}
\| \tilde \Psi ({\bs v}) -  \tilde \Psi ({\bs w})  \|_{\q X_{\nu,t_0}^{t_{1}}} & \lesssim h(C N(t_0),e^{-(t_0-t_1)},e^{-t_1}) \| {\bs v}- {\bs w} \|_{\q X_{\nu,t_0}^{t_1}}.
\end{align*} 

Consider
\[ Y = \left\{ \bs w \in \q X_{\nu,t_0}^{t_{1}} : 
\| {\bs w} \|_{\q X_{\nu,t_0}^{t_1}} \le \max(\| \Psi(0) \|_{\q X_{\nu,t_0}^{t_{1}} },\| {\bs v}_L \|_{\q Y_{s,t_0}^{t_{1}} })
\right\}. \]
If $\bs v \in Y$, then, since $\nu>0$ and $t_0\ge t_1$, we have
\begin{align*}
\| \bs v \|_{\q Y_{s,t_0}}^{t_{1}} & \le e^{-\nu (t_0-t_1)} \| \bs v \|_{\q X_{\nu,t_0}^{t_{1}}}\le e^{-\nu (t_0-t_1)}\max(\| \Psi(0) \|_{\q X_{\nu,t_0}^{t_{1}} }, \| {\bs v}_L \|_{\q Y_{s,t_0}^{t_{1}} })\\
& \le \max(e^{-\nu (t_0-t_1)}\| \Psi(0) \|_{\q X_{\nu,t_0}^{t_{1}} },\| {\bs v}_L \|_{\q Y_{s,t_0}^{t_{1}} }).
\end{align*}
Therefore, given $\bs v, \bs w \in Y$, there hold
\begin{align*}
M(t_0),  N(t_0) & \le  \| \Psi(0) \|_{\q Y_{s,t_0}^{t_{1}} } +  \max(e^{-\nu(t_0-t_1)} \| \Psi(0) \|_{\q X_{\nu,t_0}^{t_{1}} }, \| {\bs v}_L \|_{\q Y_{s,t_0}^{t_{1}} }) \\
& \le 2 \max(e^{-\nu(t_0-t_1)} \| \Psi(0) \|_{\q X_{\nu,t_0}^{t_{1}} }, \| {\bs v}_L \|_{\q Y_{s,t_0}^{t_{1}} }) = 2N. 
\end{align*} 
Therefore $h(C M(t_0),e^{-(t_0-t_1)},e^{-t_1})$, $h(C N(t_0),e^{-(t_0-t_1)},e^{-t_1}) \le  \eta$. If $\eta$ is small enough, as $\nu_0 \ge 0$, the previous estimates show that $\tilde \Psi$ is a contraction in the complete space $Y$ and so has a (unique) fixed point ${\bs r}$ there: hence $\bs v : = \bs v_L + \Psi(0) + \bs r$  has the required properties. Indeed, from definition \eqref{defpsitilde}, $\bs m:=\Psi(0) + \bs r$ satisfies \[ \bs m=\Psi(0)+\tilde \Psi(\bs r) = \Psi(\Psi(0)+\bs r) =\Psi(\bs m). \]
The same arguments as in the end of the proof of Theorem \ref{th:conf} allow to conclude that $\bs v$ satisfies the integral formulation of \eqref{eq:sys_conf}. 
 
It remains to check the decay estimates. From \eqref{est:tPsi_v}, ${\bs v}-{\bs v_{L}}-\Psi(0) = {\bs r} = \tilde \Psi(\bs r) \in \q X_{\nu+\nu_0 ,t_0}^{t_1}$, which yields estimate \eqref{est:sol_conf_nu+nu_0}. Since $\nu_0\ge 0$, the combination of \eqref{est:psi(0)_exp} and \eqref{est:sol_conf_nu+nu_0} give \eqref{cvceinftysolnu}. Uniqueness follows from uniqueness of the fixed point. 
%
%
\end{proof}

\begin{proof}[Proof of Lemma \ref{lem:th_conf2_hyp}]
Notice that, due to the assumptions, $h_1(r) = O(r)$ as $r \to 0$.

Case (1) is straightforward as $\| \bs v_L \|_{\q Y_{s,0}^0} \le \| \bs v_0 \|_{H^s(\m S^{d-1})}$.

For case (2): 
Now recall \eqref{est:Hs_Yst0_exp_2d}
\[ \| {\bs v}_L \|_{\q Y_{s,t_{1}}^{t_{1}}} \le C e^{- \left( \ell_0 + \frac{d-2}{2} \right) t_1} \| v_0 \|_{H^s(\m S^{d-1})}. \]
Fix $t_1$ so large that $h_1 \left(C e^{- \left( \ell_0 + \frac{d-2}{2} \right) t_1} \| v_0 \|_{H^s(\m S^{d-1})} \right) \le \eta$.
Now, from the definition \eqref{def:Z} and \eqref{inegYtri}, any $w\in \q X_{\nu,0}$,
\begin{align} 
\| {\bs w} \|_{\q X_{\nu,t_0}^{t_1}} & = \sup_{t \ge t_0} e^{\nu (t-t_{1})}  \| {\bs w} \|_{\q Y_{s,t}^{t_{1}}}\le  \sup_{t \ge t_0} e^{\nu (t-t_{1})}  \| {\bs w} \|_{\q Y_{s,t}}\le e^{-\nu t_{1}} \sup_{t \ge 0} e^{\nu t}  \| {\bs w} \|_{\q Y_{s,t}} \\
& \le e^{-\nu t_{1}}\| \bs w \|_{\q X_{\nu,0}} , 
\end{align}
so that for $w = \Psi(0)$,
\[ e^{-\nu(t_0-t_1)} \| \Psi(0) \|_{\q X_{\nu,t_0}^{t_1}}  \le e^{- \nu t_0}\|  \Psi(0) \|_{\q X_{\nu,0}}. \]
Therefore, we can fix $t_0$ so large that 
\[ e^{- \nu t_0}\|  \Psi(0) \|_{\q X_{\nu,0}} \le \| {\bs v}_L \|_{\q Y_{s,t_{1}}^{t_{1}}}. \]
For this choice of $t_0 \ge t_1$,
Therefore,
\[ N \le \| {\bs v}_L \|_{\q Y_{s,t_{1}}^{t_{1}}} \le  C e^{- \left( \ell_0 + \frac{d-2}{2} \right) t_1} \| v_0 \|_{H^s(\m S^{d-1})}. \]
In the series defining $h$, the coefficients and and the exponents $\nu_\vartheta \ge \nu_0 \ge 0$  are non-negative: hence $h$ is non-decreasing in its first two variables. Therefore,
\[  h(CN, e^{-(t_0-t_1)}, e^{-t_1}) \le  h(C  e^{- \left( \ell_0 + \frac{d-2}{2} \right) t_1} \| v_0 \|_{H^s(\m S^{d-1})}, 1 , e^{-t_1})). \]
Finally, using Claim \ref{cl:h_h1} with $A = \ell_0 + \frac{d-2}{2}$, $\e >0$ and $D=0$, we can conclude
\begin{align*}
h(CN, e^{-(t_0-t_1)}, e^{-t_1}) & \le h_1(C e^{-\e t_1} \| v_0 \|_{H^s(\m S^{d-1})}) \le \eta. \qedhere 
\end{align*}
\end{proof}

%
%
%

We saw in Section \ref{s:actionder} and Lemma \ref{lmderiv} that the operator $D_i$ can be written $D_i =-y_i\left(\opD  - \frac{d-2}{2} \right) + \opR_{i} $ where $\opR_i$ has a better behavior. In particular, we will see in Lemma \ref{lmchgevarinfty} below, that after performing the conformal transform, we can also write the nonlinearity as
\[ f(u(x),\nabla u(x)) = g_{\opR}(t,y,v(t,y), (\partial_t -\opD ) v(t,y), \opR v(t,y)), \quad \text{where } e^t = |x|, \ y = \frac{x}{|x|}, \]
and we have written for short $\opR v(t,y)=(\opR_i v_j(t,y), \dots,\opR_d v(t,y))_{1 \le i \le d, 1 \le j \le N}$ (i.e it acts component by component like $\gh D$). So, we are interested in solving the system on $v = ( v_1, \dots, v_N)$, given by
\begin{equation} \label{eq:sys_confrefin}
\partial_{tt} v - \gh D^2 v =  g_{\opR} (t, y, v, (\partial_t -\opD ) v,\opR v)
\end{equation}
or equivalently,
\begin{align} 
\forall i \in \llbracket 1, N \rrbracket, \quad \partial_{tt} v_i - \gh D^2 v_i = g_{\opR,i} (t, y, v,  (\partial_t - \opD) v, \opR  v),
\end{align}
for a smooth function $g_{\opR}= (g_{\opR,1}, \dots, g_{\opR,N})$, and $t \in \m R$, $v \in \m R^N$, $\partial_t v = (\partial_t v_1, \dots, \partial_t v_N) $, $\opR  v = (\opR  v_1, \dots, \opR  v_N)$.

We make similar definitions for $g_{\opR}$ as in the previous section. The only difference will be that $w$ and $z$ are meant for different derivatives of $v$: $w_i$ will have the place of $(\partial_t-\opD) v_i$ and $z_{ij}$ that of $\opR_j v_i$
 so that $(z_{ij})_{1 \le j \le d}$ will correspond to $\opR v_i$.

We assume that $g_{\opR}$ admits an expansion as in \eqref{Def:g_series} with coefficients $b_{i,\alpha, \beta, \gamma, \delta, \opR}$ which satisfies the decay \eqref{estimaabetaiota} for some rates $\kappa_{\vartheta,\opR} \in \m R$.

We denote the new exponent $\nu_{\vartheta,\opR}$ that will play the same role as $\nu_{\vartheta}$ in this new context
\begin{equation} \label{def:nu_arefin}
    \nu_{\vartheta,\opR}=\kappa_{\vartheta}- |\alpha| + (|\beta|+ |\gamma| -1) \frac{d-2}{2} +|\delta|\frac{d}{2}.
\end{equation} 

\begin{remark}
 \label{rkxplaincoeffrefin}  
As we explained before in Remark \ref{rkxplaincoeff}, the definition of the exponent $\nu_{\vartheta,\opR}$ reflects the exponential decay rate given by any term $ b_{\vartheta} (t) y^\alpha v^\beta w^\gamma z^\delta$ which now corresponds (in the scalar case) to $ b_{\vartheta} (t) y^\alpha u(t)^\beta (\partial_t -\opD u)^\gamma (\opR_i u)^\delta_i$.
 
In parallel to  this new exponent can be explained by the following contributions:
\begin{itemize}
\item $\kappa_{\vartheta}$ is the exponential decay of the constant (in $y$ )$b_{\vartheta}$.
\item The loss $|\alpha|$ comes from the Lemma \ref{lm:polyYst}.
\item A gain $|\delta|$ comes from the gain of $e^{-t}$ for the action of $\opR_i$ described in Lemma \ref{lmderiv} (it is the main difference with the previous case of $\nu_{\vartheta}$.
\item We have a multiplication of $ |\beta|+ |\gamma| +|\delta|$ functions in $Y_{s-1,t}$, so, it creates an exponential gain of factor $ (|\beta|+ |\gamma| +|\delta|-1) \frac{d-2}{2} $ due to Lemma \ref{lmprodanalytique}.
\end{itemize}
Adding all the above yield the rate \eqref{def:nu_arefin}.
These exponents will be crucial in Lemma \ref{lminegG} below. 

We emphasize that, due to their better behavior, we \emph{gain a factor} $2|\delta|$ when using the operators $\opR_i$ instead of the operators $D_i$. 
\end{remark}

Denote
\begin{equation} \label{def:nu0refin}
\nu_{0,\opR} = \inf \{  \nu_{\vartheta,\opR} : \vartheta \in \Theta, \  B_{\vartheta} \ne 0 \}.
\end{equation}
We finally define \emph{mutatis mutandis} the series $h_{\opR}$ as in with $\nu_{\vartheta}$ and $\kappa_{\theta}$ replaced by $\nu_{\vartheta,\opR}$ and $\kappa_{\vartheta,\opR}$ respectively. We will always assume that for sufficiently small $\sigma, \rho$ and $\varsigma=1$, the series defining $h_{\opR}$ is convergent, in particular
\begin{equation} \label{hypokabcnurefin}
 \nu_{0,\opR} \ge 0.
\end{equation}

\begin{thm}[Conformal variables, refined] \label{thmconfrefin}For equation \eqref{eq:sys_confrefin}, the same results as Theorem \ref{th:conf} and \ref{th:conf2} holds with $\nu_{\vartheta}$, $\nu_{0}$ and $h$ replaced by $\nu_{\vartheta,\opR}$, $\nu_{0,\opR}$ and $h_{\opR}$.
\end{thm}

\bnp
The proof is exactly the same except in Lemma \ref{lminegG} where we have to estimate instead the following term
\begin{align*}
\MoveEqLeft \| \tilde v(\tau)^\beta (\dot{\tilde v}(\tau))-  \opD\tilde v(\tau)))^\gamma  (\opR \tilde v(\tau)) ^\delta - \tilde w(\tau)^\beta (\dot{\tilde w}(\tau)-\opD \tilde w(\tau))^\gamma (\opR \tilde w(\tau)) ^\delta  \|_{Y_{s-1,\tau-t_{1}}} \\
& \lesssim  (|\beta| + |\gamma| + |\delta|) C_0^{|\beta| + |\gamma| + |\delta|-1} e^{- ((|\beta|+ |\gamma| + |\delta|-1) \frac{d-2}{2} +|\delta| ) (\tau-t_{1})} \\
& \qquad \qquad \times (\| {\bs v} \|_{\q Y_{s,\tau}^{t_{1}}} + \| {\bs w} \|_{\q Y_{s,\tau}^{t_{1}}} +\| {\bs v}_L \|_{\q Y_{s,\tau}^{t_{1}}})^{|\beta|+|\gamma|+|\delta|-1} \| {\bs v} - {\bs w} \|_{\q Y_{s,\tau}^{t_{1}}} \\
& \lesssim  (|\beta| + |\gamma|+|\delta|)(C M(t))^{|\beta| + |\gamma|+|\delta|-1}  e^{- ((|\beta|+ |\gamma| -1) \frac{d-2}{2} + |\delta| \frac{d}{2} )(\tau-t_{1})} \| {\bs v} - {\bs w} \|_{\q Y_{s,\tau}^{t_{1}}}. 
\end{align*}
where we have used similarly the product estimate of Lemma \ref{lmprodanalytique} but we used instead the refined estimates for $\opR_i$ in Lemma \ref{lmderiv} that provide the gain $e^{-t}$ instead of the loss $e^{t}$ for $D_i$.

Once again, the key point is the rate in the $e^{ - (\tau-t_1)}$ factor, with the $|\delta| d/2$ exponent, instead of $|\delta| (d-4)/2$ in \eqref{est:mult_Psi}.
\enp

\section{Scattering close to infinity}
\label{s:scattinfty}


Let $\eta_0 >0$ and
\[ D(0,\eta_0) = \{ (u ,\varpi) \in \m R^{N} \times \mc M_{N,d}(\m R) : |u| , |\varpi| \le \eta_0 \} \]
be a small polydisc centered at $0$ in $\m R^{N} \times \mc M_{N,d}(\m R)$. We consider \[ f = (f_1, \dots, f_N): D(0,\eta_0) \to \m R^N \] an analytic function on $D(0,\eta_0)$: each of its component can be decomposed
\begin{align}
\label{deff} f_i(u,\varpi) = \sum_{(p,q) \in \m N^N \times \m N^{Nd} \setminus (0,0)} a_{i,p,q} u^p \varpi^q
\end{align}
For simplicity, we assume convergence up to the boundary, that is
\begin{gather} \label{est:fDSE}
\forall i \in \llbracket 1, N \rrbracket, \quad \sum_{p,q} |a_{i,p,q}| \eta_0^{|p|+|q|}  < +\infty,
\end{gather}
We consider the system on $u = (u_1, \dots, u_N)$ (defined on subsets of $\m R^d$), given by
\begin{align} 
\label{eq:sys_orig}
\Delta u= f (u, \nabla_{x} u)
\end{align}
that is, for all $i\in \llbracket 1, N \rrbracket$ by
\[ \Delta u_{i}= f_i (u, \nabla_{x} u). \]

Given $u_0 \in H^s(\m S^{d-1})$, recall that from \eqref{def:u_L}, there exist a unique solution $u_L \in \q Z^\infty_s$ of 
\begin{gather} \label{def:u_L2}
\Delta u_L =0 \quad \text{on} \quad \m R^d \setminus B(0,1) \quad \text{and} \quad u_L|_{\m S^{d-1}} = u_0.
\end{gather}

The goal in this paragraph is now to relate $u_L$ to a solution of the nonlinear system \eqref{eq:sys_orig}. For this, we will recast this question via the conformal transform and rely on the abstract result of the previous Section \ref{subsec:conf}.

We start by relating both equations, in the original variables and in conformal variables, recalling the definitions in Section \ref{s:actionder}.

\begin{lem}[Conformal change of variable close to infinity]\label{lmchgevarinfty}
Let $d\ge 2$, $R> 0$, and $f$ as in \eqref{deff}. We define the analytic functions $g$ and $g_{\opR}$ by
 \begin{align*} g(t,y,v,w,z) &=e^{\frac{d+2}{2} t} f \left(e^{-\frac{d-2}{2}t} v, e^{-\frac{d}{2} t} ( - \frac{d-2}{2} v \otimes y + w \otimes y +  z) \right)\\
  g_{\opR}(t,y,v,w,z) &=e^{\frac{d+2}{2} t} f \left(e^{-\frac{d-2}{2}t} v, e^{-\frac{d}{2} t} ( w \otimes y  +  z) \right)
 \end{align*}

where $t \ge 0$, $y \in \m R^d$, $v \in \m R^N$, $w \in \m R^N$, $z \in \mc M_{N,d}(\m R)$ (we will actually only evaluate $g$ for $y \in \m S^{d-1}$; $w$ and $z$ correspond to the time and space components of the gradient, respectively). 
 
For $u\in \mathscr C^{2}(\R^d\setminus B(0,R), \m R^N)$ with $\nor{u}{W^{1,\infty}(\R^d\setminus B(0,R))} \le \min(\rho,\sigma)$, we have the equivalence 
 \begin{itemize}
 \item
 $u$ solves the equation 
 \begin{align}
 \label{nonlinell}
\Delta u=f(u,\nabla u), \quad x\in \R^d\setminus B(0,R).
\end{align}
\item The map $v$ defined on $[\log(R),+\infty)\times \m S^{d-1}$ by
 $v(t,y)=e^{\frac{(d-2)t}{2}}u(e^ty)$, solves the equation
\begin{align*}
\partial_{tt} v - \gh D^2 v =g (t,y,  v, \partial_t v, \nabla_{y}  v), \quad (t,y)\in [\log(R),+\infty)\times \m S^{d-1}.
\end{align*}
\item The map $v$ defined on $[\log(R),+\infty)\times \m S^{d-1}$ by
 $v(t,y)=e^{\frac{(d-2)t}{2}}u(e^ty)$, solves the equation
\begin{align*}
\partial_{tt} v - \gh D^2 v =g_{\opR} (t,y,  v, (\partial_t-\opD ) v, \opR v), \quad (t,y)\in [\log(R),+\infty)\times \m S^{d-1}.
\end{align*}
\end{itemize}
Moreover, any other function $\tilde {g}$ so that 
\[ g(t,y,v,w,z)=\tilde {g}(t,y,v,w,z) \]
for all $(t,y,v,w)\in\m R \times \m S^{d-1} \times (\m R^N)^2$ and $z\in (T_y \m S^{d-1})^N$ satisfies the same property; this holds also for any function $\tilde {g}_{\opR}$ so that  
\[g_{\opR}(t,y,v,w,z)=\tilde {g}_{\opR}(t,y,v,w,z) \]
for all $(t,y,v,w)\in\m R \times \m S^{d-1} \times (\m R^N)^2$ and $z\in (\R^d)^N$.
\end{lem}
 
 \begin{remark}
 The point in mentioning $\widetilde{g}$ is to underline the fact that the function $g$ is only applied with $w=\nabla_{y}  v$ whose coordinate lie in $T_{y}\m S^{d-1}$. We could have defined $g$ only on $T\m S^{d-1}$ for the $(y,w)$ variable, but then the series expansion property would be not as tractable.
 This will be useful when the nonlinearity has some structure, see Theorem \ref{thminftyspecial} and its corollaries below.
 \end{remark}
 
\begin{example}
Harmonic maps from $\m R^d$ to the sphere $\m S^{N} \subset \m R^{N+1}$ solve
\[ \Delta u = u |\nabla u|^2. \]
We can write the system near the north pole $e_{N+1} = (0, \dots, 0, 1)$, and consider only the coordinates $u_1, \dots u_N$, taking into account that $u$ takes value in $\m S^{N}$ with
\[ u_{N+1} = \sqrt{1- \sum_{\ell=1}^N u_\ell^2}. \]
Then
\[ |\nabla u_{N+1}|^2 = \frac{1}{\sqrt{1- \sum_{\ell=1}^N u_\ell^2}} \sum_{m,n=1}^N u_m u_n \nabla u_m \cdot \nabla u_n, \]
so that for $i=1, \dots, N$, the corresponding nonlinearity writes
\begin{align*}
f_i(u,\varpi) & =  u_i \sum_{j=1}^{N+1} \sum_{k=1}^d \varpi_{j,k}^2 \\
& = u_i \left( \sum_{j=1}^{N} \sum_{k=1}^d \varpi_{j,k}^2 + \frac{1}{\sqrt{1- \sum_{\ell=1}^N u_\ell^2}} \sum_{m,n=1}^N u_m u_n \sum_{k=1}^d \varpi_{m,k} \varpi_{n,k}. \right) 
\end{align*}
This is the formulation of the Harmonic map equation in local coordinates, which usually contains Cristoffel symbols of the target manifold. From there, one derives the formula of $g_i$, for $i=1, \dots, N$.
\end{example}

 \bnp
 In order to avoid tedious notations, we only prove the result when $u$ is scalar.
 
We denote $t=\ln|x|$, $y=\frac{x}{|x|}\in \m S^{d-1}$. 
We recall the notation \eqref{def:Lambda u} $\Lambda u$ for the angular derivative.
 We will use similar definition for $(\Delta_{\mathbb{S}^{d-1}}u)(x)$.
  As 
 \begin{align*}
 \partial_t v(t,y) & = e^{\frac{d}{2} t} \left( \frac{d-2}{2} \frac{u}{r} +\frac{\partial u}{\partial r} \right)(e^t y), \\
 \nabla_y v(t,y) & = e^{\frac{d -2}{2} t} \Lambda u(e^t y).
 \end{align*}
 We have $u(x)  = |x|^{-\frac{d-2}{2}} v\left( \ln |x|, \frac{x}{|x|} \right)$ so that (as $\nabla_y v \perp y$)
 \begin{align*}
 \nabla u (x)&= \frac{x}{|x|}\frac{\partial u}{\partial r} (x) +\frac{1}{|x|} \Lambda u(x) = - \frac{d-2}{2|x|^{\frac{d+2}{2}}} x v + \frac{x}{|x|^{\frac{d+2}{2}}} \partial_t v + \frac{1}{|x|^{\frac{d}{2}}} \nabla_y v  \\
& =e^{-\frac{d}{2} t}  \left( - \frac{d-2}{2} yv(t,y) + y \partial_t v +  \nabla_y v \right).
 \end{align*}
 In particular, 
 \begin{align*}
 \partial_i u  &=e^{-\frac{d}{2} t}  \left( - \frac{d-2}{2} y^iv + y^i \partial_t v +  D_i v \right) \\
  &=e^{-\frac{d}{2} t}  \left( y^i(\partial_t v -\opD v)+\opR_i v \right).
 \end{align*}
 So, we have 
 \begin{align*} f(u,\nabla u)&= f \left(e^{-\frac{d-2}{2}t} v (e^t y), e^{-\frac{d}{2} t}  \left( - \frac{d-2}{2} yv + y \partial_t v +  \nabla_y v \right) \right)\\
 &= f \left(e^{-\frac{d-2}{2}t} v (e^t y), e^{-\frac{d}{2} t}  \left(  y (\partial_t v -\opD v)+  \opR v \right) \right).
  \end{align*}
Concerning the Laplacian, we compute
\begin{align*}
\MoveEqLeft \partial_{tt}v+\Delta_{\mathbb{S}^{d-1}}v\\
&= \left(\frac{d-2}{2}\right)^2 e^{\frac{(d-2)t}{2}}u(e^ty)+\left(2 \left(\frac{d-2}{2}\right)+1\right) e^{\frac{(d-2)t}{2}}e^{t}\frac{\partial u}{\partial r} (e^ty) \\
& \qquad +e^{\frac{(d-2)t}{2}} e^{2t}\frac{\partial^{2}u }{\partial r^{2}} (e^ty)+e^{\frac{(d-2)t}{2}}(\Delta_{\m S^{d-1}}u)(e^ty)\\ 
&=\left(\frac{d-2}{2}\right)^2 e^{\frac{(d-2)t}{2}}u(e^ty) \\
& \qquad +e^{\frac{(d+2)t}{2}}\left[\frac{d-1}{e^{t}}\frac{\partial u}{\partial r} (e^ty)+\frac{\partial^{2}u }{\partial r^{2}} (e^ty)+\frac{1}{e^{2t}}(\Delta_{\m S^{d-1}}u)(e^ty)\right]\\
&= \left(\frac{d-2}{2}\right)^2 v(t,y)+e^{\frac{(d+2)t}{2}}(\Delta u)(e^ty).
\end{align*}
That is $(\partial_{tt} v - \gh D^2 v)(t,y)=e^{\frac{(d+2)t}{2}}(\Delta u)(e^ty)$. Then, $u$ solves \eqref{nonlinell} if and only if 
\begin{align*}
\partial_{tt} v - \gh D^2 v & = e^{\frac{d+2}{2} t} f \left( e^{-\frac{d-2}{2}t} v, e^{-\frac{d}{2} t} \left(- \frac{d-2}{2} yv +  y \partial_t v +  \nabla_y v \right) \right) \\
& = g (t,y, v, \partial_t v, \nabla_{y} v)=g_{\opR} (t,y,  v, (\partial_t-\opD ) v, \opR v),
\end{align*}
with the chosen definition. This gives the first result for $g$ and $g_{\opR}$. 
 
For $\tilde g$, we only have to notice that $g(t,y,v,\partial_t v, \nabla_y v)$ and $\tilde g(t,y,v,\partial_t v, \nabla_y v)$ take the same value for all $(t,y)$, which is the case by the assumption, since $\nabla_{y} v\in T_{y} \m S^{d-1}=y^{\perp}$.
\enp

We now state our main results for \eqref{eq:sys_orig}: the first one relates to Theorem \ref{th:conf} and the second one to Theorem \ref{th:conf2}. For $f$ as in \eqref{deff}, the relevant exponent is
\begin{equation} \label{def:nu1}
\nu_1 := \inf \{ (d-2) (|p| + |q|)-d :  a_{i,p,q} \ne 0 \}
\end{equation}

\begin{thm}\label{thmexistglobP}Assume $d\ge 3$. Assume that ${f}$ as in \eqref{deff}, satisfies the supercriticality assumption
\begin{align}
\label{hypbpqinfty}
\nu_1 > 0.
\end{align}
(So that in fact $\nu_1 \ge 1$).
Let $u_{0} \in H^s(\m S^{d-1})$, and  $u_{L} \in \q Z^{\infty}_{s}$ given in \eqref{def:u_L2}.

Then, there exist $r_{0}\ge 1$ and a unique small $u \in \q Z^{\infty}_{s,r_{0}}$ solution of \eqref{eq:sys_orig} on  $\{|x|\ge r_{0}\}$ and such that
\begin{align}
\label{scattranslrd}
\nor{(u-u_{L})(r\cdot)}{Z^{\infty}_{s,r/r_{0}}}\lesssim r^{-\nu_1} \to 0 \quad \text{as} \quad r \to +\infty.
\end{align}
Moreover, the map $u_{0}\mapsto u$ is injective; and if $\nor{u_{0} }{H^s(\m S^{d-1})}$ is small enough, we can take $r_{0}=1$.
\end{thm}

Recall that from \eqref{est:Hs_Yst0} and \eqref{equivnorm}, $\|u_{L}(r\cdot) \|_{Z^{\infty}_{s,r/r_{0}}}$ remains bounded from below as $r \to +\infty$, so that \eqref{scattranslrd} gives the leading term of the expansion of $u$, and as a consequence, $u_L$ is unique.

\begin{thm}\label{thmexistglobPgain}
Let $d\ge 3$ and $\nu>0$. Assume that $f$ as in \eqref{deff} satisfies with the supercriticality assumption:
\begin{align}
\label{hypbpqinftynu} \nu_1 \ge 0.
\end{align}
Let ${u_{0}} \in H^s(\m S^{d-1})$, and ${u_{L}}\in \q Z_{s}^{\infty}$ given in \eqref{def:u_L2}, and assume that
\begin{equation} \label{est:f(uL)_exp}
\sup_{r\ge 1}r^{2 + \nu }\nor{f({u_{L}},\nabla {u_{L}})(r\cdot)}{  Z_{s -1 , r}^{\infty}}<+\infty.
\end{equation}
We also assume that at least one of the extra conditions holds true:
\begin{itemize}
    \item $\| u_0 \|_{H^s(\m S^{-1})}$ is small.
    \item $u_0$ has mean zero: $P_0 u_0=0$.
    \item for all $i \in \llbracket 1, N \rrbracket$ and $p \in \m N^N$, if $a_{i,p,0} \ne 0$, then $|p| > \frac{d}{d-2}$ (this is always fulfilled if $d \ge 5$).
\end{itemize}

Then, there exist $r_{0}\ge 1$ and a ${u}\in \q Z^{\infty}_{s,r_{0}}$ solution of \eqref{eq:sys_orig} on  $\{|x|\ge r_{0}\}$, and such that 
\begin{align}
\label{scattranslrdcrit}
\nor{({u}-{u_{L}})(r\cdot)}{Z^{\infty}_{s,r/r_{0}}}\lesssim r^{-\nu}\underset{r\to +\infty}{\longrightarrow}0.
\end{align}
This solution is unique among those such that 
\[ \sup_{r \ge r_0} r^\nu \left[  \| (u - u_L - u_{L,1}) (r \cdot)\|_{ Z_{s,r/r_0}^\infty} + \left\|  r \frac{\partial}{\partial r}  (u - u_L - u_{L,1}) (r \cdot) \right\|_{ Z_{s-1,r/r_0}^\infty} \right] \]
is small (where $u_{L,1} \in \q Z^\infty_{s,r_0}$ is the solution to $\Delta u_{L,1} = f(u_L, \nabla u_L)$ such that $\| u_{L,1}(r \cdot) \|_{Z_{s,r/r_0}^\infty} \to 0$ as $r \to +\infty$).

Moreover, the map ${u_{0}}\mapsto {u}$ is injective (when defined);  and if $\nor{{u_{0}} }{H^s(\m S^{d-1})}$ and the left-hand side of \eqref{est:f(uL)_exp} are small enough, we can take $r_{0}=1$.
\end{thm}

Note that the assumption $|p|>d/(d-2)$ seems to be necessary in some cases where the solutions deviate from the linear asymptotic we describe by a logarithm correction, see for instance \cite{V:81}.

Theorem \ref{thmexistglobP} will typically be used in the case of the semi linear elliptic equation with supercritical exponent, whereas Theorem \ref{thmexistglobPgain} is in order when considering the same equation with critical exponent. Theorem \ref{thmexistglobPintro} is a particular case of Theorem \ref{thmexistglobP} where $f(u,\nabla u)=f(u)$.

Sometimes, notably when derivative are involved, we can make use of extra structure in the nonlinearity which leads to appropriate cancellations. One example is given by the next result.

\begin{thm}
\label{thminftyspecial}
1) Assume $d \ge 3$ and that the coordinates of $f(u,\varpi)$ can be written in the form
\begin{align}
\label{deffstruct}
\sum_{(p,q) \in \m N^N \times \mc M_N(2 \m N) \setminus (0,0)} a_{i,p,q} u^p \Omega^{q/2}
\end{align} 
where $\Omega \in \mc M_N(\m R)$ is the matrix of the scalar products (in $\m R^d$): $\Omega_{j,k} = \varpi_{j}\cdot \varpi_{k}=\sum_{l=1}^d \varpi_{j,l} \varpi_{k,l}$, and we assume that the summability condition \eqref{est:fDSE} holds. The relevant exponent is now
\[ \nu_{1,\opR} : =\inf \{  (d-2)|p| + (d-1)|q|-d: a_{i,p,q} \ne 0 \}. \]
Then if $d\ge3$,

(a) The conclusion of Theorem \ref{thmexistglobP} holds replacing assumption \eqref{hypbpqinfty} by
$\nu_{1,\opR} > 0$.

(b) The conclusion of Theorem \ref{thmexistglobPgain} holds replacing assumption \eqref{hypbpqinftynu} by $\nu_{1,\opR} \ge 0$.

2) Assume $d=2$ and that the coordinates of $f(u, \varpi)$ can be written in the form
\begin{align}
\label{deffstructbracket}
\sum_{(p,q) \in \m N^N \times \mc M_N(2 \m N) \setminus (0,0)} u^p (a_{i,p,q} \Omega^{q/2} + b_{i,p,q} \Sigma^{q/2}),
\end{align} 
where $\Sigma \in \mc M_N(\m R)$ is the matrix of the symplectic products (in $\m R^2$): $\Sigma_{j,k} = \varpi_{j,1} \varpi_{k,2}-\varpi_{j,2} \varpi_{k,1}$ (and \eqref{est:fDSE} holds for the $(a_{i,p,q})_{p,q}$ and the $(b_{i,p,q})_{p,q}$). Define now
\[  \nu_{1,\opR} : =\inf \{  (d-2)|p| + (d-1)|q|-d: a_{i,p,q} \ne 0 \text{ or } b_{i,p,q} \ne 0 \}. \]
Then, if $u_0$ either small or with mean $0$, the same conclusions can be reached as in (a) and (b).
\end{thm}

Note that in \eqref{deffstruct}, the scalar product is in the base space $\R^d$, that is with respect to the the derivative variable while in \eqref{deffstructbracket}, the index $1$ and $2$ are with respect to the derivatives in the target space $\R^2$. In particular, $\varpi_{j,1} \varpi_{k,2}-\varpi_{j,2} \varpi_{k,1}$ corresponds to the Poisson bracket $\{ u_j,u_k\}=\partial_xu_j \partial_y u_k - \partial_y u_j\partial_x u_k$ where the running variable in $\R^2$ is $(x,y)$.

This kind of special structure on $f$ will be relevant for Harmonic maps, which precisely take the form \eqref{deffstruct} while the structure \eqref{deffstructbracket} is typical of the $H$-system.

\begin{remark}
It might also be possible to add some potential $V$ in some suitable space (like $V\in \q Z^{\infty}_{s}$). Yet, it seems at first sight that it would require the initial time to be large. Indeed, the fixed point that we perform for $t=0$ requires that the nonlinearity is of order at least $2$. 
\end{remark}

Theorems \ref{thmexistglobP} and \ref{thmexistglobPgain} rely on the applications of Theorem \ref{th:conf} and \ref{th:conf2} respectively. We make the proofs simultaneously since a large part of the argument is common to both situations.

\begin{proof}[Proof of Theorems \ref{thmexistglobP} and \ref{thmexistglobPgain}]
Here we do the proof for general $N$.
We check that concerning $(\nabla u)^q$ for $q\in \mc M_{N,d}(\m N)$, we have 
\[ \left(e^{-\frac{d}{2} t} (- \frac{d-2}{2} v \otimes y +  w \otimes y +  z )\right)^{q} = e^{- \frac{d}{2}|q|  t} P_q(y,v,w,z) \]
where $P_q$ is a polynomial, of partial degree $|q|$ in the variable $(v,w,z)$, and which can be written
\begin{align*}
P_q(y,v,w,z) & = \sum_{(\alpha,\iota,\gamma,\delta) \in J_q} c_{q,\alpha,\iota,\gamma,\delta} y^\alpha v^{\iota} w^\gamma z^\delta, 
\end{align*}
where 
\begin{align*}
J_q & = \left\{ (\alpha,\iota,\gamma,\delta) \in \m N^{d} \times \m N^N \times \m N^N \times \mc M_{N,d}(\m N) : |\alpha| + |\delta|  =q, |\iota| + |\gamma| = |\alpha|
\right\},
\end{align*}
and the coefficients $c_{q,\alpha,\beta,\gamma,\delta}$ are bounded by
\[ |c_{q,\alpha,\iota,\gamma,\delta}| \le (d/2+1)^{|q|}, \]
from the multinomial formula of Newton.
Notice that if $ (\alpha,\iota, \gamma, \delta) \in J_q$, then $|\alpha|, |\iota|, |\gamma|, |\delta| \le |q|$ so that we can bound the cardinal of $J_q$
\begin{equation} \label{est:card_Jq} 
 |J_q| \lesssim (|q|+1)^{d+2N +Nd}.
\end{equation}
(Notice that in the scalar case $N=1$, $\alpha$ and $\gamma$ determine $\iota$, but for general system, this is no longer the case).
Then using Fubini (to be justified by the following computations)
\begin{align*}
g_i(t,y,v,w,z)  & =e^{\frac{d+2}{2} t} 
 \sum_{p,q} a_{i,p,q} e^{- \left(\frac{d-2}{2} |p| + \frac{d}{2} |q| \right) t} v^p P_q(y,v,w,z)\\
&=\sum_{p,q}\sum_{(\alpha, \iota, \gamma, \delta) \in J_q } a_{i,p,q} c_{q, \alpha, \iota, \gamma, \delta} e^{- \left(\frac{d-2}{2} |p| + \frac{d}{2} |q|-\frac{d+2}{2} \right) t} y^\alpha v^{p+\iota}   w^\gamma z^\delta\\
&=\sum_{\vartheta \in \Theta}  b_{\vartheta}(t) y^\alpha v^\beta w^\gamma z^\delta,
\end{align*}
where (recall we denote $\vartheta= (\alpha, \beta, \gamma, \delta,\iota)$)
\begin{align}
\label{formulaaabc}
b_{\vartheta}(t) & =  \sum_{(p,q) \in I_{\vartheta}} a_{i,p,q} c_{q,\alpha,\iota,\gamma,\delta} e^{- \left(\frac{d-2}{2} p + \frac{d}{2} |q|-\frac{d+2}{2} \right) t}, \\
 \text{with} \quad I_{\vartheta} & = \{ (p,q) \in \m N \times \m N^d : (\alpha,\iota,\gamma,\delta) \in J_q, \beta = p + \iota \}.
\end{align}
Let $\vartheta \in \Theta$ such that $b_\vartheta \ne 0$, so that $I_\vartheta \ne \varnothing$. For any $(p,q) \in  I_{\vartheta}$ then 
\begin{equation} \label{def:pq_ib}
|p| = |\beta| - |\iota| \quad \text{ and } \quad |q| = |\iota| + |\gamma| + |\delta| \quad \text{are prescribed}, 
\end{equation}
 so that the rate in the exponential as well and can be expressed in terms of $\vartheta$:
\begin{align}
\frac{d-2}{2} |p| + \frac{d}{2} |q|-\frac{d+2}{2} & = \frac{d-2}{2} (|\beta| - |\iota|) + \frac{d}{2} (|\iota| + |\gamma| + |\delta|)-\frac{d+2}{2} =: \kappa_{\vartheta}. \label{kthetacasgen}
\end{align}
Also notice that $|\beta| + |\gamma| + |\delta| = |p| + |q|$. Denoting 
\begin{align}
\label{defAabc}
B_{\vartheta} & := (d/2+1)^{|\iota|+|\gamma| + |\delta|}\sum_{(p,q) \in  I_{\vartheta}} \max_{i \in \llbracket 1, N \rrbracket} |a_{i,p,q}|,
\end{align}
($B_{\vartheta} =0$ if $I_{\vartheta}$ is empty), we have the expected estimate \eqref{estimaabetaiota}.
We can then express $\nu_{\vartheta}$ in terms of $\vartheta$, without involving $\iota$. Note that if $B_{\vartheta}\neq 0$ then $J_q\neq \varnothing$ and so $|\iota| + |\gamma| = |\alpha|$.  From \eqref{def:nu_a},
\begin{align} 
\nu_{\vartheta} & = \kappa_{\vartheta} - |\alpha| + (|\beta| + |\gamma| -1) \frac{d-2}{2} + |\delta| \frac{d-4}{2} \nonumber \\
&=(d-2)(|\beta|+|\delta|)+ (d-1)| \gamma| +|\iota|-|\alpha|-d \\
&=(d-2)(|\beta|+|\delta|+|\gamma|)-d.
\label{egal:nu}
\end{align}
In particular, since $(p,q) \in  I_{\vartheta}$, by \eqref{def:pq_ib}, we have:
\begin{align} 
\nu_{\vartheta} &= (d-2) (|p| + |q|) -d\ge \nu_1.
\label{cond:nu}
\end{align}

As this is true for any $\vartheta \in \Theta$, we infer $\nu_0 \ge \nu_1$ (recall that $\nu_0$ is defined in \eqref{def:nu0}). Under any of the assumptions \eqref{hypbpqinfty} or \eqref{hypbpqinftynu}, we see that  \eqref{hypokabcnu} is fulfilled. Also $|\beta|+|\gamma|+|\delta| = |p|+|q| \ge 1$ so that \eqref{hypokabc} holds as well.

Under the supercriticality assumption \eqref{hypbpqinfty}, and as only integers are involved, we infer $\nu_0 \ge 1 >0$: this means that under the assumptions Theorem \ref{thmexistglobP}, the first hypothesis of Theorem \ref{th:conf} holds.

We now turn to the convergence of the series defining $h$. Recalling \eqref{defhh1}, we have for $ 0 \le \sigma \le 1$:
\begin{align*}
h_1(\sigma) & = \sum_{\vartheta \in \Theta} B_{\vartheta} \langle \alpha \rangle^{s+1}(|\beta| + |\gamma| + | \delta |) \sigma^{|\beta|+|\gamma| + |\delta|-1} \\
& \lesssim \sum_{p,q} \max_{i \in \llbracket 1, N \rrbracket} |a_{i,p,q}| R_{p,q}(\sigma) \quad \text{where} \\
R_{p,q}(\sigma) & := \sum_{\substack{\vartheta \in \Theta \\ (\alpha, \iota, \gamma,\delta ) \in J_q, p = \beta-\iota}} (d/2+1)^{|\iota|+|\gamma|+|\delta|} \langle \alpha \rangle^{s+1} (|\beta| + |\gamma| + | \delta |) \sigma^{|\beta|+|\gamma|+|\delta|-1} .
\end{align*}
Given $p,q$, in the sum defining $R_{p,q}(\sigma)$, we have $|p|+|q| = |\beta| + |\gamma| + |\delta|$, and $|\alpha| \le |q|$. Also the cardinal of the indexation set is $|J_q|$ (because $\beta$ is prescribed by $p$ and $\iota$. Hence, using \eqref{est:card_Jq}, there hold, for $0 \le \sigma \le 1$
\begin{align*}
R_{p,q}(\sigma) & \le  (d/2+1)^{|q|} (|p|+|q|+1)^{s+2} \sigma^{|p|+|q|-1} \cdot |J_q| \\
& \lesssim (d/2+1)^{|q|} (|p|+|q|+1)^{s+ (N+1)(d+2)} \sigma^{|p|+|q|-1}.
\end{align*}
Since $(|p|+|q|+1)^{s+(N+1)(d+2)}\le C_{s,N,d} 2^{|p|+|q|-1}$ for some large constant $C_{s,N,d}$, and all $p,q$, we conclude
\[ h_1(\sigma)\lesssim  \sum_{p,q} \max_{i \in \llbracket 1, N \rrbracket} |a_{i,p,q}| ((d+2) \sigma)^{|p|+|q|-1}. \]
By the assumption on $f$ \eqref{est:fDSE}, this is convergent for $\sigma$ small enough, and $h_1$ is well defined for small $\sigma$.

Due to \eqref{hypbpqinftynu}, if $a_{p,q} \ne 0$, then $(d-2) (p + |q|)-d \ge 0$ and in particular $p+|q| \ge 2$, and so $h_1(0)=0$.

\bigskip

For Theorem \ref{thmexistglobP}, in the general case, we want to apply Lemma \ref{lem:conf}, assumption (2) (in the specific case where $\nor{u_{0} }{H^s(\m S^{d-1})}$ is small and $r_{0}=1$, we use Lemma \ref{lem:conf} Case (1)). So it suffices to check the lower bound \eqref{infnuvartheta} on $\nu_{\vartheta}$: in view of \eqref{egal:nu} and as $d \ge 3$, we can choose $\e = d-2 >0$ and $D = 2$. 


This shows that under the conditions of Theorem \ref{thmexistglobP}, we can apply Theorem \ref{th:conf}, and construct the solution $\bs v$ in conformal variables. Letting $u$ be the conformal inverse and $r_{0}=e^{t_{0}}$, $u$ is in the appropriate space thanks to \eqref{equivnorm}. This also gives the convergence in the original variables. Theorem \ref{thmexistglobP} is proven.

\bigskip

Regarding Theorem \ref{thmexistglobPgain}, we want to apply Theorem \ref{th:conf2}:  it remains to check \eqref{est:psi(0)_exp} and \eqref{est:h1_small}. Recall the assumption \eqref{est:f(uL)_exp}: written in conformal variable, we see  that it implies (due to \eqref{equiYZ})
\begin{equation} \label{est:g_exp}
\nor{g(t,y,v_{L},\partial_t v_L, \nabla_y {v_{L}})}{Y_{s-1,t}} = e^{- \frac{d-2}{2} t} \| r^{ \frac{d+2}{2} } f(u_L, \nabla u_L) \|_{Z_{s-1,e^t}^\infty} \lesssim e^{-\nu t},
\end{equation}
and so
 using Lemma \ref{lmDuhamelinfty}, we infer that for $t\ge 0$,
\begin{align*}
\| \Psi(0) \|_{\q Y_{s,t}^{0}}&=\| \Phi(g(t,y,{v_{L}},\partial_t v_L, \nabla_y{v_{L}})) \|_{\q Y_{s,t}^{0}} \\
&\lesssim \int_{t}^{+\infty} (1+\tau-t) \| g(t,y,{v_{L}},\nabla_{t,y}{u_{L}})(\tau) \|_{Y_{s-1,\tau-t}}  d\tau\\
&  \lesssim \int_{t}^{+\infty} (1+\tau-t)e^{-\nu (\tau - t)} d\tau\lesssim e^{-\nu t}.
\end{align*}
Taking the supremum in $t \ge 0$, we infer that $\Psi(0) \in \q X_{\nu,0}$ and \eqref{est:psi(0)_exp} is satisfied. 

We now focus on \eqref{est:h1_small}. 
If $\| u_0 \|_{H^s(\m S^{d-1})}$ is small, we can apply directly the case (1) of  Lemma \eqref{lem:th_conf2_hyp}.

Otherwise, we want to apply  Lemma \eqref{lem:th_conf2_hyp}, Case (2), with $\ell_0=0$ or $1$: so we are to check \eqref{infkvartheta} with $D=0$.

Let $\vartheta \in \Theta$ such that $B_\vartheta \ne 0$. Let $(p,q) \in I_\vartheta$ such that for some $i$, $a_{i,p,q} \ne 0$, then we can express from the definition \eqref{kthetacasgen}
\begin{multline}
\label{kvartheta+e}
\kappa_{\vartheta} +\left(\ell_0+\frac{d-2}{2}-\e\right)(|\beta|+ |\gamma|+|\delta|-1) \\
=(\ell_0+d-2-\e)(|p|+|q|-1)+|q|-2
\end{multline}
By assumption $\nu_1 \ge 0$, so that $|p|+|q|\geq \frac{d}{d-2}$.
Therefore  the above term \eqref{kvartheta+e} is bounded by below by
\begin{align}
(\ell_0+d-2-\e)\frac{2}{d-2}+|q|-2=\frac{2(\ell_0-\e)}{d-2}+|q|.
\end{align}
If $|q|\ge 1$ or $\ell_0\ge 1$, it can be made positive for $\e = 1/2$.
If $|q|=0$ and $\ell_0=0$, the term in \eqref{kvartheta+e} writes $(d-2-\e) (|p|-1) -2$. As $|p| \ge 2$, this is positive if $d \ge 5$.
Otherwise, our assumption implies that $|p| \ge \frac{d}{d-2} +1$ and so 
\[ (d-2-\e) (|p|-1) -2 \ge d-2 + \frac{d\e}{d-2} \ge 1 - d\e, \]
which is non-negative for $\e =1/d>0$.

Due to case (2) of Lemma \eqref{lem:th_conf2_hyp}, there exist $t_0 \ge t_1 \ge 0$ (large) such that \eqref{est:psi(0)_exp} holds and we are in a position to apply Theorem \ref{th:conf2}: we obtain a suitable conformal solution $\bs v \in \q Y^{t_1}_{s,t_0}$. As this space is included in $\q Y^{t_0}_{s,t_0}$, undoing the conformal transform yields a desired solution $u \in \q Z_{s, r_0}^\infty$ where $r_0 = \ln(t_0)$. This proves Theorem \ref{thmexistglobPgain}.
\end{proof}

\begin{proof}[Proof of Theorem \ref{thminftyspecial}]
Here, due to cancellations specific to the vectorial nature of the equations, we also do the computations for general $N \ge 1$.

Let us compute the term $\varpi_{i}\cdot \varpi_{j}$ according to the change of variable of Lemma \ref{lmchgevarinfty}, for $i,j \in \llbracket 1, N \rrbracket$): we obtain (for the term corresponding to $g_{\opR}$)
\begin{align}
\MoveEqLeft \sum_{k=1}^{d}e^{-d t} ( y_{k} w_{i} +  z_{ik}) (  y_{k} w_{j} + z_{jk})\\
& = e^{-d t} \left( w_{i}w_{j}+ w_i \sum_{k=1}^{d}  y_k z_{jk}+w_j \sum_{k=1}^{d}  y_k z_{ik}+ \sum_{k=1}^{d} z_{ik} z_{jk}\right).
\end{align}
We used that $\sum_{k=1}^{d} y_{k}^{2}=1$ since $y\in \m S^{d-1}$.

In dimension $2$, we compute $\varpi_{1,i} \varpi_{2,j}-\varpi_{2,i} \varpi_{1,j}$ according to the change of variable of Lemma \ref{lmchgevarinfty}. We obtain 
\begin{align}
\MoveEqLeft e^{-d t}\left(( y_{1} w_{i} +  z_{i1}) (  y_{2} w_{j} + z_{j2})-( y_{2} w_{i} +  z_{i2}) (  y_{1} w_{j} + z_{j1})\right)\\
& = e^{-d t}\left( y_{1} w_{i} z_{j2}+  z_{i1} (  y_{2} w_{j} + z_{j2})- y_{2} w_{i}z_{j1} -  z_{i2} (  y_{1} w_{j} + z_{j1})\right).
\end{align}


Observe that in both cases, terms are at most linear in $y$ and those where $y$ appear also carry a $z$ factor: hence, for any contributing $b_\vartheta$, there hold
\[ |\alpha| \le |\delta|. \] 
Also, there is no $v$ involved in any of the expressions, so that the index $\iota$ is not useful anymore, and we drop it for the rest of the computations.

We can now argue as in the previous proof of Theorems \ref{thmexistglobP} and \ref{thmexistglobPgain}.

Regarding $u^p(\Omega)^{q/2}$ for $q\in \mc M_N(2 \m N)$: denote
\[ m_{i,j}(y,w,z) :=  w_{i}w_{j}+ w_i \sum_{k=1}^{d}  y_k z_{jk}+w_j \sum_{k=1}^{d}  y_k z_{ik}+ \sum_{k=1}^{d} z_{ik} z_{jk}, \]
the contribution where we erased the $e^{-dt}$ factor.
$m(y,w,z)^{q/2} $ is a polynomial, of partial degree $q$ in the variable $(w,z)$. Therefore, $u^p \Omega^{q/2}$ can be written
\begin{align*}
e^{- \left(\frac{d-2}{2} |p| + \frac{d}{2} |q|-\frac{d+2}{2} \right) t} \sum_{(\alpha, \gamma,\delta) \in J_q} c_{\alpha,\gamma,\delta} y^\alpha v^p  w^\gamma z^\delta, 
\end{align*}
where 
\begin{align*}
J_{q} & = \left\{ (\alpha,\gamma,\delta) \in \m N^{d} \times \m N^N  \times \mc M_{N,d}(\m N) : |\gamma | + |\delta| =|q|, |\alpha| \le |\delta|
\right\}.
\end{align*}
Let $\vartheta = (\alpha,\beta, \gamma, \delta) \in \Theta$ such that $(\alpha, \gamma,\delta) \in J_{q}$ and $\beta =p$, the corresponding
\[ \kappa_{\vartheta,\opR} = \frac{d-2}{2} |p| + \frac{d}{2} |q|-\frac{d+2}{2} = \frac{d-2}{2} |\beta| + \frac{d}{2} (|\gamma | + |\delta|)-\frac{d+2}{2}, \]
and so (see Remark \ref{rkxplaincoeffrefin})
\begin{align} 
\nu_{\vartheta,\gh R} & = \kappa_{\vartheta} - |\alpha| + (|\beta|+ |\gamma|-1) \frac{d-2}{2} + |\delta| \frac{d}{2} \nonumber \\
&= (d-2)|\beta|+(d-1)|(|\gamma|+|\delta|)+|\delta|-|\alpha|-d\\
&\ge (d-2)|\beta|+(d-1)(|\gamma|+|\delta|)-d\\
& \ge  (d-2) |p|+ (d-1) |q| -d. \label{cond:nurefin}
\end{align}
Thus, we obtained $\nu_{0,\opR} \ge \nu_{1,\opR}$, for both cases $(1)$ or $(2)$.

One can then reproduce the end of the proofs of Theorem \ref{thmexistglobP} and \ref{thmexistglobPgain} using instead the refinement given by Theorem \ref{thmconfrefin}. Now, we want to check that some assumptions that allowed to apply Lemma \ref{lem:conf} or \ref{lem:th_conf2_hyp} can be fulfilled.

First, note that the assumptions also imply $|p|+|q|\ge 2$ if $a_{i,p,q}\ne 0$ or $b_{i,p,q} \ne 0$, so that we will also have $h_1(0)=0$. In particular, if we assume smallness, the Case (1) of Lemma \ref{lem:conf} or \ref{lem:th_conf2_hyp} is applicable.

For the equivalent of Theorem \ref{thmexistglobP}, we have $\nu_{\vartheta,\gh R} \ge (d-2)(|\beta|+\gamma|+|\delta)-d$, so \eqref{infnuvartheta} is satisfied with $\e=d-2>0$ since $d\ge 3$. So, we apply Lemma \ref{lem:conf} Case (2).

For the equivalent of Theorem \ref{thmexistglobPgain}, we want to verify Case (2) of Lemma \ref{lem:th_conf2_hyp}. The computation \eqref{kvartheta+e} still holds for $\kappa_{\vartheta,\opR}$. We denote that under both cases $d\ge 3$ or $d=2$, we always have $\ell_0+d-2\ge 1$. In particular, if $\e<1$, the expression in \eqref{kvartheta+e} is positive as soon as $|q|\ge2$. Moreover, under the structural assumptions of the nonlinearity we made, we have either $q=0$ or $q\ge 2$. So, it only remains to check the case $q=0$. It can only happen for $d\ge 3$, for which the same analysis as in Theorem \ref{thmexistglobP} works, under the same assumptions. We conclude the proof similarly.
\end{proof}
\begin{remark}
  In what follows, the refined Theorem \ref{thmexistglobPgain} will be applied in cases where the better behavior of the first Duhamel estimate \eqref{est:f(uL)_exp} will be valid for any $u_0$, as seen in Section \ref{s:null2} for the case of derivative nonlinearities satisfying some null condition. Yet, there can be some cases where the better behavior of the Duhamel term is due to $u_0$. For instance, in dimension $d=3$, if we consider polynomial type nonlinearities as in Theorem \ref{thmexistglobPintro}, the critical exponent is given by the condition $\nu=(d-2)p-d\ge 0$, that is $p\ge 3$. Consider the system 
  \begin{equation}
\begin{cases}
\Delta u = u^5+v^2 u, \\
\Delta v=v^5 +u^2v
\end{cases}
\end{equation}
It is critical for our criterium  because of the cubic coupling nonlinearity. Yet, if we consider some "asymptotic datum" of the form $u_0=0$, $v_0\in H^s(\m S^{d-1})$ then, \eqref{est:f(uL)_exp} will be satisfied because $f(u_L, v_L)$ will be of the form $(0,v_L^5)$ which has a good behaviour since the power $5$ is supercritical.
\end{remark}

\section{Scattering close to zero}
For problem close to zero, each one of the theorems in Section \ref{s:scattinfty} has its counterpart close to $0$. We only write those corresponding to the applications we have in mind, and omit the proofs unless when they are not as in the scattering close to infinity.

We start with the equivalent of Lemma \ref{lmchgevarinfty} but close to zero.  The results look very similar except for the sign of the exponent terms. This will not change the ideas of the proofs, but it does impact the numerology. In particular, our results close to zero require some ``small'' degree of monomials for the nonlinearity while the results close to infinity required some ``large'' degree.

\begin{lem}[Conformal change of variable close to zero]\label{lmchgevarzero}Let $d\ge 2$ and $R> 0$.  Let $D(0,(\rho,\sigma)) = \{ (w,z) = (w,z_1,\dots,z_d) \in \m R^{1+d} : |w| \le \rho, |z| \le \sigma \}$ be a polydisc in $\m R^{1+d}$ and $f: D(0,(\rho,\sigma)) \to \m R^N$ be a smooth function on $D(0,(\rho,\sigma))$. 
  
We define the smooth function $g$ by 
\begin{align*} g(t,y,v,w,z) &=e^{-\frac{d+2}{2} t} f \left(e^{\frac{d-2}{2}t} v, e^{\frac{d}{2} t} (-  \frac{d-2}{2}v\otimes  y - w\otimes y +  z) \right)\\
g_{\opR}(t,y,v,w,z) &=e^{-\frac{d+2}{2} t} f \left(e^{\frac{d-2}{2}t} v, e^{\frac{d}{2} t} ( -w \otimes y  +  z) \right)
\end{align*}
 
 where $t \ge 0$, $y \in \m S^{d-1}$, $v \in \m R^N$, $w \in \m R^N$, $z \in \mc M_{N,d}(\m R)$.

  For $u\in \q C^{2}(B(0,R))$ with $\nor{u}{W^{1,\infty}(B(0,R))}\le \min(\rho,\sigma)$, we have the equivalence 
 \begin{itemize}
 \item
 $u$ solves the equation 
 \begin{align}
 \label{nonlinell0}
\Delta u=f(u,\nabla u), \quad x\in B(0,R)
\end{align}
\item 
 $v(t,y)=e^{-\frac{(d-2)}{2}t}u(e^{-t}y)$, $v$ solves the equation
\begin{align*}
\partial_{tt} v - \gh D^2 v =g (t,y, v, \partial_t v,\nabla_{y} v), \quad (t,y)\in [-\log(R),+\infty)\times \m S^{d-1}.
\end{align*}
\item 
 $v(t,y)=e^{-\frac{(d-2)}{2}t}u(e^{-t}y)$, $v$ solves the equation
\begin{align*}
\partial_{tt} v - \gh D^2 v =g_{\opR} (t,y, v, \partial_t v+\opD v,\opR v), \quad (t,y)\in [-\log(R),+\infty)\times \m S^{d-1}.
\end{align*}
\end{itemize}
Moreover, any other function $\widetilde{g}$ so that $g(t,y,v,w,z)=\widetilde{g}(t,y,v,w,z)$ for all $(t,y,v,w)\in\m R \times \m S^{d-1} \times (\m R^N)^2$ and $z\in (T_{y}\m S^{d-1})^N$ satisfies the same property. 

The same also holds for any function $\tilde {g}_{\opR}$ so that  
 $g_{\opR}(t,y,v,w,z)=\tilde {g}_{\opR}(t,y,v,w,z)$ for all $(t,y,v,w,z)\in\m R \times \m S^{d-1} \times (\m R^N)^2 \times  \mc M_{N,d}(\m R)$.
 \end{lem}
 
\bnp
The computations are mostly the same as in Lemma \ref{lmchgevarinfty}, up to a few changes of exponent and signs, changing $t$ to $-t$, that is $t=-\ln|x|$, $y=\frac{x}{|x|}\in \m S^{d-1}$. For instance, as we have instead $u(x)  = |x|^{-\frac{d-2}{2}} v\left( -\ln |x|, \frac{x}{|x|} \right)$ so that (as $\nabla_y v \perp y$)
 \begin{align*}
 \nabla u & = - \frac{d-2}{2|x|^{\frac{d+2}{2}}} x v - \frac{x}{|x|^{\frac{d+2}{2}}} \partial_t v + \frac{1}{|x|^{\frac{d}{2}}} \nabla_y v  \\
& =e^{\frac{d}{2} t}  \left( - \frac{d-2}{2} yv(t,y) - y \partial_t v +  \nabla_y v \right)
 \end{align*}
 So, we have $f(u,\nabla u)= f \left(e^{\frac{d-2}{2}t} v, e^{\frac{d}{2} t}  \left( - \frac{d-2}{2} yv - y \partial_t v +  \nabla_y v \right) \right) $.

Concerning the Laplacian, changing $t$ to $-t$ in the formula of Lemma \ref{lmchgevarinfty}, we still have $(\partial_{tt} v - \gh D^2 v)(t,y)=e^{-\frac{(d+2)t}{2}}(\Delta u)(e^{-t}y)$. 
\enp

Here is the result without extra structure (analog to Theorem \ref{thmexistglobP}).

\begin{thm}\label{thmexistglobP0}
Assume $d\ge 2$ and $f$ as in \eqref{deff} satisfies the subcriticality assumption
\[ a_{i,p,q} \ne 0 \Longrightarrow q=0. \]
Let $u_{0} \in H^s(\m S^{d-1})$, and denote $u_{L}\in \q Z_{s}^{0}$ the associated bounded solution of 
\begin{align}\label{eqnlin0} \Delta {u_{L}}=0 \quad \text{on } B(0,1) \quad \text{and} \quad u_{L}|_{\m S^{d-1}}= u_{0}. \end{align}
1) If either $f$ is polynomial, or $u_0$ has zero mean, then, there exist $r_{0}\le 1$ and a unique $u \in \q Z^{0}_{s,r_{0}}$ solution of \eqref{eq:sys_orig} on  $B(0,r_{0})\setminus \{0\}$ so that 
\begin{align}
\label{scattranslrd0}
\nor{(u- u_{L})(r\cdot)}{Z^{0}_{s,r/r_{0}}}\lesssim r^{2} \to 0 \quad \text{as} \quad r \to 0.
\end{align}
Moreover, the map $u_{0}\mapsto u$ is injective. 

2) If instead $f$ satisfies ($a_{i,p,q} \ne 0 \Longrightarrow |p|\ge 2$) and $\nor{u_{0} }{H^s(\m S^{d-1})}$ is small enough, the conclusion of 1) holds with $r_{0}=1$.
\end{thm}

Here is the results where the first iterate is better behaved, when the nonlinearity $f$ has structure (corresponding to Theorem \ref{thminftyspecial}).

\begin{thm}\label{thmexistglobPgainzero}
Let $d \ge 2$, $\nu >0$ and assume $f$ has the structure as in Theorem \ref{thminftyspecial}, that is \eqref{deffstruct} and, if $d=2$, \eqref{deffstructbracket}, with the summability condition \eqref{est:fDSE}. Also assume that $f$ as in \eqref{deff} satisfies the subcriticality assumption:
\[ a_{i,p,q} \ne 0 \Longrightarrow   |q| \le 2. \]
Let $u_{0} \in H^s(\m S^{d-1})$, denote $u_{L} \in \q Z_{s}^{0}$ the associated bounded solution of \eqref{eqnlin0}.
We assume furthermore that
\begin{equation} \label{est:f(uL)_exp0}
\sup_{0 < r\le 1}r^{2 - \nu }\nor{f({u_{L}},\nabla {u_{L}})(r\cdot)}{  Z_{s -1 , r}^{0}}<+\infty. 
\end{equation}

1) If $u_0$ has zero mean, then, there exist $r_{0}\le 1$ and a unique $u\in \q Z^{0}_{s,r_{0}}$ solution of \eqref{eq:sys_orig} on  $B(0,r_{0})\setminus \{0\}$ so that 
\begin{align}
\label{scattranslrd0nu}
\nor{(u-u_{L})(r\cdot)}{Z^{0}_{s,r}}\lesssim r^{\nu} \to 0 \quad \text{as} \quad r \to 0.
\end{align}
Moreover, the application ${u_{0}}\mapsto {u}$ is injective (where defined). 

2) If instead $f$ satisfies 
\[ a_{i,p,q} \ne 0 \Longrightarrow |p|+|q| \ge 2, \]
and $\nor{{u_{0}} }{H^s(\m S^{d-1})}$ and the right-hand side of \eqref{est:f(uL)_exp0} are small enough, the conclusion of 1) holds with $r_{0}=1$.
\end{thm}

\begin{remark}
    The extra assumption that $u_0$ has zero mean is natural in this context: the associated linear solution converges to this value at zero and we need to be sure that $f$ is well defined there. It can be easily removed by transforming $f$ by $f(\cdot+x_0)$. It was not necessary in the problem at infinity, since all linear (bounded) solutions converge to zero at infinity, except for $d=2$.
\end{remark}

\begin{proof}[Proof of Theorems \ref{thmexistglobP0} and \ref{thmexistglobPgainzero}]

The computations are similar to that in the proofs of Theorems \ref{thmexistglobP} and \ref{thmexistglobPgain} (we use their notations): the gain and losses explained in Remarks \ref{rkxplaincoeff} and \ref{rkxplaincoeffrefin} (respectively) are the same, and the exponent in $e^t$ in front of the monomial, is now the opposite  $\ds - \frac{d+2}{2} + \frac{d-2}{2} |p| + \frac{d}{2} |q|$.

More precisely, due to Lemma \ref{lmchgevarzero} $g$ now takes the following form:
\begin{align}
g_i(t,y,v,w,z) & =e^{-\frac{d+2}{2} t} 
 \sum_{p,q} a_{p,q} e^{ \left(\frac{d-2}{2} |p| + \frac{d}{2} |q| \right) t} v^p P_q(y,v,-w,z)\\
&=\sum_{p,q}\sum_{(\alpha,\iota,\gamma,\delta) \in J_q} a_{i,p,q} c_{\alpha,\iota,\gamma,\delta} (-1)^{\gamma}e^{- \left(\frac{d+2}{2}-\frac{d-2}{2} p - \frac{d}{2} |q| \right) t} y^\alpha v^{p+\iota}   w^\gamma z^{\delta}\\
&=\sum_{\vartheta \in \Theta}  b_{i,\vartheta}(t) y^\alpha v^\beta w^\gamma z^{\delta},  \label{def:g_zero}
\end{align}
where $\vartheta = (\alpha,\beta,\gamma,\delta,\iota)$ and
\[ b_{i,\vartheta}(t) = (-1)^{\gamma} c_{\alpha,\iota,\gamma, \delta} \sum_{(p,q) \in I_{\vartheta}} a_{i,p,q} e^{- \left(\frac{d+2}{2}-\frac{d-2}{2} |p| - \frac{d}{2} |q| \right)t }. \]
As before, $p = \beta-\iota$ and $|q| = |\iota| + |\gamma| + |\delta|$ are prescribed in $I_{\vartheta}$, so for any $(p,q) \in I_{\vartheta}$
\begin{align} \label{formkappa0} \frac{d+2}{2}-\frac{d-2}{2} |p| - \frac{d}{2} |q| = \frac{d+2}{2}-\frac{d-2}{2} (|\beta|-|\iota|) - \frac{d}{2}(|\iota|+|\gamma|+|\delta|) =: \kappa_{\vartheta}.
\end{align}
We also define
\[ B_{\vartheta} = (d/2+1)^{|\iota|+|\gamma|+|\delta|} \sum_{(p,q) \in I_\vartheta} \max_{i \in \llbracket 1, N \rrbracket} |a_{i,p,q}|, \]
(otherwise $0$ if the sum is empty), and we have the expected estimate \eqref{estimaabetaiota}. 

If $b_{i,\vartheta}(t) \ne 0$, then $I_{\vartheta} \ne \varnothing$ and we have
\begin{align*}
\nu_{\vartheta} & = \kappa_{\vartheta} - |\alpha| + (|\beta| + |\gamma|-1) \frac{d-2}{2}+|\delta|\frac{d-4}{2} \\
& = 2 - |\alpha| - |\gamma| - |\iota| - 2|\delta| \\
& = 2  - |\alpha| -|q|-|\delta|.
\end{align*}

We start by considering 1) for both Theorem  \ref{thmexistglobP0} and Theorem \ref{thmexistglobPgainzero}.

In the context of Theorem \ref{thmexistglobP0}, the assumption is that $q=0$ so that $\alpha$ are $\delta$ as well (as $|\alpha| \le |q|$). Hence $\nu_{\vartheta} = 2$, and $\nu_1=2$. We want to us Theorem \ref{th:conf}. It remains to check the smallness condition \eqref{asumhvLt1}, which we do by applying Case (3) of Lemma \ref{lem:conf}, that is verifying \eqref{infkvartheta}. 

If $f$ is polynomial, this condition is always satisfied (by choosing $D$ large enough). Otherwise, we assumed that $u_0$ has zeo mean, so that $\ell_0=1$ (that is $u_0$ of zero means). Now, as for contributing $b_{i,\vartheta}$, $|\iota|=|\gamma|=|\delta|=0$ (as $q=0$) formula \eqref{formkappa0} writes 
\[ \kappa_{\vartheta} = \frac{d+2}{2}-\frac{d-2}{2} |\beta| \ge \left( - \left( 1 + \frac{d-2}{2} \right) +1 \right) (|\beta|-1) +2, \] 
so that Case (3) of Lemma \ref{lem:conf} holds with $\e =-1$ and $D=-2$.
Theorem \ref{th:conf} applies, and \eqref{cvceinftysol} gives the claimed convergence rate. 

For Theorem \ref{thmexistglobPgainzero}, we want to use the structure assumption \eqref{deffstruct} and apply Theorem \ref{thmconfrefin}. Following the proof of Theorem \ref{thminftyspecial}, first notice that
\begin{equation}
\nor{g(t,y,v_{L},\partial_t v_L, \nabla_y {v_{L}})}{Y_{s-1,t}} = e^{ \frac{d-2}{2} t} \| r^{ \frac{d+2}{2} } f(u_L, \nabla u_L) \|_{Z_{s-1,e^{-t}}^0} \lesssim e^{-\nu t}.
\end{equation}
Hence \eqref{est:psi(0)_exp} holds. 
Then we need to compute the new exponents: $g_{\opR}$ has a similar structure as \eqref{def:g_zero}   (mostly only changing $w$ to $-w$), and we recall that there is no index $\iota$, $p=\beta$, $|\gamma | + |\delta| =|q|$ and $|\alpha| \le |\delta|$. Then as before for a contributing $\vartheta$ and $(p,q) \in I_\vartheta$,
\begin{align}\label{formkappa0bis} \frac{d+2}{2}- \frac{d-2}{2} |p| - \frac{d}{2} |q| =  \frac{d+2}{2} - \frac{d-2}{2} |\beta| - \frac{d}{2} (|\gamma | + |\delta|) =:\kappa_{\vartheta,\opR},
\end{align}
and so
\begin{align*} 
\nu_{\vartheta,\gh R} & = \kappa_{\vartheta, \gh R} - |\alpha| + (|\beta| + |\gamma| - 1) \frac{d-2}{2} + |\delta| \frac{d}{2} \nonumber \\
& = \frac{d+2}{2}-\frac{d-2}{2} |p| - \frac{d}{2}|q|- |\alpha| + (|p|+|q|-1) \frac{d-2}{2}  +|\delta| \nonumber \\
& = 2  -|q|- |\alpha| + |\delta|\ge 2  -|q|.
\end{align*}
By assumption $|q| \le 2$ so that $\nu_{1,\opR} \ge 0$. 
Now, we need to check that we can obtain the smallness of \eqref{asumhvLt1} by applying Lemma \eqref{lem:th_conf2_hyp}, Case (2). The definition \eqref{formkappa0bis} gives 
\[ \kappa_{\vartheta,\opR}=2-   \frac{d-2}{2}\left( |\beta| +|\gamma | + |\delta|-1 \right) - (|\gamma | + |\delta|)\ge -\frac{d-2}{2}\left( |\beta| +|\gamma | + |\delta|-1 \right), \]
since $|\gamma | + |\delta|= |q|\le 2$.
In particular, \eqref{infkvartheta} holds true with $\ell_0=1$, $\e=1$ and $D=0$. We are in a position to apply Lemma \eqref{lem:th_conf2_hyp}, Case (2) and then Theorem \ref{thmconfrefin}. 

\bigskip

It remains to treat 2) of both Theorems. We use the case \ref{hypott0} of Lemma \ref{lem:conf} and Lemma \eqref{lem:th_conf2_hyp}. Recall the relation \eqref{def:pq_ib}, so that the extra assumption ensures that for any contributing $\vartheta$ (and $(p,q) \in I_\vartheta$)
\[ |\beta|+|\gamma|+|\delta|-1 = |p|+|q|-1 \ge 1, \]
and so $h_1(0)=0$.
\end{proof}

   


\section{General results for the Dirichlet problem}
In this part, we gather the equivalent of the previous theorem we stated for scattering at infinity or at $0$ for a Dirichlet problem. The proof are mainly the same once we have the equivalent of the Duhamel formulation for Dirichlet problem stated in Lemma \ref{lmDuhamelzero}. 

\subsection{The Dirichlet problem in conformal variable}

In conformal variable, we are interested in solving in $\q Y_s$ the problem
\begin{equation} \label{eq:sys_conf_Dir} 
\partial_{tt} v - \gh D^2 v = g(t,y,v,\partial_t v, \nabla_{y} v), \quad v(0) = v_0, 
\end{equation}
where $v_0$ is given (we underline that the condition bears only on the function, not its time derivative).

It will be convenient to introduce a map similar to $\Psi$ for which we will seek a fixed point, and adapted to data at infinity: for this, we use the operator $\Phi^{D}$ well adapted to the Dirichlet boundary condition, instead of $\Phi$. 

More precisely,
given $v_0 \in (H^s(\m S^{d-1})^N$, we denote
\begin{equation} \label{def:vL}
\bs v_L = \q S(\cdot) (v_0, - \gh D v_0),
\end{equation}
and define the operator
\begin{gather} \label{def:PsiD}
\Psi^D: \bs v \mapsto \Phi^{D}( g(t,y, v + v_L,\dot v+\dot{v}_L , \nabla_{y} (v+v_L)))
\end{gather}
where $\Phi^D$ is defined in \eqref{def:PhiD} and acts component by component.

\begin{thm}[Conformal variables] \label{th:confDir}
Under the same assumptions as Theorem \ref{th:conf} and $h_{1}(0)=0$, there exists $\eta>0$ so that for any given data 
\[ v_0 = (v_{1,0}, \dots, v_{N,0}) \in H^s(\m S^{d-1}), \quad \text{with} \quad \| v_{0}\|_{H^s(\m S^{d-1})}\le \eta, \]
there exists a unique solution $\bs v = (\bs v_1, \dots,\bs v_N) \in  \q Y_{s}$ (defined for times $t \ge 0$) to the integral formulation of the system \eqref{eq:sys_conf_Dir}, with initial condition $v(0) = v_0$.

Moreover, there exists a unique $v_+\in H^s(\m S^{d-1})$ so that
\begin{align}
\label{cvceinftysolDir}
\| \bs v - \q S(\cdot) (v_+ ,-\gh D v_+) \|_{\q Y_{s,t}} \lesssim e^{-\nu_{0}t}\to 0 \quad \text{as} \quad t \to +\infty. 
\end{align}
\end{thm}

\bnp
The proof is the same as Theorem \ref{thmexistglobP}: our goal is to construct a fixed point for $\Psi^D$. The only modification in the argument is that we have to take $t_{0}=t_{1}=0$ which imposes the smallness of $ \| v_{0}\|_{H^s(\m S^{d-1})}$: this corresponds to case in condition (1) of Lemma \ref{lem:conf}.

We need the following variant of Lemma \ref{lminegG} adapted to the Dirichlet operator: 
\begin{lem}
\label{lminegGDir2}
There exists a universal constant $C>0$ so that, given $\bs v_L$ as in \eqref{def:vL} and $\bs v, \bs w\in (\q Y_{s})^k$ so that denoting $M = \max\left(\|\bs v \|_{\q Y_{s}},\| \bs w \|_{\q Y_{s}},\| \bs v_L \|_{\q Y_{s}} \right)$, then we have  
\begin{align}
\| \Psi^{D} (\bs v) - \Psi^{D} (\bs w) \|_{\q Y_{s}} & \lesssim  \sum_{\nu \in \Theta} B_{\vartheta} \left< \alpha \right>^{s+1}(|\beta| + |\gamma|+|\delta|) (CM)^{|\beta| + |\gamma|+|\delta|-1} \\
& \qquad \qquad \times \int_0^{+\infty}(1+\tau)  e^{-\nu_{\vartheta}\tau}\| \bs v - \bs w \|_{\q Y_{s,\tau}}d\tau 
\end{align}
where $\Psi^{D}$ (depending on $\bs v_L$) is defined in \eqref{def:PsiD}. Similarly,
\begin{align}
\| \Psi^D (\bs v) \|_{\q Y_{s}} & \lesssim  \sum_{\nu \in \Theta} B_{\vartheta} \left< \alpha \right>^{s+1}(|\beta| + |\gamma|+|\delta|) (CM)^{|\beta| + |\gamma|+|\delta|-1} \\
& \qquad \qquad \times \int_0^{+\infty} (1+\tau)  e^{-\nu_{\vartheta}\tau}\left(\| \bs v \|_{\q Y_{s,\tau}} + \| \bs v_{L}\|_{\q Y_{s,\tau}}\right)d\tau 
 \end{align}
\end{lem}

We omit the proof of Lemma \ref{lminegGDir2} since it follows closely the lines of that of Lemma \ref{lminegG} where the estimates of Lemma \ref{lmDuhamelinfty} are replaced by the estimates \eqref{est:lin_DuhamelDir} in Lemma \ref{lmDuhamelzero}. 

Using now Lemma \ref{lminegGDir2}, we get for $\bs v, \bs w\in Y := \{ \bs w : \| \bs w \|_{\q Y_{s}} \le \| \bs v_L \|_{\q Y_s} \}$
\begin{align}
\label{est:Psi1diffDir}
\MoveEqLeft \| \Psi^D (\bs v) - \Psi^D (\bs w)  \|_{\q Y_{s}} \\
& \lesssim  \sum_{\nu \in \Theta} B_{\vartheta}  \left< \alpha \right>^{s+1}(|\beta| + |\gamma|+|\delta|+1)  (C \| \bs v_L \|_{\q Y_{s}})^{|\beta| + |\gamma|+|\delta|-1} \\
& \qquad \qquad \times \int_0^{+\infty}\| \bs v - \bs w \|_{\q Y_{s,\tau}} \tau e^{- \nu_{\vartheta}\tau} d\tau \\
& \lesssim  \sum_{\nu \in \Theta} B_{\vartheta}  \left< \alpha \right>^{s+1}(|\beta| + |\gamma|+|\delta|+1)  (C \|\bs v_L \|_{\q Y_{s}})^{|\beta| + |\gamma|+|\delta|-1} \\
& \qquad \qquad \times \| \bs v - \bs w \|_{\q Y_{s}} \left( \int_0^{+\infty} \tau e^{-\nu_0 \tau} d\tau \right)  \\
&\lesssim h_1(C \| \bs v_L \|_{\q Y_{s}}) \| \bs v - \bs w \|_{\q Y_{s}}.
\end{align}
And similarly, $\| \Psi^D (\bs w) \|_{Y_s} \lesssim h_1(C \| \bs v_L \|_{\q Y_{s}}) \| \bs v_L \|_{\q Y_s}$. A classical argument shows that for $\eta >0$ small enough, $\Psi^D$ admits a unique fixed point $\bs r \in Y$.

Also $g(t,y,r+v_L,\partial_t (r+v_L), \nabla_{y}(r+v_L))  \in \q Y_s$ so that Lemma \ref{lmDuhamelzero} applies (and the discussion that precedes it): in particular, $v_+$ is the first component of 
\[ \int_0^\infty \q S(-\tau) \begin{pmatrix} 0 \\ g(\tau,y,r+v_L,\partial_t (r+v_L), \nabla_{y}(r+v_L)) \end{pmatrix} d\tau . \qedhere \]
\enp

In the case of a gain of the first Duhamel iterate, we get similarly the following result.

\begin{thm}[Conformal variables 2] \label{th:confDirnu}
Under the same assumptions as Theorem \ref{th:conf2}, there exists $\eta>0$ so that for any given data $v_0 = (v_{1,0}, \dots, v_{N,0}) \in H^s(\m S^{d-1})$ such that 
\[ \| v_{0}\|_{H^s(\m S^{d-1})}+ \| \Psi^D(0) \|_{\q Y_{s} }\le \eta \]
and satisfying
\[ \Psi^D(0) \in \q X_{\nu,0}. \]
Then there exists a unique solution $\bs v = (\bs v_1, \dots, \bs v_N) \in  \q Y_{s}$ (defined for times $t \ge 0$) to the integral formulation of the system \eqref{eq:sys_conf}, with initial condition $v(0) = v_0$

Moreover, there exists a unique $v_+ \in H^s(\m S^{d-1})^{N}$ so that
\begin{align}
\label{cvceinftysolDircrit}
\| \bs v - \q S(\cdot) (w_+,-\gh D w_+) \|_{\q Y_{s,t}} \lesssim e^{-\nu t}\to 0 \quad \text{as} \quad t \to +\infty. 
\end{align}
\end{thm}


\begin{proof}
It follows the lines of the proof of Theorem \ref{th:conf2}, with the same modification as in the proof of Theorem \ref{th:confDir} using Lemma \ref{lminegGDir2}.
\end{proof}

\subsection{The Dirichlet problem close to infinity}

We consider again the system \eqref{eq:sys_orig}, now with boundary condition:
\begin{equation} \label{eq:sys_Dir}
\begin{cases}
\Delta u = f(u, \nabla_x u) \quad & \text{on } \{ |x| \ge 1 \}, \\
u|_{\m S^{d-1}} = u_0,
\end{cases}
\end{equation}
where $u_0$ is given. We have analoguous results to that in Section \ref{s:scattinfty}, and the proofs follow the same lines: we leave the details to the reader.

We assume $f$ can be expanded in power series as in \eqref{deff}, and we recall the definition  \eqref{def:nu1} of the exponent:
\[ \nu_1 : = \inf \{ (d-2) (|p| + |q|)-d :  a_{i,p,q} \ne 0 \}. \]

Here is the first general statement (corresponding to Theorem \ref{thmexistglobP}).

\begin{thm}\label{thmexistglobPDir}
Assume that $d \ge3$ and $f$ satisfies
\begin{equation} \label{est:nu1_scrit_D}
\nu_1 \ge 0.
\end{equation}
There exists $\eta>0$ so that the following property holds.

Let $u_0 \in H^s(\m S^{d-1})$ with  $ \| u_{0} \|_{H^s(\m S^{d-1})}\le \eta$,
then, there exists a unique $\bs{u}\in \q Z_{s}^{\infty}$ solution of \eqref{eq:sys_Dir}; moreover, there exists a unique $\bs u_{+,L} \in \q Z_{s}^{\infty}$ solution of $\Delta u_{+,L} =0$ so that
\begin{align}
\nor{(\bs u - \bs u_{+,L})(r\cdot)}{\q Z^{\infty}_{s,r}} \to 0 \quad \text{as} \quad r \to +\infty.
\end{align}
Actually, convergence holds with rate $r^{-\nu_1}$.
\end{thm}

The next statement consider the case when the first Duhamel iterate has improved decay (corresponding to Theorem \ref{thmexistglobPgain}).

\begin{thm}\label{thmexistglobPgainDir}
Let $d \ge 3$, $\nu >0$ and assume $f$ satisfies 
\begin{equation} \label{est:nu1_crit_D}
\nu_1 \ge 0. 
\end{equation}
There exists $\eta>0$ so that the following property holds.

Let $u_{0} \in H^s(\m S^{d-1})$, denote $\bs u_{L} \in  \q Z_{s}^{\infty}$ the solution in $\q Z_{s}^{\infty}$ of 
\[ \Delta u_{L}=0 \quad \text{on } \{|x|\ge 1\} \quad \text{and} \quad u_{L}|_{\m S^{d-1}}= u_{0}, \]
and assume that
\[ \| u_{0} \|_{H^s(\m S^{d-1})}+\sup_{r\ge 1}r^{2 + \nu }\nor{f({u_{L}},\nabla {u_{L}})(r\cdot)}{  Z_{s -1 , r}^{\infty}}\le \eta. \]
Then, there exists a unique $\bs u \in \q Z_{s}^{\infty}$ small\,\footnote{See Theorem  \ref{thmexistglobPgain} for a precise condition.} solution of \eqref{eq:sys_Dir}; moreover, there exists a unique $\bs u_{+,L}\in \q Z_{s}^{\infty}$ solution of $\Delta u_{+,L} = 0$ so that
\begin{align}
\nor{(\bs u - \bs u_{+,L})(r\cdot)}{\q Z^{\infty}_{s,r}} \to 0 \quad \text{as} \quad r \to +\infty.
\end{align}
Actually, convergence holds with rate $r^{-\nu}$.
\end{thm}

Finally, in the case when $f$ has a structure so that the corresponding $g$ does not depend on $y$, we have the analoguous of Theorem \ref{thminftyspecial}.

\begin{thm}
\label{thminftyspecialDir}
Let $d \ge 2$. Assume that $f$ has the structure as in Theorem \ref{thminftyspecial}, that is \eqref{deffstruct} and, if $d=2$, \eqref{deffstructbracket}. Recall the relevant exponant
\[ \nu_{1,\opR} : =\inf \{  (d-2)|p| + (d-1)|q|-d: a_{i,p,q} \ne 0 \}. \]
Then:

(1) The same result as Theorem \ref{thmexistglobPDir} holds replacing assumption \eqref{est:nu1_scrit_D} by $\nu_{1,\opR} >0$.

(2)  The same result as Theorem \ref{thmexistglobPgainDir} holds replacing assumption \eqref{est:nu1_crit_D} by $\nu_{1,\opR} \ge 0$.
\end{thm}

\subsection{The Dirichlet problem close to zero}
\label{s:Dirzero}

Finally, we state the equivalent of Theorem \ref{thmexistglobP0} and \ref{thmexistglobPgainzero} for Dirichlet boundary condition, close to zero, that is:
\begin{equation} \label{eq:sys_Dir0}
\begin{cases}
\Delta u = f(u, \nabla_x u) \quad \text{on } B(0,1) \setminus \{0\}, \\
u|_{\m S^{d-1}} = u_0.
\end{cases}
\end{equation}
Note that the solutions are naturally constructed outside of zero because of the change of variable. Yet, they will be proved to be solution on $B(0,1)$ in several cases. 

The proofs are the same, and we leave the details to the reader.

\begin{thm}\label{thmexistglobPDir0} Assume $d\ge 2$ and that $f$ as in \eqref{deff} satisfies the subcriticality assumptions
\[ a_{i,p,q} \ne 0 \Longrightarrow (q=0 \text{ and } |p| \ge 2). \]
There exists $\eta >0$ such that the following holds.

Let $u_{0} \in H^s(\m S^{d-1})$ with $\| u_0 \|_{H^s(\m S^{d-1})} \le \eta$, then, there exists a unique $\bs{u}\in \q Z_{s}^{\infty}$ solution of \eqref{eq:sys_Dir0}; moreover, there exists a unique $\bs u_{+,L} \in \q Z_{s}^{0}$ solution of $\Delta u_{+,L} =0$ (on $B(0,1)$) so that
\begin{align}
\nor{(\bs u - \bs u_{+,L})(r\cdot)}{\q Z^{0}_{s,r}} \to 0 \quad \text{as} \quad r \to 0.
\end{align}
Actually, convergence holds with rate $r^2$.
\end{thm}

\begin{thm}\label{thmexistglobPgainDir0}
Assume $d \ge 2$, $f$ as in \eqref{deff} satisfies the structure condition of Theorem \ref{thminftyspecial} (that is, \eqref{deffstruct} and, if $d=2$, \eqref{deffstructbracket}) and the subcriticality assumptions
\[ a_{i,p,q} \ne 0 \Longrightarrow (  |q| \le 2 \text{ and } |p| + |q| \ge 2) . \]
Let $u_{0} \in H^s(\m S^{d-1})$, denote $u_{0,L} \in \q Z_{s}^{0}$ the associated bounded solution of
\[ \Delta {u_{L}}=0 \quad \text{on } B(0,1) \quad \text{and} \quad u_{L}|_{\m S^{d-1}}= u_{0}, \]
and assume that we have the bound
\[  \| u_0 \|_{H^s(\m R^{d-1})} + \sup_{0<r\le 1}r^{2 - \nu }\nor{f({u_{L}},\nabla {u_{L}})(r\cdot)}{  Z_{s -1 , r}^{0}} \le \eta. \]
Then, there exist a unique $u\in \q Z^{0}_s$ small\,\footnote{See Theorem  \ref{thmexistglobPgain} for a precise condition.} solution of \eqref{eq:sys_Dir0}; moreover, there exists a unique $\bs u_{+,L} \in \q Z_{s}^{0}$ solution of $\Delta u_{+,L} =0$ (on $B(0,1)$) so that
\begin{align}
\nor{(\bs u - \bs u_{+,L})(r\cdot)}{\q Z^{0}_{s,r}} \to 0 \quad \text{as} \quad r \to 0.
\end{align}
Actually, convergence holds with rate $r^\nu$.
\end{thm}

\section{Applications}
\subsection{Critical semilinear equations}

This section is about the proof of the main results about the critical semilinear equation. We first give a definition of weak solution. 

\begin{definition}
\label{defsoluext}
We say that $u\in \dot{H}^1_{loc}(\{|x|\ge R\})$ is a solution of 
\[ \Delta u = f(u) \quad \text{on} \quad \{|x|\ge R\} \]
if we have
\begin{align}
\label{eqnsemilin}
\forall v \in \mc C^{\infty}_c(\{|x|> R\}), \quad \int_{|x|> R}\nabla u\cdot \nabla v~dx + \int_{|x|> R} f(u) v~dx=0
\end{align}
\end{definition}

\bnp[Proof of Theorem \ref{thmexistglobPintro}]
This is just a particular case of Theorem \ref{thmexistglobP} when $f$ does not depend on the derivatives.
\enp
\bnp[Proof of Theorem \ref{thmH1reg}]
1) Let $\e>0$ small to be chosen later. Since $u\in\dot{H}^{1}(\{|x|\ge 1\})$, there exists $r_{0}$ so that $\nor{\nabla u}{L^{2}(|x|\ge r_{0}/2)}\le \e$. Denoting $u^{r_{0}}(x)=(r_{0})^{d/2-1}u(r_{0}x)$, the function  $u^{r_{0}}$ satisfies $\nor{\nabla u^{r_{0}}}{L^{2}(|x|\ge 1/2)}\le \e$ and is solution of the same elliptic equation. By Sobolev estimate \eqref{Soblocal}, we have also $\nor{u^{r_{0}}}{L^{2^*}(|x|\ge 1/2)}\lesssim \e$. If $\e$ is small enough, the trace estimate of Proposition \ref{propelipticregtrace} given in the Appendix, yields $ \nor{u^{r_{0}}|_{\mathbb{S}^{d-1}}}{H^{s}(\mathbb{S}^{d-1})}\lesssim \e$.

We notice that for the critical exponent $p=2^{*}-1=\frac{d+2}{d-2}$, Theorem \ref{thmexistglobP} and \ref{thmexistglobPDir} hold with $\nu_{0}=(d-2)p-d=2$.
Let us choose $\e$  small enough so that $C\e\le \eta$ where $\eta$ is given in Theorem \ref{thmexistglobPDir}: it applies and yields a solution $\widetilde{u}\in\q Z_{s}^{\infty}$ so that $\widetilde{u} |_{\m S^{d-1}}=u^{r_{0}}|_{\m S^{d-1}}$. We check easily that it satisfies 
\[ \nor{\widetilde{u}}{\q Z_{s}^{\infty}}\lesssim  \nor{u^{r_{0}}|_{\mathbb{S}^{d-1}}}{H^{s}(\mathbb{S}^{d-1})}\lesssim \e. \]
Due to Lemma \ref{lminjectZ},  $\widetilde{u}\in \dot{H}^{1}(|x|\ge 1)$, NS $\nor{\widetilde{u}}{L^{2^*}(|x|\ge 1)}\lesssim \nor{\widetilde{u}}{\q Z_{s}^{\infty}}\lesssim \e$.  In particular, for $\e$ small enough, the uniqueness property given in Proposition \ref{thmH1boundaryelem}, applied with $q=2^{*}-1$ and $\frac{d(q-1)}{2}=2^{*}$, implies that $\widetilde{u}=u^{r_{0}}$ and therefore, $u^{r_{0}}\in  \q Z_{s}^{\infty}$. Hence $u\in \q Z_{s,r_{0}}^{\infty}$.

The scattering result \eqref{scattranslrdinverseH1regpuiss} is obtained from the similar statement in Theorem  \ref{thmexistglobPDir}.

2) The converse (and last part) of the Theorem is actually a consequence of Theorem \ref{thmexistglobP} and a special case of Theorem \ref{thmexistglobPintro}.
\enp

\bnp[Proof of Corollary \ref{corUCPup}]
Theorem \ref{thmH1reg} gives $r_0 \ge 1$ such that for any $\ell\in \N$,
\begin{align*}
(r/r_{0})^{d-2+\ell}\nor{P_{\ell}(u-u_L)(r\cdot)}{H^{s}(\mathbb{S}^{d-1})}\le   \nor{(u-u_{L})(r\cdot)}{Z_{s,r/r_{0}}^{\infty}}\underset{r\to +\infty}{\longrightarrow}0. 
\end{align*}
From the assumption and triangular inequality, we infer \[ (r/r_{0})^{d-2+\ell}\nor{P_{\ell}u_L(r\cdot)}{H^{s}(\mathbb{S}^{d-1})} \to 0 \quad \text{as} \quad r\to +\infty. \]
Now, notice that $P_{\ell}u_L(r\cdot)=(r/r_0)^{-(d-2)-\ell}P_{\ell}u_L(r_0\cdot)$: hence $P_{\ell}u_L=0$ for any $\ell \in \m N$, and therefore $u_L=0$. The uniqueness in Theorem \ref{thmH1reg} 2) implies $u=0$ in $\{|x|\ge r_1\}$ for a possibly larger $r_1\ge r_0$. 

Now, due to the result of Trudinger \cite{Trud:68} (see also Section \ref{s:Apsemilin} of the Appendix for a quantification), we have $u\in \mc C^{\infty}(\{|x|>1 \})$. In particular, $u$ solves the equation $\Delta u=Vu$ with $V=\kappa u^{p-1}\in L^{\infty}_{\text{loc}}(\{|x|>1 \})$. We can conclude by standard unique continuation arguments that $u=0$ in $\{|x|\ge 1\}$, see for instance \cite[Theorem 5.2]{LLR_Book1}.

For the second part of the Corollary: let $\ell\in \N$, then for any $\beta \ge 0$, the condition $u(x) = O(|x|^{-\beta})$ gives
\[ \nor{P_{\ell}u(r\cdot)}{H^{s}(\mathbb{S}^{d-1})}\le C_{\ell,s} \nor{u(r\cdot)}{L^{2}(\mathbb{S}^{d-1})}\le C_{\ell,s} \nor{u(r\cdot)}{L^{\infty}(\mathbb{S}^{d-1})}\le C_{\ell,s,\beta} r^{-\beta}. \] 
Choose $\beta>\ell+d-2 $, so that the assumptions of the first part of the Corollary are fulfilled, and this gives the result.
\enp

\bnp[Proof of Theorem \ref{thmregdecay}]
The scheme of the proof is quite similar to the critical case with different scaling and spaces. We only stress the differences. 
1) The appropriate scaling is given by denoting $u^{r_{0}}(x)=r_{0}^{\frac{2}{p-1}}u(r_{0}x)$: then  $u^{r_{0}}$ is solution of the same elliptic equation and satisfies $\nor{u^{r_{0}}}{L^{\infty}(|x|\ge 1/2)}\le Cr_0^{-\eta} $. For sufficiently large $r_0$, we can apply Proposition \ref{propelipticregtracegen} to get that $ \nor{u^{r_{0}}|_{\mathbb{S}^{d-1}}}{H^{s}(\mathbb{S}^{d-1})}\le C_{s}\e$. Then Theorem \ref{thmexistglobPDir} applies and we can construct a nonlinear solution $\widetilde{u}\in \q Z_{s}^{\infty}$ with the same Dirichlet data as $u^{r_0}$, and with convergence to a linear solution. To conclude as in the critical case, we want to apply the uniqueness Theorem  \ref{thmH1boundaryelem}: it remains to check that $\nor{u^{r_{0}}}{L^{\frac{d(p-1)}{2}}(|x|\ge 1)}$ and $\nor{\widetilde{u}}{L^{\frac{d(p-1)}{2}}(|x|\ge 1)}$ are finite and can be made small enough, possibly making $r_0$ even larger. We use the decay assumption to get 
\[ \nor{u^{r_{0}}}{L^{\frac{d(p-1)}{2}}(|x|\ge 1)}^{\frac{d(p-1)}{2}}\le Cr_{0}^{d}\int_1^{+\infty} r^{d-1}(rr_0)^{-d-\eta\frac{d(p-1)}{2}}dr \le C_{\eta,d,p}r_0^{-\eta\frac{d(p-1)}{2}}. \]
For $\widetilde{u}$, we use Lemma \ref{lminjectZ}: as $p>\frac{d}{d-2}$ there hold $\frac{d(p-1)}{2}>\frac{d}{d-2}$. 
\qedhere 
\enp

\bnp[Proof of Corollary \ref{corfocsemi}]
Note that since $p\in 2\N+1$, then $|u|^{p-1}u=u^p$ (recall that $u$ is real valued) which is the context of the equations considered in \cite{BBC:75,V:81}. V\'eron proved in \cite[Th\'eor\`eme 4.1]{V:81} that $|x|^{d-2}u(x)$ converges to a constant. In particular, since $d-2>\frac{2}{p-1}$, the decay assumptions of Theorem \ref{thmregdecay} is satisfied and we can also get by Lemma \ref{decaygradsemi} that $u\in \dot{H}^1(\{|x|\ge R\})$ for $R$ large enough. In particular, the assumptions of Theorem \ref{thmregdecay} are verified. 
\enp


\subsection{Conformal equations in dimension \texorpdfstring{$2$}{2}}
\label{s:conf2}

The purpose of this section is the proof of Theorem \ref{thmH1regHarmon}. 

The formulation \eqref{systconformembed} considering the embedding $\mathcal{N}\subset \R^M$ is well adapted for regularity results of weak solutions as is often the case in the literature. Yet, in one part of our results, we want to construct some solutions and it seems better suited to consider local coordinates on the manifold $\mathcal{N}$ to ensure that the constructed solutions indeed belong to $\mathcal{N}$. This will not be a loss of generality when the solution is regular enough and we can localize in the target manifold $\mathcal{N}$. 


Let $u_\infty \in \mathcal{N}$ and $W$ be a small open neighbourhood of $u_\infty$ in $\q N$ so that there exists some coordinate charts so that $W \sim V\subset \R^{N}$. 
In local coordinates, for $u\in \mc C^2(\Omega,\mathcal{N})$, we will study some solutions of 
\begin{align}
\label{formconf}
\tag{Conf-C}
\Delta u& =f(u,\nabla u) \quad \text{with}\\
f_{i}(u,\nabla u)&=-\sum_{j,\ell}\Gamma_{j\ell}^{i}(u)\nabla u^{j}\cdot \nabla u^{\ell}-\sum_{j,\ell}H_{j\ell}^{i}(u)\partial_x  u^{j}\partial_y u^{\ell}\\
&=-\sum_{j,\ell}\Gamma_{j\ell}^{i}(u)\nabla u^{j}\cdot \nabla u^{\ell} + \frac{1}{2} \sum_{j,\ell} H_{j\ell}^{i}(u) \nabla^{\perp}  u^{j}\cdot \nabla  u^{\ell},
\end{align}
where $\Gamma_{j\ell}^{i}$ are the Cristoffel symbols and $H_{j,\ell}^i=-H_{\ell, j}^i=-H^j_{i,\ell}$. Here, we denoted $\nabla^{\perp}u=(-\partial_y u,\partial_x u)$. Note, that in what follows, when we will say that $u$ is solution of \eqref{formconf}, it will always be implicit that it is valued in some local charts in the considered domain.

Before starting the proof of Theorem \ref{thmH1regHarmon} \emph{per se}, we begin by checking that equation \eqref{formconf} satisfies the null condition and the various conditions of our abstract theorems. 

First we verify that the assumptions of Theorem \ref{thminftyspecial} are satisfied. We have then
\begin{align}
f_i(u,\varpi)=-\sum_{j,\ell}\Gamma_{j\ell}^{i}(u)\varpi_{j}\cdot \varpi_{\ell}-\frac{1}{2}\sum_{j,\ell}H_{j\ell}^{i}(u)(\varpi_{1,k}\varpi_{2,\ell}-\varpi_{1,\ell}\varpi_{2,k})
\end{align}
We see that the harmonic part of the conformal system (the first sum in \eqref{formconf} where the Christoffel symbols $\Gamma$ appear) has the form \eqref{deffstruct} of Theorem \ref{thminftyspecial} while the $H$-system nonlinearity are sums of terms of the form
\[ H_{k,\ell}^i (u) (\partial_{x} u_k \partial_y u_\ell - \partial_x u_\ell \partial_y u_k) \]
which satisfies the typical form \eqref{deffstructbracket} in Theorem \ref{thminftyspecial}. Concerning the exponents: $d=2$ and  $|q|=2$ when $a_{i,p,q} \ne 0$, so that we compute
\[ \nu_{1,\opR} : =\inf \{  (d-2)|p| + (d-1)|q|-d: a_{i,p,q} \ne 0 \} = 0. \]
We compute the associated $g$ for $y\in \m S^{1}$ and $z\in (T_y \m S^{1})^N$: the expression is the same as in the proof of Theorem \ref{thminftyspecial} for $g_{\opR}$, but with more simplifications. Indeed,
\begin{align*}
    g_i(t,y,v,w,z)=e^{2t} f_i(v,e^{-t} (w\otimes y+z)).
\end{align*}
Let us compute the term corresponding to $\varpi_{j}\cdot \varpi_{\ell}$. It writes
\begin{align}
\MoveEqLeft \sum_{k=1}^{2} ( y_{k} w_{j} +  z_{j,k}) (  y_{k} w_{\ell} + z_{\ell,k})
 =  w_{j}w_{\ell}+ \sum_{k=1}^{2} z_{j,k} z_{\ell,k}.
\end{align}
We used that $\sum_{k=1}^{2} y_{k}^{2}=1$ (as $y\in \m S^{1}$) and $\sum_{k=1}^{2}  y_k z_{i,k}=y\cdot z_i=0$ (as $z_i\in T_y \m S^{1}$). Then, the term corresponding to $\varpi_{k,1}\varpi_{\ell,2}-\varpi_{\ell,1}\varpi_{k,2}$ writes
\begin{align}
\MoveEqLeft ( y_{1} w_{k} +  z_{k,1}) (  y_{2} w_{\ell} + z_{\ell,2})-( y_{1} w_{\ell} +  z_{\ell,1}) (  y_{2} w_{k} + z_{k,2})\\
&=w_k( y_1 z_{\ell,2}-y_2z_{\ell,1})+w_{\ell}(y_2z_{k,1}-y_1z_{k,2})+ z_{k,1}z_{\ell,2}-z_{\ell,1}z_{k,2} \\
&=w_k z_{\ell}\cdot y^{\perp}-w_{k}z_{k}\cdot y^{\perp}.
\end{align}
We used that $z_{k,1}z_{\ell,2}-z_{\ell,1}z_{k,2}=0$ (as $z_k$ and $z_{\ell}$ are necessarily colinear). Notice that for two function $u$ , $v$ defined on $ \m S^{1}\approx \R/2\pi \Z$ with running point  $\theta $, we have $\nabla_y u\cdot y^{\perp}= \partial_{\theta} u$ and $\sum_{k=1}^{2} (e_k\cdot \nabla_y u )(e_{k} \cdot \nabla_y v)=\nabla_y  u\cdot \nabla_y v=\partial_{\theta}u \partial_{\theta }v$. In particular, we get
\begin{align}
\label{gconf2}
    g_i(t,y,v,\partial_t v,\nabla_y v)&=-\sum_{j,\ell}\Gamma_{j\ell}^{i}(v)\nabla_{t,\theta} v^{j}\cdot \nabla_{t,\theta} v^{\ell}-\frac{1}{2}\sum_{j,\ell}H_{j\ell}^{i}(v)\nabla_{t,\theta}^{\perp}  v^{\ell}\cdot \nabla_{t,\theta}  v^{j}.
\end{align}
This expression also allows to recover the conformal invariance of the equation.

Let us now check \eqref{est:f(uL)_exp}. As in the proof of Theorem \ref{thmexistglobPgain} (see the computations \eqref{est:g_exp}, it suffices to verify that there holds
\[ \nor{g(t,y,v_{L},\partial_t v_L, \nabla_y {v_{L}})}{Y_{s-1,t}} \lesssim e^{-\nu t}, \]
for some $\nu >0$.

Now, in view of Remark \ref{rk:nullmat}, the nonlinearity $g$ given in \eqref{gconf2} satisfies the elliptic null condition at each point. Hence  Proposition \ref{propNullgain} applies.
As $\| v_L \|_{Y_{s-1,t}} \lesssim \| u_L \|_{\q Z^\infty_{s,t_0}}$ for $t \ge t_0 := \ln (r_0)$ (and the standard product law \eqref{estimprodYst}), the improved product law \eqref{est:null_ell_2d} gives
\[ \nor{g(t,y,v_{L},\partial_t v_L, \nabla_y {v_{L}})}{Y_{s-1,t}} \lesssim e^{-2 t}. \]
In particular, due to \eqref{est:g_exp}, we see that \eqref{est:f(uL)_exp} is satisfied with $\nu=2$. Hence, we can apply our general Theorems \ref{thminftyspecial} and \ref{thminftyspecialDir} with $\nu=2$. For $y_{\infty}\in V\subset \R^{N}$, it means that there exists $\eta>0$, so that for any $u_{0}\in H^{s}(\m S^{d-1})$ with $\nor{u_{0}}{H^{s}(\m S^{d-1})}\le\eta$, we can construct
\begin{enumerate}[(a)]
    \item \label{i:solinfty} a solution of the problem with prescribed data at infinity, that is a unique small solution $u\in \q Z^{\infty}_{s}$ of \eqref{formconf} on $\R^{2}\setminus B(0,1)$ so that 
 \begin{align}
\label{scattranslrdHM}
\nor{u(r\cdot)-(y_{\infty}+u_{L}(r\cdot))}{Z^{\infty}_{s,r}}\lesssim r^{-2}.
\end{align}
\item \label{i:solinftyDir} a solution of the Dirichlet problem at infinity, that is 
a unique small solution $u\in \q Z^{\infty}_{s}$ of \eqref{formconf} on $\R^{2}\setminus B(0,1)$ so that 
 \begin{align}
u_{\left|\m S^{d-1} \right.}=y_{\infty}+u_{0}
\end{align}
and moreover, there exists a unique $w_{L}\in \q Z_{s}^{\infty}$ solution of $\Delta w_{L}=0$ so that
\begin{align}
\nor{(u-w_{L})(r\cdot)}{Z^{\infty}_{s,r}}\lesssim r^{-2}.
\end{align}
\end{enumerate}
Concerning the problems close to zero, a similar analysis (or using conformal invariance of both problems) allows to apply Theorem \ref{thmexistglobPgainDir0} (since $|q|=2$ for non-zero coefficients). With the same assumptions, we can construct
\begin{enumerate}[resume*]
\item \label{i:solzeroDir} a solution of the Dirichlet problem close to zero, that is 
a unique small solution $u\in \q Z^{\infty}_{s}$ of \eqref{formconf} on $B(0,1)\setminus \{0\}$ so that 
 \begin{align}
u_{\left|\m S^{d-1} \right.}=y_{\infty}+u_{0}.
\end{align}
\end{enumerate}

Item (a) answers part 2) of Theorem \ref{thmH1regHarmon}.

We now focus on part 1) of Theorem \ref{thmH1regHarmon}. For this, we need two extra independent results. The first one is the following theorem of removable singularity; it was proved by Sacks-Uhlenbeck \cite{SU:81} in the particular case of Harmonic maps, i.e. $H=0$. The proof, presented in the appendix, is mostly the same once we have proved the improved regularity Theorem \ref{thm:W1inftyH} and an equipartition result (Lemma \ref{lmPoho}).

\begin{thm}\label{thmsackuhlextensconf}If $u: B(0,1)\setminus \{0\} \longrightarrow \mathcal{N}$  is solution of \eqref{systconformembed} with finite energy, then $u$ extends to a smooth function $u: B(0,1) \longrightarrow \mathcal{N}$, solution of the same equation.
\end{thm}

We will deduce from this result some decay for harmonic maps at infinity (by conformal equivalence).

\begin{prop}
\label{propchartsasym}
There exists $\e>0$ so that for any $u$ solution of \eqref{formconf} in $\R^{2}\setminus B(0,1/12)$ with the smallness assumption  
\begin{align*}
\nor{\nabla u}{L^{2}(\R^{2}\setminus B(0,1/12))}+\nor{u-y_{\infty}}{L^{\infty}(\R^{2}\setminus B(0,1/12)))}\le \e
\end{align*}
 for one $y_{\infty}\in \R^{N}$, then, $u\in \q Z_{s}^{\infty}$ and $u$ is the unique small solution on $\R^{2}\setminus B(0,1)$ defined by Item \ref{i:solinftyDir} with $u|_{\m S^1} = y_{\infty}+u_{0}$.  
 
 In particular, there exists $v_{L}\in \q Z_{s}^{\infty}$ solution of $\Delta v_{L}=0$ (with $\| v_L \|_{Z_s^\infty} \lesssim \e$) so that 
  \begin{align}
\nor{u(r\cdot)-(y_{\infty}+v_{L}(r\cdot))}{Z^{\infty}_{s,r}}\lesssim r^{-2}.
\end{align}
\end{prop}

\bnp

We consider $\widetilde{u}(x)=u\left(\frac{x}{|x|^{2}}\right)$ which is also a harmonic map with the same bound for the energy on $B(0,12)\setminus \{0\}$. Theorem \ref{thmsackuhlextensconf} implies that $\widetilde{u}$ is solution on $B(0,12)$. By Proposition \ref{propelipticregtraceHarmon}, if $\e$ is small enough, we have 
\[ \nor{\nabla \widetilde{u}}{L^{\infty}(B(0,3/2))}\le C\e \quad \text{and} \quad \nor{\widetilde{u}|_{\m S^1} -y_{\infty} }{H^{s} (\mathbb{S}^{1})}\le C_{s}\e. \]

Using Item \ref{i:solzeroDir}, we define $\widetilde{v}\in \q Z^{0}_{s}$ small solution of \eqref{formconf} on $B(0,1)\setminus \{0\}$ and so that $\widetilde{v}|_{\m S^1} =\widetilde{u}|_{\m S^1} =u|_{\m S^1}$. We claim that we have a similar smallness bound as that for $\widetilde u$ on $\widetilde v$, namely
\[ \nor{\nabla \widetilde{v}}{L^{\infty}(B(0,1))}+\nor{\widetilde{v}-y_{\infty}}{L^{\infty}(B(0,1))} \le C\e. \]
(This is not direct from being in $\q Z^0_s$). Indeed, following Theorem \ref{thmexistglobPgainDir0}, we decompose 
\[ \widetilde{v}=y_{0}+\widetilde{v}_{+,L}+\widetilde{w}, \]
with $\widetilde{v}_{+,L}$ solution of $\Delta \widetilde{v}_{+,L}=0$  and $r^{-\nu }\widetilde{w}\in \q Z_{s}^{0}$ for $\nu=2$. The term $\widetilde{v}_{+,L}$ satisfies the expected estimates since for a linear solution and $s>1$, 
\[ \nor{\nabla \widetilde{v}_{+,L}}{L^{\infty}(B(0,1))}+\nor{\widetilde{v}_{+,L}}{L^{\infty}(B(0,1))} \lesssim \nor{\widetilde{v}_{+,L}(0)}{H^{s}(\mathbb{S}^{1})}. \]
We use Lemma \ref{lminjectZ0} to estimate $\widetilde{w}$. Note also that the proof of Theorem \ref{thmexistglobPgainDir0} provides the smallness of $\widetilde{v}_{+,L}$ and $\widetilde{w}$ in the norms used. We define $\tilde v_L = \tilde v_{+,L} + y_0 - y_\infty$.

Theorem \ref{thmsackuhlextensconf} implies that  $\widetilde{v}$ can extended to a smooth solution on $B(0,1)$ of \eqref{formconf} (there is no singularity at $0$). Proposition \ref{propUCPDiskHarmon} then gives $\widetilde{u}=\widetilde{v}$. In particular, this implies that $\widetilde{u}\in \q Z^{0}_{s}$ and therefore $u\in \q Z^{\infty}_{s}$ by Lemma \ref{lmisom0infty}. Finally, denote $v_L(x) = \tilde v_L(\frac{x}{|x|^2})$. This gives the expected result, due to uniqueness in the class of small solutions in $\q Z^{\infty}_{s}$.
\enp

\begin{remark}
In the above proof we perform a conformal transform so as to work in $B(0,1)$: it allows to conveniently make use of the Poincar\'e inequality in the uniqueness result and to get smallness in $W^{2,\infty}$ using results of the existing literature \cite{SU:81} or \cite{Schoen:84}. We believe however that it should be possible to complete a proof directly in $\R^{2}\setminus B(0,1)$ without resorting to the conformal transform. 
\end{remark}

The second result is the existence of adapted coordinate charts where the Christoffel symbols vanish at a point.

\begin{lem}
\label{lmprojectiso}
Let $y_{\infty}\in \mathcal{N}\subset \R^M$, and consider $\pi_{T_{y_{\infty}}\mathcal{N}}$ the orthogonal projection (with respect to the Euclidian metric of $\R^M$) on $T_{y_{\infty}}\mathcal{N}$. There exists a small neighbourhood $W$ of $y_\infty \in \q N$ such that $\Phi: = \pi_{T_{y_{\infty}}\mathcal{N}}|_{\q N \cap W}$ is a diffeomorphism to its image (in $\m R^M$). 

Moreover, if we take orthonormal coordinates on $T_{y_{\infty}}\mathcal{N} \oplus T_{y_{\infty}}\mathcal{N}^\perp$, its inverse can be obtained writing $\mathcal{N}$ as a local graph 
\[ \psi(x_{1},\dots, x_{N})= \left(x_{1},\dots, x_{N}, \psi_{N+1}((x_{1},\dots, x_{N}),\dots, \psi_{M}(x_{1},\dots, x_{N})\right) \]
where the $\psi_{i}$ are analytic functions. In these coordinates, we have $\Gamma_{jm}^{i}(y_{\infty})=0$, for the Christoffel symbols.

Moreover, if $u$ is solution of \eqref{systconformembed} on with $\Omega\subset \R^2$ with $u(\Omega)\subset W$, then $\Phi \circ u$ is solution of \eqref{formconf} on $\Omega$.
\end{lem}
\bnp
Writing $\mathcal{N}$ as a local graph is classical. In these coordinates, we compute for $i=1,\dots,N$, 
\[ \psi^*(\frac{\partial}{\partial x_i})=\frac{\partial \psi}{\partial x_i}=(0,\dots,0,1,0,\frac{\partial\psi_{N+1}}{\partial x_i},\dots,\frac{\partial\psi_{M}}{\partial x_i}), \] where the $1$ is in position $i$. Note also, that by definition of the tangent space, we have 
\begin{equation} \label{psi(y_infty)}
\frac{\partial\psi_{\ell}}{\partial x_i}(y_{\infty})=0 \quad \text{ for  all } i=1,\dots,N \text{ and } \ell=N+1,\dots, M.
\end{equation}

Concerning the Christoffel symbols: denoting $g$ the metric of $\mathcal{N}$ in the coordinates given by $\psi$, and $(g^{-1})_{ij}=g^{ij}$ its inverse, recall that 
\[ \Gamma_{jm}^{i}=\frac{1}{2}\sum_{\ell =1}^{N} g^{i \ell}\left(\partial_j g_{m \ell}+\partial_m g_{\ell j}-\partial_\ell g_{jm}\right). \]
(see for instance \cite[p71]{GallotHulinLaf}). Now,
\[ g_{ij}=g(\frac{\partial}{\partial x_i},\frac{\partial}{\partial x_j})=\delta_{i,j} +\sum_{l=N+1}^{M} \left(\frac{\partial\psi_{l}}{\partial x_i}\right)\left(\frac{\partial\psi_{l}}{\partial x_j}\right). \] 
In view of \eqref{psi(y_infty)}, we infer that $(\partial_m g_{ij})(0)=0$ for all $i,j,m \in \llbracket 1, N \rrbracket$, and so $\Gamma_{jm}^{i}(y_{\infty}) =0$.
\enp

We can now conclude the proof of Part 1) in Theorem \ref{thmH1regHarmon}.

We consider a harmonic map $u$ as specified. We apply Corollary \ref{corscalinvers} with $\e$ small enough so that $u_{r_0}(\cdot)=u(r_{0} \cdot)$ defined in $\R^{2}\setminus B(0,1)$ is supported in some coordinate charts of $\mathcal{N}$ close to $y_{\infty}$, and moreover  satisfies 
\begin{align}\label{ur0petit}
\nor{\nabla u_{r_0}}{L^{2}(\R^{2}\setminus B(0,1))}+\nor{u_{r_0}-y_{\infty}}{L^{\infty}(\R^{2}\setminus B(0,1))} \le \e.
\end{align}
In particular, we can consider the solution of \eqref{formconf} with value in a single chart. More specifically, for $W\subset \mathcal{N}$ a small neighborhood of $y_{\infty}$ in $\mathcal{N}$ we consider the chart $\Phi: \mathcal{N}\cap W \to T_{y_{\infty}}\mathcal{N}$ defined by $\Phi(x)=\pi_{T_{y_{\infty}}\mathcal{N}} (x)$. Due to Lemma \ref{lmprojectiso}, we see that for $W$ small enough, $\Phi$ is a diffeomorphism to its image $V=\Phi(W)$. In particular, for $r_{0}$ large enough, $\Phi \circ u_{r_0}=\pi_{T_{y_{\infty}}\mathcal{N}}(u_{r_{0}})$ with value in $T_{y_{\infty}}\mathcal{N} \simeq \R^{N}$ is a solution of an equation of the type \eqref{formconf} with analytic Christoffel symbols $\Gamma $ and coefficients $H$, and $\pi_{T_{y_{\infty}}\mathcal{N}}(u_{r_{0}})$  satisfies similar estimates as \eqref{ur0petit}. We can then apply Proposition \ref{propchartsasym} to $\pi_{T_{y_{\infty}}\mathcal{N}}(u_{r_{0}})$: we obtain that $\pi_{T_{y_{\infty}}\mathcal{N}}(u_{r_{0}})\in \q Z_{s}^{\infty}$ and 
there exists $w_{L}\in \q Z_{s}^{\infty}$ solution of $\Delta w_{L}=0$ so that 
\begin{align}
\nor{\pi_{T_{y_{\infty}}\mathcal{N}}(u_{r_{0}})(r\cdot)-(y_{\infty}+w_{L}(r\cdot))}{Z^{\infty}_{s,r}}=\nor{\pi_{T_{y_{\infty}}\mathcal{N}}(u)(r_{0}r\cdot)-(y_{\infty}+w_{L}(r\cdot))}{Z^{\infty}_{s,r}} \lesssim r^{-2}.
\end{align}
In particular, since $r_{0}$ is fixed now, we can define $u_{L}(x)=w_{L}(x/r_{0})$ so that $w_{L}(r\cdot)=v_{L}(r_{0}r\cdot)$ and  \[ \nor{\pi_{T_{y_{\infty}}\mathcal{N}}(u)(r\cdot)-(y_{\infty} + u_{L}(r\cdot))}{Z^{\infty}_{s,r/r_{0}}}\lesssim r^{-2}. \]
It is the expected result \eqref{scattranslrdinverseHarmonMap}. Moreover, the same computation gives $\pi_{T_{y_{\infty}}\mathcal{N}}(u)\in \q Z_{s,r_{0}}$. Note that the fact that $u(x)\underset{|x|\to +\infty}{ \longrightarrow}y_{\infty}$ implies $P_0 u_L=0$. Under the assumption $P_0 u_L=0$, the uniqueness is obtained as in other cases (this assumption is necessary since in dimension $2$, this component do not decay).

Up to now, we have only proved $\pi_{T_{y_{\infty}}\mathcal{N}}(u)$ is in the correct space, but we do not control the orthogonal component. Yet, as explained in Lemma \ref{lmprojectiso}, up to some rotation and translation, we can assume that $\pi_{T_{y_{\infty}}\mathcal{N}}=\R^{N}\times \{0\}$ and $\mathcal{N}\cap W$ can be parametrized by a local graph 
\[ \mathcal{N}\cap W=\left\{\left(x_{1},\dots, x_{N}, \psi_{N+1}(x_{1},\dots, x_{N}),\dots, \psi_{M}(x_{1},\dots, x_{N})\right); (x_1, \dots,x_N)\in V\right\}, \]
where the $\psi_{i}$ are analytic functions. In particular, since $u(x)\in \mathcal{N}$, we can write 
\[ u= \left(\pi_{T_{y_{\infty}}\mathcal{N}}(u),\psi_{N+1}(\pi_{T_{y_{\infty}}\mathcal{N}}(u)),\dots, \psi_{M}(\pi_{T_{y_{\infty}}\mathcal{N}}(u)\right). \]
Therefore, due to Corollary \ref{coranalyt}, $u$ belongs to $\q Z_{s,r_{0}}^\infty$.

This completes the proof of part 1) of Theorem \ref{thmH1regHarmon}.

We are now concerned with part 2) of Theorem \ref{thmH1regHarmon}. If $u_0$ is small, Item \ref{i:solinfty} settles this result. If not, $u_0$ is assumed to have zero mean, and Theorem \ref{thminftyspecial} also concludes for $r_0$ large enough: we do exactly the same analysis, but in the other direction. We first construct the solution of \eqref{formconf} in the tangent plane using Theorem \ref{thminftyspecial} and then raise it to $\mathcal{N}\subset \R^M$ by adding the other coordinates. It is then solution of \eqref{systconformembed}. We leave the details to the reader since it is very similar to part 1).



\subsection{Harmonic maps in dimension \texorpdfstring{$d\ge 3$}{d >= 3}}
The purpose of this section is the proof of Theorem \ref{thmH1Harmondimd}. We begin with a few generalities about Harmonic maps.  First, we define the weak solutions. 
\begin{definition}
\label{defsoluextHM}
We say that $u\in H^{1}_{\textrm{loc}}(\Omega,\mathcal{N})$ is a weak solution of \eqref{eqnHarmonicMaps} on $\Omega$ if, for any $v\in \mc C^{\infty}_{c}(\Omega,\R^{M})$, there hold 
\begin{align}
\label{eqnHM}
\sum_{\alpha=1}^{d}\sum_{i=1}^{M}\int_{|x|\ge R}\left[\partial_{\alpha }u^{i}\partial_{\alpha }v^{i} -\sum_{j=1}^{M}\sum_{k=1}^{M}v^{i}A^{{i}}_{jk}(u)\partial_{\alpha} u^{j}\partial_{\alpha}u^{k}\right]~dx=0
\end{align}
\end{definition}

In fact, we will mostly work in suitable coordinate charts on the target manifold as in the previous Section \ref{s:conf2}, from which we borrow the notations given in the beginning.

For a function $u$ with value in $V$, we will consider the following equation of Harmonic maps,
 \begin{align}
 \label{eqnHarmonicMapsChristo}
 \tag{HM-C}
 \Delta u^{i}+\sum_{j,k}\Gamma_{jk}^{i}(u)\nabla u^{j}\cdot \nabla u^{k}=0
 \end{align}
 where $\Gamma_{jk}^{i}$ are the Christoffel symbols for the metric $g$ in the chosen coordinate charts. We will only consider charts where the coefficients are analytic. For simplicity, we will sometimes write equation \eqref{eqnHarmonicMapsChristo} as
  \begin{align*}
 \Delta u+\Gamma(u)(\nabla u,\nabla u)=0.
 \end{align*}

As before, we state some results concerning the decay, convergence and uniqueness that will be the main tool to prove that the solution is actually the one that can be constructed by our general theorem.  We begin by some result about the decay of small solutions.
\begin{lem}
\label{lmregHMHd}
Let $d\ge 3$. Let $u\in \mc C^{2}(\R^{d}\setminus B(0,1))$ of finite energy (see \eqref{defNRJHM}) solution of the Harmonic Maps equation \eqref{eqnHarmonicMaps}. Then, for any $\e>0$, there exists $R_{0}>0$ and $C>0$ so that we have 
\begin{itemize}
\item $\ds  |\nabla u(x)|\le C \frac{\nor{\nabla u}{L^{2}(\R^{d}\setminus B(0,R_{0}))}}{|x|^{d/2}} \le \frac{\e}{|x|^{-d/2}}$ for all $x\in \R^{d} \setminus B(0,R_{0})$,
\item $\nor{\nabla u}{L^{2}(\R^{d}\setminus B(0,R_{0}))}+\nor{\nabla u}{L^{d}(\R^{d}\setminus B(0,R_{0}))}+\nor{\nabla u}{L^{\infty}(\R^{d}\setminus B(0,R_{0}))}\le \e$.
\end{itemize}
\end{lem}
The next results proves that the solutions has a limit at infinity.
\begin{lem}
\label{lm:H1HM}
Under the same assumptions as Lemma \ref{lmregHMHd}, there exists $u_{\infty}\in \mathcal{N}\subset \R^{M}$ so that $u(x)\underset{|x|\to +\infty}{\longrightarrow} u_{\infty}$. Moreover, $u-u_{\infty}\in \dot{H}^{1}(\R^{d}\setminus B(0,1))$.
\end{lem}
The following results concerns uniqueness of small $\dot W^{1,d}$ solutions in appropriate norms.
\begin{prop}
\label{propUCPRdHarmon}Let $d\ge 3$.
There exists $\e>0$ (only depending on $\mathcal{N}$ the compact target manifold) so that if $u$ and $v$ are two weak solutions in the energy space $u_{\infty}+\dot{H}^{1}(\R^{d}\setminus B(0,1))$ of Harmonic Maps equation \eqref{eqnHarmonicMaps} in $\R^{d}\setminus B(0,1))$ so that $u=v$ on the unit sphere $\m S^{d-1}$ and we have the assumptions  
\begin{align*}
\nor{\nabla u}{L^{d}(\R^{d}\setminus B(0,1))}+\nor{\nabla v}{L^{d}(\R^{d}\setminus B(0,1))}\le \e.
\end{align*}
Then, $u=v$ in $\R^{d}\setminus B(0,1))$.  
\end{prop}

The proofs of the above three statements are done in the Appendix, Section \ref{s:app:HM3}.

\bnp[Proof of Theorem \ref{thmH1Harmondimd}]
We first need to prove additional smallness and regularity proved previously. Denoting $u_{R_{0}}(x)=u(R_{0}x)$, using Lemma \ref{lmregHMHd} we get for $R_{0}$ large enough, $\nor{\nabla u_{R_{0}}}{L^{\infty}(\R^{d}\setminus B(0,1/2))}\le CR_{0}^{1-d/2}$ which can be made arbitrary small. In particular, for $R_{0}$ large enough,
\[ \nor{A(u_{R_{0}}) (\nabla u_{R_{0}},\nabla u_{R_{0}})}{L^{\infty}(\R^{d}\setminus B(0,1/2))}\le \e. \] 
Moreover, with Lemma \ref{lm:H1HM}, and up to making $R_0$ even larger, we can select $u_{\infty}\in \mathcal{N}\subset \R^{M}$ so that \[ \nor{u_{R_{0}}-u_{\infty}}{L^{\infty}(\R^{d}\setminus B(0,1/2))} \le \e. \]
Using standard iterated elliptic regularity, we get that for any $s\in \R$ and small $\eta >0$ 
\[ \nor{u_{R_{0}}-u_{\infty}}{H^{s}(B(0,1+\eta)\setminus B(0,1-\eta))}\le C_{s,\eta} \e. \] 
In particular, we have the trace estimate $\nor{(u_{R_{0}}-u_{\infty})|_{\m S^{d-1}}}{H^{s}(\m S^{d-1})}\le C_s\e$.

Now that we have gained enough regularity and decay, we see that $u_{R_{0}}$ is actually solutions of the equation \eqref{eqnHarmonicMapsChristo} in appropriate coordinates for the target manifold $\mathcal{N}$. We choose the coordinates given by the orthogonal projection on $T_{y_{\infty}}\mathcal{N} $ as in Lemma \ref{lmprojectiso}. 

We can now proceed as in Section \ref{s:conf2} the other treated cases, except that we can apply directly Theorem \ref{thminftyspecial} and \ref{thminftyspecialDir} without using the null structure (which might hold true anyway). The equation for $v=u-y_{\infty}$ writes 
\[ \Delta v+\Gamma( y_{\infty}+v)(\nabla v,\nabla v)=0, \] 
with $\Gamma(y_{\infty})=0$ (see Lemma \ref{lmprojectiso}). The quadratic nonlinearity corresponds to $q\ge 2$ and, the cancellation $\Gamma(y_{\infty})=0$ corresponds to $p\ge 1$. So, we get, for non zero coefficients, 
\[ (d-2)p + 2(d-1)q-d\ge 2(d-2)=: \nu_{1,\opR}>0. \]
Theorem \ref{thminftyspecialDir} therefore yields the existence of a solution $\widetilde{u}\in \q Z_{s}^{\infty}$ of the Harmonic map equation with the same boundary condition as $u_{R_0}-y_{\infty}$ on $\m S^{d-1}$. We wnat to prove that
\[ u_{R_0}-y_{\infty}=\widetilde{u}, \]
and for this, we will apply Proposition \ref{propUCPRdHarmon} to $u_{R_0}$ and $y_{\infty}+\widetilde{u}$. Note that Proposition \ref{propUCPRdHarmon} applies either for formulation \eqref{eqnHarmonicMaps} or \eqref{eqnHarmonicMapsChristo} in the case of regular solutions. We need both solutions to lie in $y_{\infty}+\dot{H}^{1}(\R^{d}\setminus B(0,1))$ (not necessarily small), and with gradient small in $L^{d}(\R^{d}\setminus B(0,1))$. 
For the constructed solution $y_{\infty}+\widetilde{u}$, we claim the following inequality, for $s>\frac{d}{2} + \frac{1}{2}$,
\begin{align}
\label{injectZsiLd}
\nor{\nabla \widetilde{u}}{L^{2}(\R^{d}\setminus B(0,1))}+\nor{\nabla \widetilde{u}}{L^{d}(\R^{d}\setminus B(0,1))}\le C \nor{\widetilde{u}}{\q Z_{s}^{\infty}}.
\end{align}
Indeed, it can be obtained by combining Lemma \ref{lminjectZ},  estimate \eqref{est:Zinfty_W1infty} of Lemma \ref{lminjectZ0}  and an interpolation argument between $L^{2}$ and $L^{\infty}$. 

For the solution considered $u_{R_0}$, we use Lemma \ref{lm:H1HM} to get $u_{R_0}\in y_{\infty}+\dot{H}^{1}(\R^{d}\setminus B(0,1))$ and Lemma \ref{lmregHMHd} for the norm $L^d$ of the gradient, observing the scale invariance 
\[ \nor{\nabla u_{R_{0}}}{L^{d}(\R^{d}\setminus B(0,1))}=\nor{\nabla u}{L^{d}(\R^{d}\setminus B(0,R_{0}))}. \]

We are now in position to apply Proposition \ref{propUCPRdHarmon} and get $u_{R_0}=y_{\infty}+\widetilde{u}$. We conclude as in dimension $2$ to get \eqref{scattranslrdinverseHMgeq3}, recalling that $\nu_{1,\opR}=2(d-2)$.

Finally for the orthogonal part, as in dimension $2$, we write 
\[ u= (\pi_{T_{y_{\infty}}\mathcal{N}}(u),\psi_{N+1}(\pi_{T_{y_{\infty}}\mathcal{N}}(u)),\dots, \psi_{M}(\pi_{T_{y_{\infty}}\mathcal{N}}(u)). \]
It is easy to see that in these coordinates, by orthogonality, we have for any $j=N+1,\dots, M$.
\[ \psi_{j}((x_{1}, \dots,x_{N})=O(\nor{(x_{1},\dots,x_{N})}{}^{2}). \]

Therefore, due to Corollary \ref{coranalyt} with $n=2$, each term $\psi_{j}(\pi_{T_{y_{\infty}}\mathcal{N}}(u))$ belongs to $\q Z_{s,r_{0}}$ and  the orthogonal component decays like $r^{-\frac{d-2}{2}}$.

Regarding the additional decay \eqref{addidecaythmHM}: it is obtained after applying the conformal transform, as a consequence of the last refinement \eqref{est:sol_conf_nu+nu_0} of Theorem \ref{th:conf2}, with $\nu= \nu_0 = 2(d-2)$.

Part 2) (on the existence of solution with prescribed scattering data) is proved applying Theorem \ref{thminftyspecial} in local coordinates.
\enp

\begin{remark}
\label{rkLamy}
In principle, our general theorem should allow to make an expansion to any order of our solutions using the iterated versions of the Duhamel formula. With the first iteration, we can recover a result of Alama-Bronsard-Lamy-Venkatraman \cite{ABLV22}. More precisely, it shows that $u$ can be written
\[ u=u_{0}+u_{harm}+u_{corr}+\mathcal{O}\left(\frac{1}{r^{4}}\right), \]
with $u_{0}\in \m S^{2}$, 
\begin{align*} 
u_{harm} &  = \frac{1}{r}v_{0}+\sum_{j=1}^{3}p_{j}\partial_{j}\left(\frac{1}{r}\right)+\sum_{k,\ell=1}^{3}c_{k,\ell}\partial_{k}\partial_{\ell}\partial_{\ell}\left(\frac{1}{r}\right), \\
\text{and} \quad u_{corr} & =-\frac{|v_{0}|^{2}}{r^{2}}n_{0}-\frac{|v_{0}|^{2}}{6r^{3}}v_{0}-\frac{3}{r}\sum_{k,\ell=1}^{3}v_{0}\cdot p_{j}\partial_{j}\left(\frac{1}{r}\right) n_{0},
\end{align*}
with $v_{0}$ and $p_{j}$, $ j=1, 2, 3$ orthogonal to $u_{0}$.

According to our description, we can divide the terms  in several parts:
\begin{itemize}
\item $u_{0}$ is our constant part $y_{\infty}$ at infinity
\item $u_{harm}$ is the first sphree spherical components of our linear part $u_{L}$
\item $-\frac{|v_{0}|^{2}}{r^{2}}n_{0}-\frac{3}{r}\sum_{k,\ell=1}^{3}v_{0}\cdot p_{j}\partial_{j}\left(\frac{1}{r}\right) n_{0}$ is the nonlinear correction orthogonal to the tangent space. In coordinates $(x_{1},x')$ so that $n_{0}=e_{1}$, the sphere $\m S^{2}$ can be written as a graph $x_{1}=\sqrt{1-|x'|^{2}}=1-\frac{|x'|^{2}}{2}+\mathcal{O}|x'|^{4}$. So up to terms $\mathcal{O}\left(\frac{1}{r^{4}}\right)$, the main terms will be \[ 1-\frac{1}{2}\left|\frac{1}{r}v_{0}+\sum_{j=1}^{3}p_{j}\partial_{j}\left(\frac{1}{r}\right)\right|^{2}=1-\frac{1}{2r^{2}}\left|v_{0}\right|^{2}-  \frac{1}{r} \sum_{j=1}^{3}p_{j}\cdot v_{0}\partial_{j}\left(\frac{1}{r}\right)+\mathcal{O}(\frac{1}{r^4}). \]
\item $-\frac{|v_{0}|^{2}}{6r^{3}}v_{0}$ is the first nonlinear tangential correction that corresponds to the first iterate of the Duhamel formula in our setting, taking only $\frac{v_0}{r}$ as linear component, i.e. it solves 
\[ -\Delta \left(-\frac{|v_{0}|^{2}}{6r^{3}}v_{0}\right)= \frac{v_{0}|v_{0}|^{2}}{r^5}=\frac{v_0}{r}\left|\nabla \left(\frac{v_0}{r}\right)\right|^2. \]
All the other terms of the linear part lead to more decaying terms in this formula. 
\end{itemize}
\end{remark}


\subsection{Local problem for analytic functions}
\label{subsectFischer}

The purpose of this subsection is the proof of Theorem \ref{thmexistzero}. 

We begin by recalling the notion of Fischer decomposition and related results, we refer to \cite{Render:08}.

\begin{definition}
\label{defFischer}
A polynomial $P$ and a differential operator $Q(D)$ form a Fischer pair for the space $E$, and we say that
$(P,Q(D))$ is a Fischer pair if for each $f \in E$, there exist a unique elements $q\in E$ and $r\in E$ such that $f = Pq+r$ and $Q(D)r = 0$.
\end{definition}

Let $B_{R} := \left\{x \in \R^{n} : |x| < R\right\}$ be the open ball in $\R^{n}$ with center zero and radius
$0 < R \le +\infty$, and let $A(B_{R})$ be the set of all infinitely differentiable functions
$f : B_{R}\mapsto \m C$ so that for any compact subset $K\subset B_{R}$, the homogeneous Taylor
series $\sum_{m=0}^{+\infty} f_{m}(x)$ 
converges absolutely and uniformly to $f$ on $K$, where $f_{m}$ are the
homogeneous polynomials of degree $m$ defined by the Taylor series of $f$.


It follows from Lemma \ref{lmdecompHarmo} (sometimes called Gauss decomposition in this context) that $(|x|^2,\Delta)$ is a Fischer pair for $\m C[x_1,..., x_n]$. We need this result for analytic function. It is proved in the following theorem which is a particular case of a much more general result in \cite{Render:08}.

\begin{thm}[Theorem 13 of \cite{Render:08}]
\label{thmRender}
Let $0 < R \le +\infty$. Then $(|x|^2,\Delta)$ is a Fischer pair for  $A(B_{R})$.
\end{thm}

$A(B_{R})$ seems to have naturally a structure of Fr\'echet space while it would seem preferable to deal with normed spaces. We will use the following space and norm for analytic functions on $\R^{d}$. For $R>0$, we say that $f\in  \mathcal{A}(R)$ if, for $|x| \le R$, it can be written as its Taylor expansion 
\begin{align*}
f(x) = \sum_{\alpha\in \N^{d}}\frac{1}{\alpha !} \frac{\partial^{\alpha} f}{\partial x^{\alpha} }(0) x^{\alpha} 
\end{align*}
with 
\begin{align*}
\nor{f}{\mathcal{A}(R)}= \sum_{\alpha\in \N^{d}}\frac{1}{\alpha !} \left|\frac{\partial^{\alpha} f}{\partial x^{\alpha} }(0)\right|  R^{|\alpha|}<+\infty.
\end{align*}
With these notations, we have $\mathcal{A}(R)\subset A(B_{R})\subset \mathcal{A}(c_d R)$ for some constant $c_d>$ only depending on the dimension. To prove the last embedding, notice that due to \cite[Lemma 22]{Render:08} (actually quoting \cite[Lemma 1]{S:74}), any function $f$ in $A(B_R)$ can be extended holomorphically to some domain of $\m C^d$ containing $B_{\m C^d}(0,R/\sqrt{2})$. In particular, $f$ is also holomorphic in the polydisc $B_{\m C}(0,R/\sqrt{2d})^d \subset B_{\m C^d}(0,R/\sqrt{2})$. The Cauchy inequality for polydiscs (see for instance \cite[Theorem 2.2.7]{HSeveralBook}) gives $\left|\frac{\partial^{\alpha} f}{\partial x^{\alpha} }(0)\right|\le C\frac{\alpha !}{(R/\sqrt{2d})^{|\alpha|}}$. This gives $f\in \mathcal{A}(c_d R)$ for $c_d< (1/\sqrt{2d})$.

\begin{prop}
\label{propanalZ}
Assume $f\in \mathcal{A}(R)$ is an analytic function around $0$ of radius $R>1$, then, $u\in  \q Z_{s}^{0}$ with 
\begin{align*}
 \nor{f}{ \q Z_{s}^{0}}\le C(R)\nor{f}{\mathcal{A}(R)}.
\end{align*}
\end{prop}
\bnp
We write for $|x| \le 1<R$
\begin{align*}
f(x) = \sum_{\alpha\in \N^{d}}\frac{1}{\alpha !} \frac{\partial^{\alpha} f}{\partial x^{\alpha} }(0) x^{\alpha} 
\end{align*}
We can write using the second estimate in Lemma \ref{lm:polyYst},
\begin{align*}
\nor{f}{\q Z_{s}^{0}}\le \sum_{\alpha\in \N^{d}}\frac{1}{\alpha !} \left|\frac{\partial^{\alpha} f}{\partial x^{\alpha} }(0)\right|  \nor{x^{\alpha}}{\q Z_{s}^{0}}\le C(R)\nor{f}{\mathcal{A}(R)}
\end{align*}
with $C(R)=\sup_{\alpha\in \N^{d}}\left< |\alpha|\right>^{s+{2}}R^{-|\alpha|}$ which is finite for any $R>1$.
\enp
\begin{lem}
\label{lemFischercvgece}
Let $f\in A(B_{R})$ for $R>1$. Let $f=f_{L}+|x|^{2}q$ with $\Delta f_{L}=0$ and $f_{L}$, $q\in A(B_{R})$ be the Fischer decomposition of $f$ for $(|x|^2,\Delta)$ and $A(B_{R})$. Then, we have, for any $0<r\le 1$,
\begin{align*}
\nor{f(r\cdot)-f_{L}(r\cdot)}{Z_{s,r}^{0}}\le C r^{2}.
\end{align*}
where $C$ depends on $R$ and $f$.
\end{lem}
\bnp
Select $\widetilde{R}=c_d R$ so that $A(B_{R})\subset \mathcal{A}(\widetilde{R})$ and $f$, $f_L$ and $q$ are in $\mathcal{A}(\widetilde{R})$. Then, we write
\begin{align*}
\nor{f(r\cdot)-f_{L}(r\cdot)}{Z_{s,r}^{0}}=\nor{r^{2}q(r\cdot)}{Z_{s,r}^{0}} =r^{2}\nor{q(r\cdot)}{Z_{s,r}^{0}} \le r^{2}\nor{q}{\q Z_{s}^{0}}\le C(\widetilde{R})r^{2}\nor{q}{\mathcal{A}(\widetilde{R})}
\end{align*}
where we have used Proposition \ref{propanalZ} for the last estimate.
\enp
\bnp[Proof of Theorem \ref{thmexistzero}]
The first part is a direct application of known results on Fischer decomposition. Indeed, it is known that $\q C^{0}$ solutions are actually smooth and also analytic (see for instance \cite{Friedman:58})in $B(0,1)$. In particular, they belong to $A(B_R)$ for some $R>0$. By Theorem \ref{thmRender}, there exists $u_{L}$ and $q$ in $A(B_{R})$ so that \eqref{Fischerthm} is satisfied.

For the converse, let $u_{L}$ be a bounded solution. Up to rescaling (by a factor $2$ for instance), we can assume that $u_{L}|_{\m S^{d-1}}$ is in $H^{s}(\m S^{d-1})$. We can therefore apply Theorem \ref{thmexistglobP0} and there exist $r_{0}\le 1$ and a unique $u\in \q Z^{0}_{s,r_{0}}$ solution of $\Delta u=f(u) $ on $B(0,r_{0})\setminus \{0\}$ so that 
$\nor{(u-u_{L})(r\cdot)}{Z^{0}_{s,r/r_{0}}}\lesssim r^{2}$. Due to Lemma \ref{lmsol0pointe} in the Appendix, $u$ is also an analytic solution in the classical sense on $B(0,r_{0}/2)$.  We can therefore apply Theorem \ref{thmRender} to $u$ to get that there exist $\widetilde{u}_{L}$ and $\widetilde{g}$ analytic on $B(0,r_{0}/2)$ so that $\Delta \widetilde{u}_{L}=0$ and $u=\widetilde{u}_{L}+|x|^{2}\widetilde{g}$. Lemma \ref{lemFischercvgece} gives after scaling $\nor{(u-\widetilde{u}_{L})(r\cdot)}{Z_{s,r/3r_{0}}^{0}}\le C r^{2}$. In particular since $Z_{s,r/3r_{0}}^{0}\subset Z_{s,r/r_{0}}^{0}$, we obtain $\nor{(u_{L}-\widetilde{u}_{L})(r\cdot)}{Z_{s,r/3r_{0}}^{0}}\le C r^{2}$. From Lemma \ref{lmlienlinY}, we conclude that $u_{L}=\widetilde{u}_{L}$ on $B(0,r_{0}/3)$. In particular, $u=u_{L}+|x|^{2}\widetilde{g}$ with $\widetilde{g}$ analytic in $B(0,r_{0}/3)$, which is the desired result.
\enp


\appendix
\section{Various estimates about the applications}

\subsection{A continuation result}
\begin{lem}
\label{lmsol0pointe}
Let $d\ge 2$ and $r_0>0$. 

1) Let $u\in L^{\infty}(B(0,r_0))$ solution of $\Delta u=f$ on $B(0,r_{0})\setminus \{0\}$ in the distributional sense for some $f\in L^p(B(0,r_0))$ for $p>\frac{d}{2}$. Then, it is solution in $B(0,r_{0})$ in the distributional sense. 

2) Assume that $s> \frac{d}{2} - \frac{1}{2}$ and 
$u\in \q Z^{0}_{s,r_{0}}$ satisfies $\Delta u=f(u)$ on $B(0,r_{0})\setminus \{0\}$ in the distributional sense, where $f$ is analytic. Then, $u$ is solution $\Delta u = f(u)$ in $B(0,r_{0})$ in the distributional sense and is therefore analytic.

3) Assume $d=2$ and $u\in L^{\infty}(B(0,r_0))$ is a solution of \eqref{systconformembed} on $B(0,r_{0})\setminus \{0\}$ and so that $\nabla u$, viewed as a distribution of $B(0,r_{0})\setminus \{0\}$, belongs to $L^q(B(0,r_{0})\setminus \{0\})$ for some $q>2$. Then $u$ solves \eqref{systconformembed} on $B(0,r_0)$ and is analytic.
\end{lem}

\bnp
1) To simplify the exposition, we assume $r_0=1$, and $p<+\infty$ ($p=+\infty$ is in fact a stronger assumption). Consider $v=\Delta u-f\in \mc {D}'(B(0,1))$: it is a distribution supported in $0$ and so can be writen 
\[ v=\sum_{|\alpha|\le N}c_{\alpha}\partial_{\alpha}\delta_{0}, \]
where $N\in\N$ and $c_{\alpha}$ are constants. Denote $G=\frac{g_{d}}{|x|^{d-2}}\in L^{1}(B(0,1))\subset \mc {D}'(B(0,1))$ (or $-\frac{1}{2\pi}\ln|x|$ if $d=2$) the Green function: $-\Delta G=\delta_{0}$. Define the distribution $h=-\sum_{|\alpha|\le N} c_{\alpha}\partial_{\alpha}G$. Then 
\[ \Delta h=\sum_{|\alpha|\le N}c_{\alpha}\partial_{\alpha}\delta_{0}=v=\Delta u-f \quad \text{in } \mc D'(B(0,1)). \] Let $w$ be the solution of $\Delta w=f$ on $B(0,1)$ and $w=0$ on $\m S^{d-1}$. Since $f\in  L^{p}(B(0,1))$, for some $1<p<+\infty$, $w\in W^{2,p}(B(0,1))$ by elliptic regularity (see for instance \cite[Theorem 9.15]{GTbook}). Since $p>d/2$, the Sobolev embedding gives that $w\in L^{\infty}(B(0,1))$ (and in fact has some Hölder regularity).

Let $r=u-h-w$. We have 
\[ \Delta r=\Delta u-(\Delta u-f)-f=0 \quad \text{in } \mc D'(B(0,1)). \]
In particular, $r\in \mc C^{\infty}(B(0,1/2)\subset L^{\infty}(B(0,1/2))$, and so $h = u-r-w\in L^{\infty}(B(0,1/2))$. To finish, it suffices observe the:

\medskip

\textbf{Claim:} Assume $h\in L^{\infty}(B(0,1/2))$ can be written $h=-\sum_{|\alpha|\le N} c_{\alpha} \partial_{\alpha}G$ (where $c_\alpha$ are constants), then $c_{\alpha}=0$ for all $\alpha$ and $h=0$ in $\mc D'(\m R^d)$.

\medskip

The claim immediately implies that $v=0$ and so $\Delta u=f$ in $\mathcal{D}'(B(0,1))$. We postpone the proof of the claim after the other items, which are consequences of 1).

\medskip

2) Using Lemma \ref{lminjectZ0}, we can extend $u$ to $B(0,1)$ with $u\in \q Z^{0}_{s}\subset L^{\infty}(B(0,1))\subset \mathcal{D}'(B(0,1))$. Since $f(u)\in L^{\infty}(B(0,1))$, elliptic regularity gives $u\in W^{2,p}(B(0,1/2))$ for any $1\le p<+\infty$. We can then iterate to prove that $u$ is smooth and then analytic by classical analytic regularity, see \cite{Friedman:58}.

\medskip
3) The regularity result of Rivi\`ere gives the result once we have proved that $u$ is a solution of \eqref{systconformembed} on $B(0,r_0)$ (or we bootstrap the estimates). Now, as $f$ is only quadratic in $\nabla u$, and from the assumptions, we actually have $f(u,\nabla u) \in L^{q/2}$, with $q/2 >1$. We can conclude using 1).

\medskip

This finishes the proofs up to the verification of the Claim, which we do now.

As $h\in \mc C^{\infty}(\R^d\setminus \{0\} \cap L^{\infty}(B(0,1/2))$, we can consider, for $\omega \in \m S^{d-1}$, the function $h_\omega: t \mapsto h(t \omega)$. Then $h_\omega \in C^{\infty}(\m R \setminus \{0\} \cap L^{\infty}((-1/2,1/2))$.

Assume $d \ge 3$ for the moment. In view of the formula for $G$, we can write $h_\omega (t)=\sum_{k=d-2}^{(d-2)+N} \frac{c_k}{t^k}$ for $t>0$ and for some constants $c_k\in\m C$ (which may depend on $\omega$). Considering the asymptotics close to $0$, as $h_{\omega}$ is bounded close to $0$, we infer that $c_k=0$ for all $k$. In particular, $h_{\omega}=0$ on $(-1/2,1/2\setminus \{0\}$. Now this is true for all $\omega \in \m S^{d-1}$, hence $h(x) =0$ for all $x \in B(0,1/2) \setminus \{ 0 \}$. As $h \in L^\infty(B(0,1/2)$ and is analytic on $\m R^d \setminus \{ 0 \}$, $h=0$ in $\mc D'(\m R^d)$.

The case $d=2$ is similar taking precaution with the value $\alpha=0$ which contains the term $c_0 \ln |x|$. But the asymptotics close to $0$ imply the same result.

So, we have obtained $h=0$. In particular, in the sense of distribution of $B(0,1)$, we have $0=\Delta h=\sum_{|\alpha|\le N}c_{\alpha}\partial_{\alpha}\delta_{0}$. By the uniqueness of this decomposition, we get $c_{\alpha}=0$.
\enp

\subsection{Some estimates for semilinear equations}

\label{s:Apsemilin}
We recall the following classical fact.
\begin{lem}
\label{lemsoluext}
Assume that $d\ge 3$, $q=2^{*}-1$ and
$u\in \dot{H}^1(\{|x|\ge R\})$ is solution of 
\[ \Delta u=\kappa u^{q} \quad \text{on} \quad \{|x|\ge R\}. \]
(in the sense of Definition \ref{defsoluext}; if $q$ is not an integer, we write $u^q$ for either $|u|^{q-1} u$ or $|u|^q$). Then, we have
\begin{align}
\label{eqnsemilinH1}
\forall v\in \dot{H}^1_0(|x|\ge R), \quad \int_{\{|x|\ge R\}}\nabla u\cdot \nabla v~dx+\kappa\int_{|x|\ge R}u^{q} v~dx=0.
\end{align}
\end{lem}

\bnp
Due to Sobolev embedding, $u \in L^{2^*} (\{|x|\ge R\})$ so that $u^q \in L^{(2^*)'}(\{|x|\ge R\})$, and $v \mapsto \int_{|x| \ge R} u^q vdx$ is a continuous linear form on $\dot H^1_0(\{ |x| \ge R \})$. We can conclude 
 by density of $\mc C^{\infty}_c(\{|x|> 1\})$ in $\dot{H}_0^1(\{|x|\ge 1\}$.
\enp

\begin{prop}\label{thmH1boundaryelem}Assume $d\ge 3$ and $q>1$.  Then, there exists a universal constant $\e_{d,q}>0$ so that if $u_{i}\in \dot{H}^{1}(\{|x|\ge 1\})$ , $i=1, 2$ both satisfy 
\begin{itemize}
\item $\Delta u_{i}=\kappa u_{i}^{q} $ on $\{|x|\ge 1\}$ in the sense of Definition \ref{defsoluext} for some $\kappa$ with $|\kappa|\le 1$,
\item $u_{1}=u_{2}$ on $\{|x|= 1\}$,
\item $\nor{u_{1}}{L^{\frac{d(q-1)}{2}}(|x|\ge 1)}+\nor{u_{2}}{L^{\frac{d(q-1)}{2}}(|x|\ge 1)} \le \e$,
\end{itemize}
then $u_{1}=u_{2}$.
\end{prop}

\bnp
The result is certainly classical, but we provide a short proof for self containedness. Consider $r=u_{1}-u_{2}$: its satisfies $-\Delta r=Vr$ on $\{ |x| \ge 1 \}$, where 
\[ |V(x)|\le C_{q}\left(|u_{1}(x)|^{q-1}+|u_{2}(x)|^{q-1}\right). \]
That is 
\begin{align}
\label{eqnsemilinr}
\int_{|x|\ge 1}\nabla r\cdot \nabla v~dx=\int_{|x|\ge 1}Vrv~dx
\end{align}
for any $v\in \mc C^{\infty}_c(\{|x|> R\})$. Since $r\in \dot{H}^{1}_0(\{|x|\ge 1\})$, we can pick $v_n\in \mc C^{\infty}_c(\{|x|> R\})$ so that $\nor{\nabla r-\nabla v_n}{L^2} \to 0$. This gives the convergence of the first term to $\int_{|x|\ge 1}|\nabla r|^{2}dx$.
For the second term, using the bound on $V$, H\"older estimates with
\[ \frac{1}{2^*} + \frac{1}{2^*} + \frac{2}{d} = 2 \left( \frac{1}{2} - \frac{1}{d} \right) + \frac{2}{d} = 1, \]
and Sobolev embedding \eqref{Soblocal}, we get 
\begin{align*}
\MoveEqLeft \left|\int_{|x|\ge 1} V rv_n(x)dx\right| \le  \nor{r}{L^{2^{*}}(|x|\ge 1)}\nor{v_n}{L^{2^{*}}(|x|\ge 1)}\nor{V}{L^{\frac{d}{2}}(|x|\ge 1)}\\
&\le  C_{d,q} \nor{r}{L^{2^{*}}(|x|\ge 1)}\nor{v_n}{L^{2^{*}}(|x|\ge 1)}(\nor{u_{1}}{L^{\frac{d(q-1)}{2}}(|x|\ge 1)}^{q-1}+\nor{u_{2}}{L^{\frac{d(q-1)}{2}}(|x|\ge 1)}^{q-1})\\
&\le C_{d,q} \nor{\nabla r}{L^{2}(|x|\ge 1)}\nor{\nabla v_n}{L^{2}(|x|\ge 1)}\e^{q-1}.
\end{align*}
In particular, after passing to the limit, we get 
\begin{align}
\label{gradV}
\int_{\{|x|\ge 1\}}|\nabla r|^{2}dx\le  C_{d,q} \nor{\nabla r}{L^{2}(|x|\ge 1)}^{2} \e^{q-1}.
\end{align}
If $C_{d,q} \e^{q-1}<1$, this yields $\nabla r=0$ and then $r=0$.
\enp

The following Lemma is a quantified version of the regularity result of Trudinger \cite[Theorem 3]{Trud:68}. We follow the original proof, tracking the estimates.

\begin{lem}\label{lmTrudi}Assume $q=2^{*}-1$. There exists $\e_{0}>0$ and $C>0$ so that 
for any real valued $u\in\dot{H}^{1}(\{|x|\ge 1/2\})$ solution of  $\Delta u=\kappa u^{q} $ on $\{|x|\ge 1/2\}$, with $|\kappa|\le 1$ and so that \[ \e : = \nor{u}{L^{2^{*}}(\{|x|\ge 1/2\})} \le \e_0, \]
then  we have, for $p=\frac{(2^{*})^{2}}{2}=\frac{2d^{2}}{(d-2)^{2}}>2^{*}$,
\begin{align*}
\nor{u}{L^{p}(\{|x|\in (3/4,3/2)\})} & \le  C \e.
\end{align*}
\end{lem}
\bnp
 We denote $\alpha=2^{*}/2$, $\beta=2^{*}-1$ and notice $\beta>1$, $\alpha=\frac{\beta+1}{2}\in (1,\beta)$. For any any (large) $L>0$, we define the Lipschitz functions on $\R$
\begin{align*}
G_{L}(t)&=\left\{\begin{array}{ll}0&\textnormal{ if } t\le 0\\
t^{\beta}&\textnormal{ if } 0\le t\le L\\
 L^{\alpha-1}\left(\alpha L^{\alpha-1}t-(\alpha-1)L^{\alpha}\right) &\textnormal{ if }t> L\end{array}\right.\\
 F_L(t) & = \left\{\begin{array}{ll}0&\textnormal{ if }t\le 0\\
 t^{\alpha}&\textnormal{ if }0\le t \le L\\
 \alpha  L^{\alpha-1}t-(\alpha-1)L^{\alpha} &\textnormal{ if }t> L\end{array}\right.
\end{align*}
They satisfy for every $t\in \R$ (except for $t\in \{0,L\}$ for  \eqref{F2Gprime}) and uniformly in $L>0$, 
\begin{align}
\label{F2Gprime} \left(F_L'(t)\right)^{2}&\le \alpha G_{L}'(t)\\
\label{F2G} \left(F_{L}(t)\right)^{2}&\ge t G_{L}(t)\\
\label{GLborne} G_{L}(t) & \le |t|^{\beta}, \quad F_{L}(t)\le |t|^{\alpha}\\
\label{HL}G_{L}(t) & \le \frac{1}{\sqrt{\alpha}}\sqrt{G_{L}'(t)}F_{L}(t)
\end{align}

%
Denote $\underline{u}=\max(u,0)$. Let $\eta\in \mc C^{\infty}_{0}(\{|x|\in (1/2,2)\})$ non negative nonnegative values, so that $\eta(x)=1$ when $|x|\in [3/4, 3/2]$. Then $v:=\eta^{2}G_{L}(u)=\eta^{2}G_{L}(\underline{u}) \in\dot{H}^1\cap L^{2^{*}+1}(\{|x|\ge 1/2\})$, and 
\[ \nabla v= \eta^{2}G_{L}'(\underline{u})\nabla u+2 G_{L}(\underline{u})\eta \nabla \eta. \]
According to Lemma \ref{lemsoluext}, we can now substitute the test function $v$ in \eqref{eqnsemilinH1} to get
 \begin{align*}
0=\int_{|x|\ge 1/2}\eta^{2}G_{L}'(\underline{u})|\nabla u|^{2}+ 2\int_{|x|\ge 1/2} \eta G_{L}(\underline{u})\nabla u\cdot \nabla \eta+\kappa  \int_{|x|\ge 1/2}\eta^{2} G_{L}(\underline{u})u^{q}
\end{align*}
So, noticing that $G_{L}(\underline{u})u^{q}=G_{L}(\underline{u})\underline{u}^{q}$ and using \eqref{HL} and then Cauchy-Schwarz, we get
\begin{align*}
\MoveEqLeft \int_{|x|\ge 1/2}\eta^{2}G_{L}'(\underline{u})|\nabla u|^{2} \\
&\le \frac{2}{\sqrt{\alpha}}\int_{|x|\ge 1/2}\eta  \sqrt{G_{L}'(\underline{u})}F_{L}(\underline{u})\left|\nabla u\right| \left|\nabla \eta\right|+ \int_{|x|\ge 1/2}\eta^{2} G_{L}(\underline{u})\underline{u}^{q}\\
&\le \frac{1}{2}\int_{|x|\ge 1/2}\eta^{2}G_{L}'(\underline{u})|\nabla u|^{2} + \frac{2}{\alpha}\int_{|x|\ge 1/2} \left(F_{L}(\underline{u})\right)^{2} \left|\nabla \eta\right|^{2}+ \int_{|x|\ge 1/2}\eta^{2} G_{L}(\underline{u})\underline{u}^{q}.
\end{align*}
We now bound separately the three terms of the right-hand side.
For the first one, for a constant $C= C(\eta,\alpha)$,
\begin{align}
 \int_{|x|\ge 1/2}\eta^{2}G_{L}'(\underline{u})|\nabla u|^{2}
\label{IPPTrudinger}&\le C \int_{|x|\in (1/2,2) } \left(F_{L}(\underline{u})\right)^{2}+C \int_{|x|\ge 1/2}\eta^{2} G_{L}(\underline{u})\underline{u}^{q}.
\end{align}
Using that $\nabla (F_{L}(\underline{u}))=F_{L}'(\underline{u}) \nabla u$ and \eqref{F2Gprime}, we get  
\begin{align*}
\int_{|x|\ge 1/2}\left|\eta \nabla (F_{L}(\underline{u}))\right|^{2} = \int_{|x|\ge 1/2}\eta^{2}(F_{L}'(\underline{u}))^{2}|\nabla u|^{2}\ \le \alpha \int_{|x|\ge 1/2}\eta^{2}G_{L}'(\underline{u})|\nabla u|^{2}.
\end{align*}
Concerning the third term, we use \eqref{F2G} to get 
\begin{align*}
 \int_{|x|\ge 1/2}\eta^{2} G_{L}(\underline{u})\underline{u}^{q}&\le  \int_{|x|\ge 1/2}\eta^{2} \left(F_{L}(\underline{u})\right)^{2}\underline{u}^{q-1}.
 \end{align*}
Summing up, at that point, we have proved
\begin{align}
\MoveEqLeft \nonumber \int_{|x|\ge 1/2}\left|\eta \nabla (F_{L}(\underline{u}))\right|^{2}
 \le C \nor{ F_{L}(\underline{u})}{L^{2}(|x|\in (1/2,2)) }^{2} +C \int_{|x|\ge 1/2}\eta^{2} \left(F_{L}(\underline{u})\right)^{2}\underline{u}^{q-1}\\
& \qquad \le C \nor{ F_{L}(\underline{u})}{L^{2}(\{|x|\in (1/2,2)\}) }^{2}+C \nor{\eta F_{L}(\underline{u})}{L^{2^{*}}}^{2} \nor{u}{L^{2^{*}}(\{|x|\ge 1/2\})}^{q-1},  \label{intermediTrud}
\end{align}
where we have used H\"older inequality using that $\frac{2}{2^{*}}+\frac{q-1}{2^{*}}=1$. Now, using Sobolev embeding for $\eta F_{L}(\underline{u})$, we get
\begin{align*}
\nor{\eta F_{L}(\underline{u})}{L^{2^{*}}(\{|x|\ge 1/2\})}^{2} & \le C \nor{\nabla\left(\eta  F_{L}(\underline{u})\right)}{L^{2}(\{|x|\ge 1/2\})}^{2} \\
&\le C \nor{\eta \nabla F_{L}(\underline{u})}{L^{2}}^{2}+C\nor{(\nabla \eta) F_{L}(\underline{u})}{L^{2}}^{2}\\
&\le C \nor{\eta \nabla F_{L}(\underline{u})}{L^{2}}^{2}+C\nor{ F_{L}(\underline{u})}{L^{2}(|x|\in (1/2,2)) }^{2}. 
\end{align*} 
So, combining with \eqref{intermediTrud} and using $ \nor{u}{L^{2^{*}}(|x|\ge 1/2)} = \e\le 1$, we get
\begin{align*}
(1-C \e^{q-1})\nor{\eta F_{L}(\underline{u})}{L^{2^{*}}}^{2}&\le  C\nor{ F_L(\underline{u})}{L^{2}(|x|\in (1/2,2)) }^{2}. 
\end{align*}
Finally, using \eqref{GLborne} which holds uniformly in $L$, for $\e \le (2C)^{-\frac{1}{q-1}}$, we get uniformly in $L> 0$
\begin{align*}
\nor{\eta F_{L}(\underline{u})}{L^{2^{*}}}^{2}&\le  C\nor{\underline{u}}{L^{2\alpha}(|x|\in (1/2,2)) }^{2\alpha} = C\e^{2\alpha}.
\end{align*}
(recall $\alpha=2^{*}/2$). This estimate is uniform in $L$: letting $L \to +\infty$, we obtain,
\begin{align*}
\nor{\eta \underline{u}^{\alpha}}{L^{2^{*}}}^{2}&\le  C\e^{2\alpha}. 
\end{align*}
Replacing $u$ by  $-u$, we obtain the same result for $u$ and an estimate
\begin{align*}
\nor{u}{L^{ 2^{*}\alpha}(|x|\in (3/4,3/2))}^{2\alpha }&\le  C\e^{2\alpha}.  \qedhere
\end{align*}
\enp

\begin{remark}
Note that a quite twisted way to prove the previous result of Lemma \ref{lmTrudi} would be to use Strichartz estimates for the non linear wave equation outside the translated cone. Indeed, a solution of $\Delta u=\kappa u^{q}$ on $|x|\ge 1/2)$ is also a constant solution $\Box u=\kappa u^{q} $ on $|x|\ge t+1/2)$. This should prove that $u\in L^{p}$ for some $p>2^{*}$.
\end{remark}

\begin{prop}[Trace regularity]\label{propelipticregtrace}
For any $s >0$, there exists $C_s>0$, given $u$ under the conditions of Lemma \ref{lmTrudi} with $q\in \N$,
$u|_{\mathbb{S}^{d-1}}\in H^{s}(\mathbb{S}^{d-1})$ and  $ \nor{u|_{\mathbb{S}^{d-1}}}{H^{s}(\mathbb{S}^{d-1})}\le C_{s}\e$.
\end{prop}

\bnp
We will need a (finite) sequence of decreasing domains around $\mathbb{S}^{d-1}$. Define $\Omega_{n}=\{|x|\in (1-1/(n+4),1+1/(n+4))\}$ so that $\Omega_{n+1} \Subset \Omega_{n}$ and $\mathbb{S}^{d-1}\subset \Omega_{n}$.

The previous Lemma \ref{lmTrudi} gives $\nor{u}{L^{p_{0}}(\Omega_{0})}\le  C\e$ with $p_{0}=\frac{(2^{*})^{2}}{2}>2^{*}$. In particular, $\nor{u^{q}}{L^{p_{0}/q}(\Omega_{0})}\le  C\e^{q}$ with $1<\frac{p_{0}}{q}<+\infty$, and we are in position to use elliptic regularity. Using the equation, we get
\begin{align*}
\nor{u}{W^{2, p_{0}/q}(\Omega_{1})}&\le C\nor{\Delta u}{L^{p_{0}/q}((\Omega_{0})}+C\nor{ u}{L^{p_{0}q}(\Omega_{0})} \\
&\le C \nor{ u^{q}}{L^{p_{0}/q}(\Omega_{0})} +C\nor{ u}{L^{p_{0}q}(\Omega_{0})}\\
&\le C\nor{u}{L^{p_{0}}(\Omega_{0})}^{q}+C\nor{u}{L^{p_{0}}(\Omega_{0})}\le C\e. 
\end{align*}
By Sobolev embedding, we get $\nor{u}{L^{p_{1}}(\Omega_{1})}\le C\e$ with $p_{1}=\frac{p_{0}d}{qd-2p_{0}}$, except if $qd\le 2p_{0}$ in which case we get the same estimate for any $1\le p_{1}<+\infty$. We get $\frac{p_{1}}{p_{0}}>1+\delta$ ($\delta >0$) if and only if $p_0 > \frac{qd}{2} - \frac{d}{2(1+\delta)}$, which is the case for $\delta$ sufficiently small. 

Then, we can iterate the previous process with some increasing sequence $p_{i}$ with $p_{i+1}\ge (1+\delta)p_{i}$ to get that for any $i\in \N$, there exists $C_{i}$ so that $\nor{u}{L^{p_{i}}(\Omega_{i})}\le C_{i}\e$. 

Fix $j$ so large that $p_{j}/q>d$. Using once again the equation, we get $\nor{u}{W^{2,p_{j}/q}(\Omega_{j+1})}\le C_{j}\e$, and since $p_{j}/q>d$, Sobolev embedding implies that
\[ \nor{u}{C^{1}(\Omega_{i})}\le C\e. \]
In particular,  $\nor{u^{q}}{\mc C^{1}(\Omega_{i})}\le C\e^{q}$ and using again the equation and elliptic regularity,  $\nor{u}{W^{2,r}(\Omega_{j+2})}\le C_{r}\e$ for any $1\le r<+\infty$. Choose $r >  d \ge 2$  so that $W^{k,r}$ is an algebra for all $k\in \m N$, then using repetitively the equation and elliptic regularity, we infer that $k\in \N$, \[ \nor{u}{W^{k,r}(\Omega_{i+2+k})}\le C_{r, k}\e. \]
In particular, for all $s >0$ integer, $\nor{u}{H^{s+1/2}(\Omega_{i+3+s})}\le C_{s}\e$ and we get the result by trace estimates.
%
%
%
\enp

The previous results were about the energy critical equation. We now obtain similar result for more gene general pure power nonlinearities, but under a much stronger assumption of small $L^{\infty}$ norm. We only sketch the proof.

\begin{prop}[Trace regularity]\label{propelipticregtracegen}

Let $q\in \N^*$. Then, for any $s>0$ and $R_0>0$, there exists $C>0$ so that for any $u\in\dot{H}^{1}(\{|x|\in (1/2,2)\})$, solution of  $\Delta u=\kappa u^{q} $ on $\{|x|\in (1/2,2)\}$ and so that 
\[ \e := \nor{u}{L^{\infty}(|x|\in (1/2,2))}\le R_{0}, \]
then $u_{\left|\mathbb{S}^{d-1}\right.}\in H^{s}(\mathbb{S}^{d-1})$ and $ \nor{u|_{\mathbb{S}^{d-1}}}{H^{s}(\mathbb{S}^{d-1})}\le C \e$.
\end{prop}
\bnp
By elliptic regularity and using the equation, we get for any $1<p<+\infty$
\begin{align*}
\nor{u}{W^{2, p}(\Omega_{1})}&\le C\nor{\Delta u}{L^{p}((\Omega_{0})}+C\nor{ u}{L^{p}(\Omega_{0})} \\
 &\le C \nor{ u^{q}}{L^{\infty}(\Omega_{0})} +C\nor{ u}{L^{\infty}(\Omega_{0})}\le C(R_0)\e. 
\end{align*}
We can them iterate as before to get the expected result.
\enp
\begin{lem}
\label{decaygradsemi}
Assume $d\ge 3$ and $q\in \N^*$. Let  $u\in L^{\infty}(\{|x|\ge 1\})$ be a solution of  $\Delta u=\kappa u^{q} $ on $\{|x| \ge 1\}$ and assume that there exists $C>0$ so that 
\[ \forall |x| \ge 1, \quad |u(x)|\le C |x|^{-(d-2)} . \]
Then, there exists $C'>0$ and $R \ge 2$ so that $u\in \dot{H}^1(\{|x|\ge R\})$ and
\[ \forall |x| \ge R, \quad |\nabla u(x)|\le C' |x|^{-d+1}. \] 
\end{lem}

\bnp
For $x_0\in \R^d$, we will use the rescaled and translated solution $u_{x_0}(x)=|x_0|^{\frac{2}{p-1}}u(x_0-|x_0|x)$. For $|x_0|\ge 2$, it is solution of the same equation on $B(0,1/2)$ with $\nor{u_{x_0}}{L^{\infty}(B(0,1/2))}\le C2^{d-2}|x_0|^{\frac{2}{p-1}-(d-2)}$. We will use the following Claim.

\medskip

\textbf{Claim:} Let $v$ be a solution of $\Delta v=\kappa v^{p} $ on $B(0,1/2)$ with $\nor{v}{L^{\infty}(B(0,1/2))}\le 1$. Then, $\nabla v$ is bounded on $B(0,1/4)$ and there exists a constant $D$ so that 
\[ \nor{\nabla v}{L^{\infty}(B(0,1/4))} \le D\nor{v}{L^{\infty}(B(0,1/2))}. \]

The Claim is classical by elliptic estimates, we omit the proof. 

Let $R$ so large that $C2^{d-2}R^{\frac{2}{p-1}-(d-2)} \le 1$. We apply it to $u_{x_0}$
and obtain that $|\nabla u^{x_0}(0)|\le C|x_0|^{\frac{2}{p-1}-(d-2)}$, that is 
\[ |\nabla u(x_0)|\leq C|x_0|^{-d+1}. \]
This holds uniformly for $|x_0| \ge R$. This implies $\nabla u\in L^2(\{|x|\ge 2\})$, since $d>2$. Note also that the decay $|u(x)|\le C |x|^{-(d-2)}$ also implies $ u\in L^{2^*}(\{|x|\ge 1\})$, so that $u\in \dot{H}^1(\{|x|\ge R\})$.
\enp

\subsection{Conformal equations in dimension \texorpdfstring{$2$}{2}}

In all this section, we are in dimension $d=2$ and consider solutions of equations of the type  \eqref{systconformembed}. We gather some already known facts and also some results that are quite classical consequences of them. For some of them, the results are already written for Harmonic maps, but we did not find the exact similar statement for equation  \eqref{systconformembed}. We refer to \cite{HeleinBook,HeleinSurvey, Schoen:84} for books or survey on Harmonic maps.

It has been noticed by Rivière (see the proof of \cite[Theorem 1.2.]{R:07}) that if $u$ is solution of \eqref{systconformembed}, then, it is also solution of $-\Delta u=\Omega \cdot \nabla u$ (scalar product in $\m R^2$)\footnote{That is $-\Delta u_i= \sum_{j=1}^M\Omega_{i}^{j}\cdot \nabla u_j=\sum_{j=1}^M\left(\Omega_{i,1}^{j} \partial_x u_j+ \Omega_{i,2}^{j} \partial_y u_j\right)$ for $i=1,\dots ,M$.} with $\Omega =(\Omega_{j}^{i})_{1\le i,j\le M } $ defined by 
\begin{align}
\Omega_{j}^{i}=-\sum_{\ell=1}^{M}\left(A_{j\ell}^{i}(u)-A_{i\ell}^{j}(u)\right)\nabla u^{\ell}+\frac{1}{4} \sum_{\ell=1}^{M}\left(H_{j\ell}^{i}(u)-H_{i\ell}^{j}(u)\right)\nabla^{\perp} u^{\ell}
\end{align}
which satisfies $\Omega_{j}^{i}=-\Omega_{i}^{j}$. This is a consequence of the fact that $H_{j\ell}^{i}=-H_{i\ell}^{j}$ and we have $\sum_{j=1}^{M}A_{i\ell}^{j}(u)\nabla u^{j}=\left<A(u)(e_{i},e_{\ell}),\nabla u\right>=0$ (the last scalar product being on $\R^M$) since $A(u)(e_{i},e_{\ell})\perp T_{u}\mathcal{N}$ and $\nabla u\in (T_{u}\mathcal{N})^2$.

In particular, if $\nabla u\in L^{2}(B(0,1))$, we have $\Omega\in L^{2}(B(0,1),so(M)\otimes \R^{2})$ together with $\nor{\Omega}{L^{2}(B(0,1))}\le C \nor{\nabla u}{L^{2}(B(0,1))}$.

The result of Rivi\`ere \cite{R:07} combined with the result of Giusti-Miranda and Morrey (see Theorem 9.8 of \cite{GM:book:12} for the H\"older continuity of the gradient which can be iterated by Schauder estimates) provides the following regularity result. This followed an earlier result of H\'elein \cite{Helein:91} for Harmonic maps.
\begin{thm} \label{th:Hel91}
Any $u\in H^{1}(B(0,1),\mathcal{N})$ weak solution of \eqref{systconformembed} is smooth. 
 \end{thm}

We have the following result of Rivi\`ere \cite{R:07} (we found it written in this way in \cite[Theorem 3.2.]{LR:14}) and \cite{ST:13}:
\begin{prop}\cite{R:07}
\label{propregconfLq}
 There exists $\e_{0}>$ and $C_{p}$ only depending on $p\in \N^{*}$ so that for every $\Omega\in L^{2}(B(0,1),so(M)\otimes \R^{2})$ with $\nor{\Omega}{L^{2}(B(0,1))}\le \e_{0}$ and every $u\in W^{1,2}(B(0,1))$ solution of  $-\Delta u=\Omega \cdot\nabla u$, we have
\begin{align}
\nor{\nabla u}{L^{p}(B(0,1/4))}\le C_{p}\nor{\nabla u}{L^{2}(B(0,1))}.
\end{align}
\end{prop}
We immediately obtain the following result (also written in \cite[Lemma 4.3]{Laur:XEDP}). The statement was also obtained for Harmonic maps in \cite{SU:81}.
\begin{thm}
\label{thm:W1inftyH}
Suppose that $u\in H^{1}(B(0,r),\mathcal{N})$ is a solution of \eqref{systconformembed}. There exists $\e>0$ and $C>0$ depending only on $\mathcal{N}$ and $\omega$ such that if
\begin{align*}
\int_{B(0,r)}|\nabla u(x)|^{2}dx\le \e,
\end{align*}
then $u$ satisfies the inequality 
\begin{align*}
\sup_{x\in B(0,r/8)}|\nabla u(x)|^{2}\le Cr^{-2}\int_{B(0,r)}|\nabla u(x)|^{2}dx.
\end{align*}
\end{thm}

\bnp
By scaling, we need to prove it only for $r=1$. Fix an integer $p>2$. The equation and Proposition \ref{propregconfLq} give 
\[ \nor{\Delta u}{L^{p}(B(0,1/4))}\le C\nor{\nabla u}{L^{2p}(B(0,1/4))}^{2}\le C_{p}\nor{\nabla u}{L^{2}(B(0,1))}^{2}. \]
Let $\chi\in \mc C^{\infty}_{c}(B(0,1/4))$ equal to $1$ on $B(0,1/8)$ and denote \[ u_{B_{1/4}}=\frac{1}{|B(0,1/4)|}\int_{B(0,1/4)}u(x)dx. \]
Applying elliptic estimates to the compactly supported function $v=\chi (u-u_{B_{1/4}})$,  the Poincar\'e-Wirtinger inequality and our previous bounds, we get 
\begin{align*}
\MoveEqLeft \nor{v}{W^{2,p}(B(0,1/4))}\le C\nor{\Delta v}{L^{p}(B(0,1/4))}\\
&\le C\left(\nor{\Delta u}{L^{p}(B(0,1/4))}+\nor{\nabla  u}{L^{p}(B(0,1/4))}+\nor{u-u_{B_{1/4}}}{L^{p}(B(0,1/4))}\right)\\
&\le C \left(\nor{\Delta u}{L^{p}(B(0,1/4))}+\nor{\nabla  u}{L^{p}(B(0,1/4))}\right) \lesssim_p  (\sqrt{\e} + 1) \nor{\nabla u}{L^{2}(B(0,1))},
\end{align*}
Since $p>2$, the Sobolev inequality gives 
\[ \nor{\nabla u}{L^{\infty}(B(0,1/8))}=\nor{\nabla v}{L^{\infty}(B(0,1/8)}\le \nor{v}{W^{2,p}(B(0,1/4))}. \] This gives the expected result.
\enp

The following equality is an equipartition result for solutions of \eqref{systconformembed} on $B(0,1)\setminus \{0\}$, in the energy space. This was proved for Harmonic maps in \cite[Lemma 3.5]{SU:81} using the holomorphy of the Hopf differential $u_{x}^{2}+u_{y}^{2}+iu_{x}\cdot u_{y}$. We prove it here by a Pohozahev type identity, following Laurain-Riviere \cite{LR:14}; there it is proved for solutions on the full $B(0,1)$ (see also (VII.14) in \cite{R:cours} for the harmonic case). So, we have to be careful about the cutoff introduced to avoid the point $0$ where $u$ might not be solution.

\begin{lem}
\label{lmPoho}
Let $u\in \mc C^{2}(B(0,1)\setminus \{0\},\mathcal{N})$ with finite energy on $B(0,1)$, and solution of \eqref{systconformembed}. Then, for any $0<r\le 1$, we have 
\begin{align} \label{eq:Poho}
\int_{\m S^{1}} |\partial_{\theta}u(r\theta)|^{2}~d \theta=r^{2}\int_{\m S^{1}} |\partial_{r}u(r\theta)|^{2}~d \theta.
\end{align}
\end{lem}
\bnp
By scaling, we suffices to prove the result for $r=1$.
The key property is the orthogonality when $f$ is as in \eqref{systconformembed}, which holds pointwise on $B(0,1) \setminus \{ 0 \}$
\begin{align}
\label{orthoPoho}
0=\left<\partial_{x}u,f(u,\nabla u)\right>_{\R^M}=\left<\partial_{y}u,f(u,\nabla u)\right>_{\R^M}.
\end{align}
Indeed, $A(u)(\nabla u,\nabla u) \perp T_{u}\mathcal{N}$ and $\partial_x u, \partial_y u \in T_{u}\mathcal{N}$ at each point of $B(0,1) \setminus \{ 0 \}$. For the $H$, it is a consequence of \eqref{formH}, which gives
\[ \left<\partial_{x}u,H(u)(\partial_x u,\partial_y u) u\right>_{\R^M}=d\widetilde{\omega}_{u(x)} (\partial_x u,\partial_x u,\partial_y u)=0. \]
The same holds for $\partial_y u$.

Define now the vector fields
\[ X=x\partial_{x}+y\partial_{y} \quad \text{and} \quad X_{n}=(1-\chi)(n \cdot)X, \]
where $\chi\in C^{\infty}_{c}(\R^{2})$ equals $1$ near zero. We claim that, for any function $w\in \mc C^{2}(B(0,1)\setminus \{0\},\R)$ with finite energy, we have
\begin{align}
\label{egalPoho}
\underset{n\to +\infty}{\lim} \int_{B(0,1)}\Delta  w X_{n}\cdot \nabla w ~ dx = &\int_{\partial B(0,1)} |\partial_{r}w|^{2} ~d\sigma-\frac{ 1}{2}\int_{\partial B(0,1)}|\nabla w|^{2} ~d\sigma
\end{align}
Let us assume that is holds for now, and complete the proof. Due to \eqref{orthoPoho},
\[ \sum_{i=1}^{M}\Delta u_{i}\partial_{x}u_{i}=\sum_{i=1}^{M}\Delta u_{i}\partial_{y}u_{i}=0 \quad \text{on } B(0,1) \setminus \{ 0 \}, \]
and so, as any singularity in $0$ is voided by $X_n$,
\[ \sum_{i=1}^{M}\Delta u_{i}X_{n}\cdot \nabla u^{i} \quad \text{on } B(0,1). \] 
We integrate on $B(0,1)$: using \eqref{egalPoho} with $w=u_{i}$ for each $i$, we can let $n \to +\infty$ to get
\[ 0 = \int_{\partial B(0,1)} |\partial_{r} u|^{2} ~d\sigma-\frac{ 1}{2}\int_{\partial B(0,1)}|\nabla u|^{2} ~d\sigma. \]
Finally remark that $|\nabla u|^{2}=|\partial_{r}u|^{2}+|\partial_{\theta}u|^{2}$ on $\partial B(0,1)$, so that we obtained \eqref{eq:Poho} for $r=1$, as desired.
 
It remains to prove \eqref{egalPoho}. We obtain by integration by parts on $B(0,1)$:
\begin{align}
\int_{B(0,1)}\Delta  w X_{n}\cdot \nabla w ~dx =\int_{\partial B(0,1)} \partial_{n}w X_{n}\cdot \nabla w ~d\sigma-\int_{B(0,1)}\nabla  w\cdot \nabla(X_{n}\cdot \nabla w) ~dx.\end{align}
This leads to compute the following two terms
\begin{align}
\nabla  w\cdot \nabla(X_{n}\cdot \nabla w)& =d(X_{n}\cdot \nabla w)(\nabla w)= D_{\nabla w}X_{n}\cdot \nabla w+X_{n}\cdot D_{\nabla w}\nabla w \\
& =D_{\nabla w}X_{n}\cdot \nabla w+ \text{Hess} (w)(X_{n},\nabla w),\\
\text{div} \left(X_{n}\frac{ |\nabla w|^{2}}{2}\right) & =X_{n}\cdot \nabla \frac{ |\nabla w|^{2}}{2}+ \text{div}(X_{n})\frac{ |\nabla w|^{2}}{2} \\
& = \text{Hess} (w)(X_{n},\nabla w)+ \text{div}(X_{n})\frac{ |\nabla w|^{2}}{2}.
\end{align}
For our specific choice of $X_{n}$, we have 
\begin{gather*}
D_{\nabla w}X_{n}\cdot \nabla w=(1-\chi)(n\cdot) |\nabla w|^{2} - n (\nabla \chi(n\cdot)\cdot\nabla w)(X\cdot \nabla w) \\ 
\text{and} \quad  \text{div}(X_{n})=2(1-\chi)(n\cdot) + n\nabla \chi(\cdot)\cdot X.
\end{gather*}
So, we obtain
\begin{align}
\MoveEqLeft \int_{B(0,1)}\Delta  w X_{n}\cdot \nabla w ~dx =\int_{\partial B(0,1)} \partial_{n}w X_{n}\cdot \nabla w~d\sigma-\int_{B(0,1)} \text{div} \left(X_{n}\frac{ |\nabla w|^{2}}{2}\right)~dx\\
& +\int_{B(0,1)} \left[ \left( \frac{n}{2}\nabla \chi(n\cdot)\cdot X\right)|\nabla w|^{2}+n (\nabla \chi(n\cdot)\cdot\nabla w)(X\cdot \nabla w) \right]~dx.\end{align}
Since $\nabla w\in L^2(B(0,1))$ and $n|X| |\nabla \chi(n\cdot)|$ is uniformly bounded, we get by dominated convergence that the last integral converges to zero. For the first two terms of the right hand side, we use the definition of $X_n$ and integration by parts, and we get that for each $n$,
\begin{multline*}
\int_{\partial B(0,1)} \partial_{n}w X_{n}\cdot \nabla w ~d\sigma - \int_{B(0,1)} \text{div} \left(X_{n}\frac{ |\nabla w|^{2}}{2}\right) ~dx\\
=\int_{\partial B(0,1)} |\partial_{r}w|^{2} ~d\sigma~ -\frac{ 1}{2}\int_{\partial B(0,1)}|\nabla w|^{2} ~d\sigma.
\end{multline*}
which gives our claim \eqref{egalPoho}.
\enp

We can now give the proof of removable singularity. The proof follows \cite[Theorem 3.6]{SU:81} which was performed for Harmonic maps.

\bnp[Proof of Theorem \ref{thmsackuhlextensconf}]
By scaling, since $u$ is of finite energy, we can assume without loss of generality that $u$ is defined and a solution of \eqref{systconformembed} on $B(0,2) \setminus \{ 0 \}$ and furthermore satisfies the smallness condition
\[ \delta^2 := \int_{B(0,2)}|\nabla u(x)|^{2}dx\le \min(1/5,\e), \]
where $\e$ is given by Theorem \ref{thm:W1inftyH}.

\emph{Step 1.} Let $q$ be a radial function of the form $a\log(|x|)+b$ on each annulus of the form $2^{-m} \le |x| < 2^{-m+1}$ ($m \in \m N$) so that $q(2^{-m})=\frac{1}{2\pi}\int_{\m S^{1}}u(2^{-m}\theta)d\theta$. We claim that
\begin{align}
\label{difquH}
\forall x\in B(0,1)\setminus \{0\},\quad |q(x)-u(x)|\le C\nor{\nabla u}{L^{2}(B(0,{2}))}.
\end{align}
Indeed, for $2^{-m}\le|x|\le 2^{-m+1}$, we have, since $q$ is monotonous on this interval as a variable of $r = |x|$,
\begin{align*}
|q(x)-u(x)| & \le |q(x)-q(2^{-m+1})|+|q(2^{-m+1})-u(x)| \\
& \le |q(2^{-m})-q(2^{-m+1})|+|q(2^{-m+1})-u(x)|.
\end{align*}
Note that using a finite suitable covering of the annulus, Theorem \ref{thm:W1inftyH} also gives uniformly for $0<r\le 1$
\begin{align}
\label{gradinftyann} 
\sup_{|x|=r}|\nabla u(x)|^{2}\le Cr^{-2}\int_{B(0,2r)\setminus B(0,r/2)}|\nabla u(x)|^{2}dx.
\end{align} 
So, we get, for all $m\in \N^*$,
\begin{align*}
\max \left\{|u(x)-u(y)|; 2^{-m}\le|x|,|y| \le 2^{-m+1}\right\} &\le 2^{-m+1}\max_{2^{-m}\le|x| \le 2^{-m+1}} |\nabla u(x)| \\
& \le C\nor{\nabla u}{L^{2}(B(0,2))}.
\end{align*}
In particular, taking $y=2^{-m+1}\theta$ and integrating in $\theta \in\m S^{1}$, we get
\[ |q(2^{-m+1})-u(x)|\le C\nor{\nabla u}{L^{2}(B(0,2))}. \]
Similarly, taking $y=2^{-m+1}\theta$, $x=2^{-m}\theta$ and integrating in $\theta \in\m S^{1}$, we get
\[ |q(2^{-m})-q(2^{-m+1})|\le C\nor{\nabla u}{L^{2}(B(0,2))}. \]
Summing up the above two bounds, we obtained \eqref{difquH}. 

\bigskip

\emph{Step 2.} 
For all $r \in (0,1]$, there hold
\begin{align} \label{est:H1_bord_reg_HM}
\left(1-2 \delta \right)\int_{B(0,r)} |\nabla u|^{2} ~dx \le r\nor{\nabla u}{L^{2}(\partial B(0,r))}^{2}.
\end{align}
By dilation, (and as $\delta = \| \nabla u \|_{B(0,2)} \ge \| \nabla u \|_{B(0,2r)}$ for $r \le 1$), it suffices to prove it for $r=1$. 

Using Lemma \ref{lmPoho} (observe that $u$ is smooth on $B(0,1) \setminus \{ 0 \}$ due to Theorem \ref{th:Hel91}), we get for any $0<r\le1$, 
\begin{align}
\int_{\m S^{1}} |\partial_{\theta}u(r\theta)|^{2}~d \theta =r^{2}\int_{\m S^{1}} |\partial_{r}u(r\theta)|^{2}~d \theta,
\end{align}
and after integrating in $r$, 
\begin{align}
\int_{B(0,1)} \frac{|\partial_{\theta}u(x)|^{2}}{|x|^2}~dx=\int_{B(0,1)} |\partial_{r}u(x)|^{2} ~dx =\frac{1}{2}\int_{B(0,1)} |\nabla u(x)|^{2} ~dx.
\end{align}
Also, since $q$ is radial, 
\begin{align}
\int_{B(0,1)} \frac{|\partial_{\theta}u(x)|^{2}}{|x|^2}~dx\le \int_{B(0,1)}|\nabla q(x)-\nabla u(x)|^{2}dx.
\end{align}
Hence we obtained
\begin{align} \label{est:Gu_q-u} 
\frac{1}{2} \int_{B(0,r)} |\nabla u|^{2} ~dx \le \int_{B(0,1)}|\nabla q(x)-\nabla u(x)|^{2}dx.
\end{align}
To bound the right hand term, we decompose dyadically:
\begin{align*}
\int_{B(0,1)}|\nabla q(x)-\nabla u(x)|^{2}dx= \sum_{m\in \N^{*}}\int_{B(0,2^{-m+1})\setminus B(0,2^{-m})}|\nabla q(x)-\nabla u(x)|^{2}dx.
\end{align*}
For fixed $m\in \N^*$, integration by parts give
\begin{align*}
\MoveEqLeft \int_{B(0,2^{-m+1})\setminus B(0,2^{-m})}|\nabla q(x)-\nabla u(x)|^{2}dx \\
& = -\int_{B(0,2^{-m+1})\setminus B(0,2^{-m})}\Delta  (q- u)\cdot (q-u)dx\\
& \qquad + 2^{-m+1}\int_{\m S^{1}}(q-u)(2^{-m+1}\theta)\cdot \frac{\partial (q-u)}{\partial r^+}(2^{-m+1}\theta) ~d\theta \\ 
& \qquad -2^{-m} \int_{\m S^{1}}(q-u)(2^{-m}\theta)\cdot \frac{\partial (q-u)}{\partial r^-}(2^{-m}\theta) ~d\theta.
\end{align*}
Let us precise that $\frac{\partial q}{\partial r}(r)$ is piecewise smooth, so that $\frac{\partial q}{\partial r^{\pm}}$ means the respective left and right derivative; $u$ is regular outside of $0$, so we can write $\frac{\partial u}{\partial r}$ instead of $\frac{\partial u}{\partial r^{\pm}}$. Also $\frac{\partial q}{\partial r^{\pm}}(2^{-m}\theta)$ is constant in $\theta$ because $q$ has radial symmetry, and so that from the definition of $q(2^{-m})$,
\begin{align*}
\int_{\m S^{1}}(q-u)(2^{-m}\theta)\cdot \frac{\partial q}{\partial r^{\pm}}(2^{-m}\theta) ~d\theta=0.
\end{align*}
Now, using \eqref{difquH} and \eqref{gradinftyann}, we have the estimate
\begin{multline*}
\left|2^{-m+1}\int_{\m S^{1}}(q-u)(2^{-m+1}\theta)\cdot \frac{\partial u}{\partial r}(2^{-m+1}\theta)~d\theta\right| \\
\le C \nor{\nabla u}{L^{2}(B(0,2))}\nor{\nabla u}{L^{2}(B(0,2^{-m+2})\setminus B(0,2^{-m}))} \to 0 \quad \text{as } m \to +\infty.
\end{multline*}
Hence, summing up the telescopic series and recalling that $\Delta q(x)=0$ on each annulus $2^{-m+1} >  |x| > 2^{-m}$, and the equation on $u$, we infer 
\begin{align*}
\MoveEqLeft \int_{B(0,1)}|\nabla q(x)-\nabla u(x)|^{2}dx
 = \sum_{m\in \N^{*}} \int_{B(0,2^{-m+1})\setminus B(0,2^{-m})}\Delta   u \cdot (q-u)dx \\
& \hspace{48mm} + \int_{\m S^{1}}(q-u)(\theta)\cdot \frac{\partial (q-u)}{\partial r^+}(\theta) ~d\theta \\
& = \int_{B(0,1} \Delta u \cdot (q-u)dx +  \int_{\m S^{1}}(q-u)(\theta)\cdot \frac{\partial (q-u)}{\partial r^+}(\theta) ~d\theta \\
&=  \int_{B(0,1)} f(u,\nabla u)\cdot (q-u)dx - \int_{\m S^{1}}(q-u)(\theta)\cdot \frac{\partial u(\theta)}{\partial r} ~d \theta.
\end{align*}
Using again the estimate \eqref{difquH} and that $f$ is quadratic in $\nabla u$, and $A$ and $H$ are (smooth and so)  bounded on the compact manifold $\q N$ we can estimate
\begin{align*}
\left| \int_{B(0,1)}f(u,\nabla u)\cdot (q-u)dx\right| &\le C( \| A \|_{L^\infty}, \| H \|_{L^\infty}) \nor{\nabla u}{L^{2}(B(0,1))}^{2}\nor{q-u}{L^{\infty}(B(0,1))} \\
& \le C\nor{\nabla u}{L^{2}(B(0,1))}^{2}\nor{\nabla u}{L^{2}(B(0,2))},\\
\left| \int_{\m S^{1}}(q-u)(\theta)\cdot \frac{\partial u(\theta)}{\partial r}\right|~d\theta&\le \nor{q-u}{L^{2}(\m S^{1})}\nor{\frac{\partial u}{\partial r}}{L^{2}(\m S^{1})}.
\end{align*}
So, inserting in \eqref{est:Gu_q-u}, we arrive at
\begin{align*}
\frac{1}{2}\int_{B(0,1)} |\nabla u|^{2} & \le \int_{B(0,1)}|\nabla q(x)-\nabla u(x)|^{2}dx \\ 
& \le C\nor{\nabla u}{L^{2}(B(0,1))}^{2}\nor{\nabla u}{L^{2}(B(0,2))}+\nor{q-u}{L^{2}(\m S^{1})}\nor{\frac{\partial u}{\partial r}}{L^{2}(\m S^{1})}
\end{align*}
Recall that $\nor{\nabla u}{L^{2}(B(0,2))} = \delta$ and, again due to Lemma \ref{lmPoho}, $\nor{\frac{\partial u}{\partial r}}{L^{2}(\m S^{1})}^2 = \frac{1}{2} \nor{\nabla u}{L^{2}(\m S^{1})}^2$ so that equivalently, this writes
\begin{align}
\left(\frac{1}{2}-\delta\right)\int_{B(0,1)} |\nabla u|^{2} ~dx\le \frac{1}{\sqrt{2}}\nor{q-u}{L^{2}(\m S^{1})}\nor{\nabla u}{L^{2}(\m S^{1})}.
\end{align}
Since on $\m S^{1}$, $q$ is the average of $u$, we have, by Poincaré-Wirtinger inequality (and Lemma \ref{lmPoho}),
\begin{align}
\nor{q-u}{L^{2}(\m S^{1})}\le \nor{\partial_{\theta} u}{L^{2}(\m S^{1})} = \frac{1}{\sqrt{2}}\nor{\nabla u}{L^{2}(\m S^{1})}.
\end{align}
So, we get
\begin{align}
\left(1-2\delta\right)\int_{B(0,1)} |\nabla u(x)|^{2}~dx\le \nor{\nabla u}{L^{2}(\m S^{1})}^{2},
\end{align}
as desired.

\bigskip

\emph{Step 3.}

The inequality \eqref{est:H1_bord_reg_HM} writes 
\[ (1-2\delta)\int_0^r g(s)ds\le r g(r) \quad \text{with} \quad g(r)=\nor{\nabla u}{L^{2}(\partial B(0,r))}^{2}. \]
This differential inequality in $r$ integrates to yield
\begin{align}
\label{e:intBr}
\forall r \in (0,1], \quad \int_{B(0,r)} |\nabla u|^{2}~dx\le r^{1-2\delta}\int_{B(0,1)} |\nabla u|^{2}~dx.
\end{align}
Applying Theorem \ref{thm:W1inftyH} on balls $B(x_0, |x_0|)\subset B(0,2 |x_0|)$ and some translation of \eqref{e:intBr},
\begin{align}
\forall x_0 \in B(0,1/2), \quad |\nabla u|^{2}(x_{0})\le |x_0|^{-2}\int_{B(x_{0},|x_0|)} |\nabla u|^{2}\le C|x_0|^{-1-2\delta}\int_{B(0,2)} |\nabla u|^{2}dx.
\end{align}
As we choose $\delta < 1/2$, then $-\frac{1}{2}-\delta > -1$, and $u$ has a finite limit at $0$ (see the proof of Lemma \ref{lm:H1HM}). We can therefore extend $u$ to $B(0,1)$ with $u\in \mc C^0(B(0,1))$.

Also, so there exists $q >2$ such that $(-\frac{1}{2}-\delta)q >-2$. Therefore, $\nabla u\in L^{q}(B(0,1/2))$. We are then in a position to apply Lemma \ref{lmsol0pointe}, which assures that $u$ is smooth.
\enp

\begin{cor}
\label{corscalinvers}
Let $u$ be a weak solution of \eqref{systconformembed} on $\R^{2}\setminus B(0,1)$ with finite energy. Then, there exists $u_{\infty}\in \mathcal{N}\subset \R^{M}$ so that $u(x)\underset{|x|\to +\infty}{\longrightarrow} u_{\infty}$. 

In particular, for any $\e>0$, there exists $R\ge 1$ so that for any $r \ge R$, denoting $u_{r}(x)=u\left(rx \right)$, $u_r$ satisfies \eqref{systconformembed} on $\R^{2}\setminus B(0,1)$ and
\begin{align*}
\nor{\nabla u_{r}}{L^{2}(\R^{2}\setminus B(0,1))}+\nor{u_{r}-u_{\infty}}{L^{\infty}(\R^{2}\setminus B(0,1))} \le \e.
\end{align*}
\end{cor}

\bnp
Define $\widetilde{u}(x)=u\left(\frac{x}{|x|^{2}}\right)$ which is also a solution with finite energy on $B(0,1)\setminus \{0\}$. Theorem \ref{thmsackuhlextensconf} implies that $\widetilde{u}$ can be extended to a smooth function on $B(0,1)$. In particular, denoting $u_{\infty} = \tilde u(0) \in \q N$,
$u(x) = \tilde u(x/|x|^2) \to u_\infty$  as $x \to +\infty$. 
 The second result is then direct once we check that $u_{r}$ is also a solution of \eqref{systconformembed} on $\{|x|\ge  1/r\}$ with 
\begin{align*}
\int_{|x|\ge 1}|\nabla u_{r}(x)|^{2}dx= \int_{|x|\ge 1} r^2 |\nabla u(rx)|^{2}dx=\int_{|y|\ge r}|\nabla u(y)|^{2}dy.
\end{align*}
So, it can be made arbitrary small. 
\enp

\begin{prop}[Regularity and trace]\label{propelipticregtraceHarmon}
Let $u_{\infty}\in \q N$. There exists $\e_{0}>0$ and $C>0$ such that the following holds. Let $u\in H^{1}(B(0,12))$ be solution of \eqref{formconf} taking values some chart around $u_{\infty}$, so that 
\[ \nor{\nabla u}{L^{2}(B(0,12))}+\nor{u-u_{\infty}}{L^{\infty}(B(0, 12))} =: \e \le \e_0. \]
Then, we have the estimate
\[ \nor{\nabla u}{L^{\infty}(B(0,3/2))}\le C\e. \] 
Moreover, for any $s<4$, there hold $u_{\left|\mathbb{S}^{1}\right.}\in H^{s}(\mathbb{S}^{1})$ and $ \nor{u_{\left|\mathbb{S}^{1}\right.}-u_{\infty}}{H^{s}(\mathbb{S}^{1})}\le C_{s}\e$.
 \end{prop}
 
\bnp

As $u$ is a solution of \eqref{formconf} with values in some fixed coordinate charts around $u_{\infty}$, by considering some embedding $\mathcal{N}\subset \R^M$, we can identify $u$ with a solution of \eqref{systconformembed} and we use this representation from now on; the result by Rivière \cite{R:07} ensures that $u\in \mc C^{\infty}(B(0,12),\mathcal{N})$.
Due to Theorem \ref{thm:W1inftyH},  
\begin{align*}
\sup_{B(0,3/2)}|\nabla u|^{2}\le C \int_{B(0,12)}|\nabla u(x)|^{2}dx\le C \e^{2}.
\end{align*} 
Using the equation \eqref{systconformembed}, it implies
\begin{align*}
\sup_{\{3/4\le r\le 5/4\}}|\Delta u|\le C\e^{2}.
\end{align*} 
Let $\chi$ be  a cutoff function supported in $\{3/4\le |x|\le 5/4\}$ and equal to $1$ on $\{7/8\le |x|\le 9/8\}$, and denote $\widetilde{u}=(u-u_{\infty})\chi$. For any $1<p<+\infty$, we have 
\begin{align*}
\MoveEqLeft \nor{\Delta \widetilde{u} }{L^{p}(\R^{2)}}\le C\nor{\Delta \widetilde{u} }{L^{\infty}(3/4\le r\le 5/4)}\\
&\le C\nor{\Delta u }{L^{\infty}( 3/4\le r\le 5/4)}+\nor{\nabla u}{L^{\infty}(3/4\le r\le 5/4)}+\nor{u-u_{\infty} }{L^{\infty}(3/4\le r\le 5/4)} \\
& \le C\e.
\end{align*} 
By Calder\'on-Zygmund estimates \cite[Theorem 9.11]{GTbook}, we infer
\begin{align}\nor{\widetilde{u}-u_{\infty} }{W^{2,p}(7/8\le |x|\le 9/8)}\le \nor{\widetilde{u}-u_{\infty} }{L^p(3/4\le r\le 5/4)}+\nor{\Delta \widetilde{u}}{L^p(3/4\le r\le 5/4)} \le C_{p}\e.
\end{align}
Since $u=\widetilde{u}+u_{\infty}$ on $\{7/8\le |x|\le 9/8\})$, we can iterate by checking the equation satisfied by $\nabla u$ and then $\nabla^2 u$ to get 
\begin{align}\nor{u-u_{\infty} }{W^{4,p}(31/32\le |x|\le 33/32)}\le  C_{p}\e.
\end{align}


For any $7/2<s<4$, choosing $p$ so that $2 < p <\frac{1}{4-s}$, then $s < 4 - 1/p$ and by trace estimates:
\begin{align*}
\nor{u_{\left|\mathbb{S}^{1}\right.}-u_{\infty}}{H^{s}(\mathbb{S}^{1})} &\le C_{s}\nor{u_{\left|\mathbb{S}^{1}\right.}-u_{\infty}}{W^{4-1/p,p}(\mathbb{S}^{1})} \\
& \le C_s \nor{u-u_{\infty} }{W^{4,p}(31/32\le |x|\le 33/32)}\le  C_{s} \e. \qedhere
\end{align*}

\enp

\begin{prop}
\label{propUCPDiskHarmon}
There exists $\e>0$ (depending on $\mathcal{N}$ and $\omega$) such that the following holds. Let $u$ and $v$ be two smooth solutions of \eqref{formconf} in $B(0,1)$ so that $u=v$ on the unit sphere $\m S^{1}$ and are small in the sense that  
\begin{align*}
\nor{\nabla u}{L^{\infty}(B(0,1))}+\nor{\nabla v}{L^{\infty}(B(0,1))}\le \e.
\end{align*}
Then, $u=v$ in $B(0,1)$.  
\end{prop}

\bnp
Let $w=u-v$, then $w$ satisfies
\begin{align}
\label{HarmonicR}
\Delta w & = -[\Gamma(u)-\Gamma(v)] (\nabla u,\nabla u)-[\Gamma(v) (\nabla w,\nabla u+\nabla v)\\
& \qquad +[H(u)-H(v)] (\partial_x u,\partial_y  u)+H(v) (\partial_x  w, \partial_y u)+H(v) (\partial_x  v,\partial_y  w).
\end{align}
where we have used the bilinearity of $\Gamma$ and $H$ and the symmetry of $\Gamma$. 
Taking scalar product in $\m R^N$ of \eqref{HarmonicR} with $w$, integrating and then performing an integration by parts, we get, using the Dirichlet boundary condition for $w$ and that $\| \Gamma \|_{\mc C^1(\q N)}, \| H \|_{\mc C^1(\q N)} \le C_{\mathcal{N},\omega}$ are bounded (since the target manifold $\mathcal{N}$ is compact)
\begin{align}
\MoveEqLeft \int_{|x|\le 1}|\nabla w|^{2}dx =\int_{|x|\le 1}[\Gamma(u)-\Gamma(v)] (\nabla u,\nabla u)\cdot w~dx \\
& \quad +\int_{|x|\le 1}\Gamma(v) (\nabla w,\nabla u+\nabla v)\cdot w~dx - \int_{|x|\le 1}[H(u)-H(v)] ({\partial_x} u,\partial_y  u)\cdot w~dx \\
& \quad - \int_{|x|\le 1}H(v) ({\partial_x}  w, {\partial_y}  u)\cdot w -\int_{|x|\le 1}H(v) ({\partial_x} v,{\partial_y}  w)\cdot w~dx\\
\label{est:Harmonic_diff} &\le C\int_{|x|\le 1}|\nabla u|^{2} |w|^{2}~dx+C\int_{|x|\le 1}|\nabla w| \left(|\nabla u|+|\nabla v|\right) |w|~dx\\
&\le C\e^{2}\int_{|x|\le 1} |w|^{2}~dx+C\e \int_{|x|\le 1}|\nabla w|^{2}+ |w|^{2}~dx \\
& \le  \tilde{C} \e\int_{|x|\le 1}|\nabla w|^{2}dx.
\end{align}
(we used the Poincar\'e inequality on the bounded set $B(0,1)$, which holds due to the Dirichlet boundary for $w$). For  $ \tilde{C} \e<1$, this gives $w=0$.
\enp



\subsection{Some results about Harmonic Maps in dimension \texorpdfstring{$d\ge 3$}{d>=3}}
\label{s:app:HM3}
In this section, we gather some already known facts about Harmonic Maps in $\R^{d}$ (for the initial manifold). We refer to \cite{HeleinBook,HeleinSurvey, Schoen:84} for books or survey on this very studied topic. 

One important tool will be the following $\e$-regularity result of Schoen-Uhlenbeck \cite[Theorem 2.2]{Schoen:84}.
\begin{thm}[Theorem 2.2 of \cite{Schoen:84}]\label{thmSchoen} Suppose $u\in \mc C^{2}(B(0,r),\mathcal{N})$ is a harmonic map. 
There exists $\e>0$ and $C>0$, depending only on $d$, $\mathcal{N}$ such that if
\begin{align*}
r^{2-d}\int_{B(0,r)}|\nabla u|^2\le \e
\end{align*}
then $u$ satisfies the inequality 
\begin{align*}
\sup_{B(0,r/2)}|\nabla u|^2\le Cr^{-d}\int_{B(0,r)}|\nabla u|^2.
\end{align*}
\end{thm}

We begin by a result about the decay of regular solution in the energy space. The proof relies on the regularity result of \cite{SU:83} for small solutions and a scaling argument. Some similar results also appear in \cite{ABLV22}.

\bnp[Proof of Lemma \ref{lmregHMHd}]
 For $x_{0}\in \R^{d}$ and $R=|x_{0}|\ge 2$, we denote $\tilde{u}(x)=u(x_{0}+Rx/2)$ which is still solution of the Harmonic map on $B(0,1)$ . We compute 
\begin{align*} 
\int_{B(0,1)}|\nabla \tilde{u}(x)|^{2}dx & = (R/2)^{2}\int_{B(0,1)}|\nabla u(x_{0}+Rx/2)|^{2}dx \\
& = (R/2)^{2-d}\int_{B(x_{0},R/2)}|\nabla u(y)|^{2}dy.
 \end{align*}
Since $d\ge 3$, this becomes small for large $R$, so that we can apply Theorem \ref{thmSchoen} to $\tilde{u}$ and $r=1$ to get
\begin{align*}
|\nabla \tilde{u}(0)|^{2} & \le {C}\int_{B(0,1)}|\nabla \tilde{u}(x)|^{2}dx \le C R^{2-d}\int_{B(x_{0},R/2)}|\nabla u(y)|^{2}dy \\
& \le C R^{2-d} \int_{B(0,3R/2)\setminus B(0,R/2)}|\nabla u(y)|^{2}dy.
\end{align*}
Since $|\nabla \tilde{u}(0)|=R|\nabla u(x_{0})|$, We get
\begin{align*}
|\nabla u(x_{0})|^{2}\le C R^{-d}\int_{\R^{d}\setminus B(0,R/2)}|\nabla u(y)|^{2}dy,
\end{align*}
which is the first item once $R_{0}$ is chosen large enough.

Concerning the second point, the $L^{2}$ part is immediate when $R_{0}$ is taken large enough, the $L^{\infty}$ part is a consequence of the first item, while the $L^{d}$ part is obtained by interpolation.
\enp

\bnp[Proof of Lemma \ref{lm:H1HM}]
We first fix a direction $e_{1}=(1,0,\cdots, 0)\in \R^{d}$ and prove that $u(ne_{1})$ is convergent in $\q N \subset \R^{M}$. Lemma \ref{lmregHMHd} and the fundamental Theorem of calculus give for $n$ large enough $\left|u((n+1)e_{1})-u(ne_{1})\right|= \frac{o(1)}{n^{d/2}}$. In particular, since $d/2>1$, the series is convergent and $u(ne_{1})$ is convergent to some $u_{\infty}\in \mathcal{N}\subset \R^{M}$. Now, for any $n\in \N$ and $x\in \R^{d}$ with $|x|\in [n,n+1]$, there exists a path $\gamma\subset \R^{d}\setminus B(0,n)$ piecewise affine and of length $|\gamma|\le C_{d}n$ so that $\gamma(0)=ne_{1}$ and $\gamma(1)=x$. The fundamental Theorem of Calculus gives $\left|u(x)-u(ne_{1})\right|\lesssim n \sup_{s\in [0,1]}|\nabla u(\gamma(s))|= o(n^{1-d/2})$ after having used again Lemma \ref{lmregHMHd}. Since $1-d/2<0$, we get the expected convergence. Note also that the proof gives more precisely
\begin{align}
\label{e:cvgHMinfty}
\nor{u-u_{\infty}}{L^{\infty}(B(0,n+1)\setminus B(0,n))}= o(n^{1-d/2}) + o(1)\sum_{k\ge n}k^{-d/2} = o(n^{1-d/2}).
\end{align}

Now, we prove $u-u_{\infty}\in \dot{H}^{1}(\R^{d}\setminus B(0,1))$, that is $u-u_{\infty}$ can be approximated for the energy norm by a sequence of functions in $\mc C^{\infty}_{c}(\R^{d})$. 

First consider $u_{n}=\chi\left(\frac{x}{n}\right)(u-u_{\infty})$ where $\chi\in \mc C^{\infty}_{c}(B(0,2);[0,1])$ equals to $1$ on $B(0,1)$. By assumption, $u_n \in \mc C^2_c(\m R^d)$ and for
\begin{align*}
v_{n}& :=u-u_{\infty}-u_{n}= (u-u_{\infty})\left(1- \chi\left(\frac{x}{n}\right)\right),\\
\nabla v_{n}& =\nabla u\left(1- \chi\left(\frac{x}{n}\right)\right)-\frac{1}{n} (u-u_{\infty})(\nabla\chi)\left(\frac{x}{n}\right),
\end{align*}
we can bound, using estimate \eqref{e:cvgHMinfty},
\begin{align*}
\int_{\R^{d}} |\nabla v_{n}|^{2} dx&\le \int_{\R^{d}\setminus B(0,n)} |\nabla u|^{2 }dx+\nor{u-u_{\infty}}{L^{\infty}(B(0,2n)\setminus B(0,n))}^{2}\frac{1}{n^{2}} \int_{\R^{d}}\left|\nabla\chi\left(\frac{x}{n}\right)\right|^{2}\\
&\le \int_{\R^{d}\setminus B(0,n)} |\nabla u|^{2 }dx+o(n^{2-d})n^{d-2} \to 0 \quad \text{as } n \to +\infty.
\end{align*}
Once $u-u_{\infty}$ has been approximated by some $\mc C^{2}$ compactly supported functions, it is easy to approximate it by some smooth compactly supported functions by standard approximation process.
\enp
We will use the following result of uniqueness. Some uniqueness result appear in bounded domain in \cite{Struwe:98} for small data in some more refined norms.

\bnp[Proof of Proposition \ref{propUCPRdHarmon}]
The beginning of the proof is a similar computation as in Proposition \ref{propUCPDiskHarmon} except that since we are on an unbounded set, we need to interpret the integration by parts as the weak formulation of Definition \ref{defsoluextHM}. Also, we consider the equation for the embedded formulation. Let $w=u-v\in \dot{H}^{1}(\R^{d}\setminus B(0,1))$ with $w=0$ on $\m S^{d-1}$, that is $w\in \dot{H}^{1}_0(\R^{d}\setminus B(0,1))$ 
In particular, there exists a sequence of $w_{n} \in \mc C^{\infty}_{c}(\R^{d}\setminus B(0,1))$ so that $\nor{\nabla (w-w_{n})}{L^{2}}\to 0$. 
Using that $w$ solves \eqref{HarmonicR}, we arrive as in \eqref{est:Harmonic_diff} to
\begin{align*}
\int_{\{|x|\ge 1\}}\nabla w\cdot \nabla w_{n}dx
& \le C\int_{\{|x|\ge 1\}}|\nabla u|^{2} |w||w_{n}|~dx \\
& \qquad +C\int_{\{|x|\ge 1\}}|\nabla w| \left(|\nabla u|+|\nabla v|\right) |w_{n}|~dx.
\end{align*}
Using H\"older inequality for the exponents $\frac{2}{d}+\frac{1}{2^*}+\frac{1}{2^*}=1$ and $\frac{1}{2}+\frac{1}{d}+\frac{1}{2^*}=1$,  together with the Sobolev embedding \eqref{Soblocal}, we obtain for $n$ large enough
\begin{align*}
\MoveEqLeft \int_{\{|x|\ge 1\}}\nabla w\cdot \nabla w_{n}dx\\
&\le C\nor{\nabla u}{L^{d}(\{|x|\ge 1\})}^{2}\nor{w}{L^{2^*}(\{|x|\ge 1\})}\nor{w_{n}}{L^{2^*}(\{|x|\ge 1\})} \\
& +C\nor{\nabla w}{L^{2}(\{|x|\ge 1\})}\left(\nor{\nabla u}{L^{d}(\{|x|\ge 1\})}+\nor{\nabla v}{L^{d}(\{|x|\ge 1\})}\right)\nor{w_{n}}{L^{2^*}(\{|x|\ge 1\})}\\
& \le \tilde{C}(\e^{2}+\e)\nor{\nabla w}{L^{2}(\{|x|\ge 1\})}\nor{\nabla w_n}{L^{2}(\{|x|\ge 1\})}.
\end{align*}
Taking the limit $n \to +\infty$, this gives 
\[ \nor{\nabla w}{L^{2}(\{|x|\ge 1\})}^2 \le \tilde{C}(\e^{2}+\e)\nor{\nabla w}{L^{2}(\{|x|\ge 1\})}^2. \]
For $\e$ so small that $\tilde{C}(\e^{2}+\e)<1/2$, we get $\nabla w=0$ and therefore $w=0$ as expected.
\enp
\bibliographystyle{alpha} 
\bibliography{biblio}

\bigskip
\bigskip

\textsc{Rapha\"el C\^ote}\\
Universit\'e de Strasbourg \& University of Strasbourg Institute for Advanced Study\\
CNRS, IRMA UMR 7501,\\
7 rue Ren\'e Descartes,\\
67084 Strasbourg, France\\
{\tt raphael.cote@unistra.fr}\\

\bigskip

\textsc{Camille Laurent}\\
CNRS UMR 7598 \& Sorbonne Universit\'es UPMC Univ Paris 06,\\
Laboratoire Jacques-Louis Lions,\\
F-75005, Paris, France\\
{\tt camille.laurent@sorbonne-universite.fr}\\
\vspace*{0.5cm}
\end{document}